\RequirePackage{ifpdf}
\ifpdf 
\documentclass[pdftex]{sigma}
\else
\documentclass{sigma}
\fi

\usepackage{multirow,bm,amscd,comment}
\usepackage{array}
\usepackage{enumitem}
\usepackage{tikz-cd}

\usepackage{mathtools}

\newcommand{\coev}{\stackrel{\longrightarrow}{\operatorname{coev}}}
\newcommand{\ev}{\stackrel{\longrightarrow}{\operatorname{ev}}}
\newcommand{\tev}{\stackrel{\longleftarrow}{\operatorname{ev}}}
\newcommand{\tcoev}{\stackrel{\longleftarrow}{\operatorname{coev}}}

\def \cat{\mathcal{C}}
\def \Z{\mathbb{Z}}
\def \Q{\mathbb{Q}}
\def \R{\mathbb{R}}
\def \C{\mathbb{C}}
\def \vert{\mathrm{Vert}}
\def \Spin{\mathrm{Spin}}
\def \so{{{\rm SO}(3)}}
\def \CN{\mathcal{N}}
\def \CT{\mathcal{T}}
\def \CL{\mathcal{L}}
\def \CC{\mathcal{C}}
\def \N{\mathrm{N}}
\def \WRT{\operatorname{WRT}}
\def \lk{{\ell k}}
\def \Tr{\operatorname{Tr}}
\def \Tor{\operatorname{Tor}}
\def \Hom{\operatorname{Hom}}
\def \Id{\mathrm{Id}}
\def \qdim{\operatorname{qdim}}
\def \End{\operatorname{End}}

\def \Go{\mathfrak{G}}
\def \Lo{\mathfrak{L}}

\def \VV{\mathbb{V}}

\def\CB{{\mathcal B}}

\def\CH{{\mathcal H}}
\def\CU{{\mathcal U}}

\renewcommand{\bar}{\overline}
\renewcommand{\hat}{\widehat}

\numberwithin{equation}{section}

\newtheorem{Theorem}{Theorem}[section]

\newtheorem{Lemma}[Theorem]{Lemma}
\newtheorem{Proposition}[Theorem]{Proposition}
\newtheorem{conj}[Theorem]{Conjecture}

 { \theoremstyle{definition}
\newtheorem{Definition}[Theorem]{Definition}

\newtheorem{Example}[Theorem]{Example}
\newtheorem{Remark}[Theorem]{Remark} }

\begin{document}
\allowdisplaybreaks

\newcommand{\arXivNumber}{2107.14238}

\renewcommand{\thefootnote}{}

\renewcommand{\PaperNumber}{010}

\FirstPageHeading

\ShortArticleName{Non-Semisimple TQFT's and BPS $q$-Series}

\ArticleName{Non-Semisimple TQFT's and BPS $\boldsymbol{q}$-Series\footnote{This paper is a~contribution to the Special Issue on Enumerative and Gauge-Theoretic Invariants in honor of Lothar G\"ottsche on the occasion of his 60th birthday. The~full collection is available at \href{https://www.emis.de/journals/SIGMA/Gottsche.html}{https://www.emis.de/journals/SIGMA/Gottsche.html}}}

\Author{Francesco COSTANTINO~$^{\rm a}$, Sergei GUKOV~$^{\rm b}$ and Pavel PUTROV~$^{\rm c}$}

\AuthorNameForHeading{F.~Costantino, S.~Gukov and P.~Putrov}

\Address{$^{\rm a)}$~Institut de Math\'ematiques de Toulouse, 118 route de Narbonne, F-31062 Toulouse, France}
\EmailD{\href{mailto:francesco.costantino@math.univ-toulouse.fr}{francesco.costantino@math.univ-toulouse.fr}}

\Address{$^{\rm b)}$~Walter Burke Institute for Theoretical Physics, California Institute of Technology,\\
\hphantom{$^{\rm b)}$}~Pasadena, CA 91125, USA}
\EmailD{\href{mailto:gukov@theory.caltech.edu}{gukov@theory.caltech.edu}}

\Address{$^{\rm c)}$~The Abdus Salam International Centre for Theoretical Physics, Strada Costiera 11,\\
\hphantom{$^{\rm c)}$}~Trieste 34151, Italy}
\EmailD{\href{mailto:putrov@ictp.it}{putrov@ictp.it}}

\ArticleDates{Received January 21, 2022, in final form February 10, 2023; Published online March 15, 2023}

\Abstract{We propose and in some cases prove a precise relation between 3-manifold invariants associated with quantum groups at roots of unity and at generic $q$. Both types of invariants are labeled by extra data which plays an important role in the proposed relation. Bridging the two sides~-- which until recently were developed independently, using very different methods~-- opens many new avenues. In one direction, it allows to study (and perhaps even to formulate) $q$-series invariants labeled by spin$^c$ structures in terms of non-semisimple invariants. In the opposite direction, it offers new insights and perspectives on various elements of non-semisimple TQFT's, bringing the latter into one unifying framework with other invariants of knots and 3-manifolds that recently found realization in quantum field theory and in string theory.}

\Keywords{3-manifold invariants; knot invariants; TQFT}

\Classification{57K16; 81T45}

{\small \tableofcontents}

\renewcommand{\thefootnote}{\arabic{footnote}}
\setcounter{footnote}{0}

\section{Introduction and summary}

As part of a larger quest for new quantum invariants of 3-manifolds, we wish to establish a~precise relation between the CGP invariants $\N_r (M, \omega)$ and the GPPV invariants \smash{$\hat Z_\mathfrak{s} (M;q)$}, introduced in~\cite{costantino2014quantum} and~\cite{Gukov:2017kmk, Gukov:2016gkn}, respectively. While the CGP invariants are defined for almost all closed connected 3-manifolds (suitably decorated, see later), the GPPV invariants are currently defined for a much smaller class, e.g., surgeries on closures of homogeneous braids with framings satisfying certain inequalities (see Sections~\ref{sec:GPPV} and~\ref{sec:Zhat-def-b1} for a detailed discussion and various alternative mathematical definitions). More explicitly there exist formulas for the GPPV invariants for some (but not all) $3$-manifolds but their invariance is proved in a smaller set of cases (e.g., surgeries over negative plumbing links).

Here we summarize some of the essential features of these invariants, compare side-by-side what they depend on, and present a motivation for why one should expect a connection between these two rather different sets of invariants. First, perhaps the most obvious part of input data that both sets of invariants need is a choice of 3-manifold $M$. In either case, however, it needs to be equipped with additional structure; in the case of CGP invariants $\N_r (M, \omega)$ it involves a~choice of $\omega \in H^1(M;\C/2\Z) \setminus H^1(M;\Z/2\Z)$, whereas in the case of GPPV invariants \smash{$\hat Z_\mathfrak{s} (M;q)$} it depends on $\mathfrak{s} \in \operatorname{Spin}^c (M) / \Z_2$. Although these structures are clearly different, they both are related\footnote{Recall that non-canonically $\operatorname{Spin}^c (M) \cong H_1 (M;\Z)$.} to $H_1 (M)$ and should be regarded as mutual counterparts in identifying the two sets of invariants, as we will see below.

Similarly, the dependence of $\N_r (M, \omega)$ on a positive integer $r\neq 0\bmod 4$ should be compared to the $q$-dependence of \smash{$\hat Z_\mathfrak{s} (M;q)$}.
Indeed, both sets of invariants are quantum group invariants of 3-manifolds, with
$\xi = {\rm e}^{\frac{\pi {\rm i}}{r}}$ and $q$ respectively playing the role of the quantum parameters. The definition of $\N_r (M, \omega)$ is based~\cite{costantino2014quantum} on the representation theory of the unrolled quantum group $ \CU_{\xi}^H (\mathfrak{sl}_2)$, whereas \smash{$\hat Z_\mathfrak{s} (M;q)$} should be thought of as a quantum group invariant associated with $\CU_q (\mathfrak{sl}_2)$ at generic $|q|<1$. Indeed, \smash{$\hat Z_\mathfrak{s} (M;q)$} basically gives a non-perturbative definition\footnote{At the perturbative level, in general it contains contributions of {\it all} complex ${\rm SL}(2,\C)$ flat connections on~$M$, much like, e.g., the Teichm\"uller TQFT~\cite{MR3204520} contains contributions of a particular component of the space of ${\rm SL}(2,\mathbb{R})$ flat connections. Note, as a result, the Teichm\"uller TQFT is strictly speaking not a non-perturbative completion of any perturbative Chern--Simons theory, and is much closer to the ``real Chern--Simons theory'' than to the ``complex Chern--Simons theory.''
} of ``${\rm SL}(2,\C)$ Chern--Simons theory'' that behaves well under surgery and has line operators labeled by Verma modules of arbitrary complex weight~\cite{Gukov:2003na,Park:2020edg}.
\begin{equation}\renewcommand{\arraystretch}{1.5}
\begin{tabular}{l|c|c}
\hline
 & $\N_r (M, \omega)$ & \smash{$\hat Z_\mathfrak{s} (M;q)$}
\\
\hline
Quantum parameter &$\xi = {\rm e}^{\frac{\pi {\rm i}}{r}}$ root of unity & generic $|q|<1$
\\
Quantum group & $\CU^H_\xi(\mathfrak{sl}_2)$ & $\CU_q (\mathfrak{sl}_2)$
\\
Additional structure & $\omega \in H^1(M;\C/2\Z) \setminus H^1(M;\Z/2\Z)$ & $\mathfrak{s} \in \operatorname{Spin}^c (M) / \Z_2$
\\
\hline
\end{tabular}
\label{NZhatcomparison}
\end{equation}

Since the relation between $\CU_q (\mathfrak{sl}_2)$ and \smash[t]{$\CU^H_\xi(\mathfrak{sl}_2)$} involves, among other things, specializing $q$ to be a root of unity, one might expect $\N_r (M, \omega)$ to be related to limiting values of \smash{\smash{$\hat Z_\mathfrak{s} (M;q)$}} at roots of unity. This relation between quantum groups was a large part of the motivation in~\cite{Gukov:2020lqm}, where a relation of this form was proposed for knot complements.
Indeed, the invariants~$\N_r$ are defined in the more general setup of triples $(M,\omega,L)$, where $L\subset M$ is a $\C$-colored framed oriented link and $\omega\in H^1(M\setminus L;\C/2 \mathbb{Z})$ is a cohomology class whose period on the meridians of $L$ is congruent (mod $2 \mathbb{Z}$) to the color of $L$. In this paper, we will only consider the case of links in $S^3$; in this case the cohomology class $\omega$ is uniquely determined by the coloring of the link therefore, instead of $\N_r\big(S^3,\omega,L\big)$, we shall use the abusive notation $\N_r(L_\alpha)$, with $\alpha$ the coloring of $L$.
Coming back to the conjectural relations between $\hat{Z}$ and $N_r$, we recall that in~\cite{Gukov:2019mnk} it was proposed that for the complement of a knot $K$ in $S^3$ the correct version of $\hat{Z}$ can be summarised by a single $2$ variable series $F_K (x,q) := \hat Z \big(S^3 \setminus K\big)$ (not depending on any further structure on $K$). This series is conjecturally related to the invariants $\operatorname{ADO}_r (x;K)$ introduced by Akutsu--Deguchi--Ohtsuki~\cite{MR1164114} as follows:

\begin{conj}\label{conj:park}\quad
\begin{enumerate}\itemsep=0pt
\item[$(a)$] There exists $W_K(x,q)\in \Z\big[q^{-1},q\big]\big[x^{\pm \frac{1}{2}}\big]$ such that the following holds for every $r\in \N$:
\begin{equation}
 W_K (x,q) |_{q \rightarrow \xi^2} =
\frac{\mathrm{ADO}_r \big(x/\xi^2;K\big)}{\Delta_K (x^{r})} \cdot \big( x^{1/2} - x^{-1/2} \big)
\label{FKADOnorm}
\end{equation}
$($Recall that for a knot $\Delta_K\big(x^2\big)=\nabla_K(x)\big(x-x^{-1}\big)$ where $\Delta_K$ is the Alexander polynomial and $\nabla_K$ the Alexander--Conway function.$)$

\item[$(b)$] When the series $F_K(x,q)$ is defined, it provides such\footnote{The equality~\eqref{FKADOnorm} does not by itself uniquely fix $W_K(x,q)$, as one in particular can always add to $W(x,q)$ a~multiple of $\prod_{n\geq 1}(1-q^n)$.} a $W_K(x,q)$.
 \end{enumerate}
\end{conj}

Here we address the problem of establishing a similar relation for more general 3-manifolds. This requires 3-manifold invariants which specialize to ADO invariants for knot complements, and CGP invariants $\N_r$ perfectly fit the bill:
\begin{equation}
\operatorname{ADO}_r \big(x^2/\xi^2;K\big) =
 \frac{x^r - x^{-r}}{x - x^{-1}} \N_r (K_{\alpha}), \qquad \text{where}\quad
 x = {\rm e}^{\frac{\pi {\rm i}\alpha}{r}},\quad \xi={\rm e}^{\frac{\pi {\rm i} }{r}}.
\label{ADOvsNK}
\end{equation}

The composition of~\eqref{FKADOnorm} and~\eqref{ADOvsNK} not only gives us the first instance of the sought-after relation between CGP and GPPV invariants in a certain class of 3-manifolds, but also provides a clue for the relation between parameters $r$ and $q$ for more general $M$:
\begin{equation*}
\N_r (M, \omega)
\, \stackrel{?}{\longleftrightarrow} \,
 \hat Z_\mathfrak{s} (M;q) \big|_{q \to \xi^2 = {\rm e}^{2\pi {\rm i}/r}}.
\end{equation*}
Another useful clue that follows from~\eqref{FKADOnorm} and~\eqref{ADOvsNK} is that the relation between $\N_r (M, \omega)$ and $\hat{Z}_\mathfrak{s} (M;q)$ should be linear,
\begin{equation}
\N_r (M,\omega)
=\bigg(\sum_{\mathfrak{s}} c^{\mathrm{CGP}}_{\omega,\mathfrak{s}} \hat Z_\mathfrak{s} (M;q)\bigg)\bigg|_{q \to {\rm e}^{\frac{2 \pi {\rm i}}{r}}}
\label{CGPZhatviac}
\end{equation}
much like analogous relations between $\hat Z (M;q)$ and other invariants of $M$, such as the inverse Turaev torsion~\cite{turaev1990euler,turaev1997torsion}, Witten--Reshetikhin--Turaev (WRT) invariants~\cite{reshetikhin1991invariants,Witten:1988hf}, and Rokhlin invariants.
Remark that in the above formula we are summing over all spin$^c$ structures although the $\hat{Z}_\mathfrak{s}(M;q)$ is invariant under the $\Z/2\Z$ involution on $\mathrm{Spin}^c(M)$, so although some terms could be grouped, this convention is often more convenient to work with.
\begin{figure}[ht]
	\centering
	\includegraphics[width=2.9in]{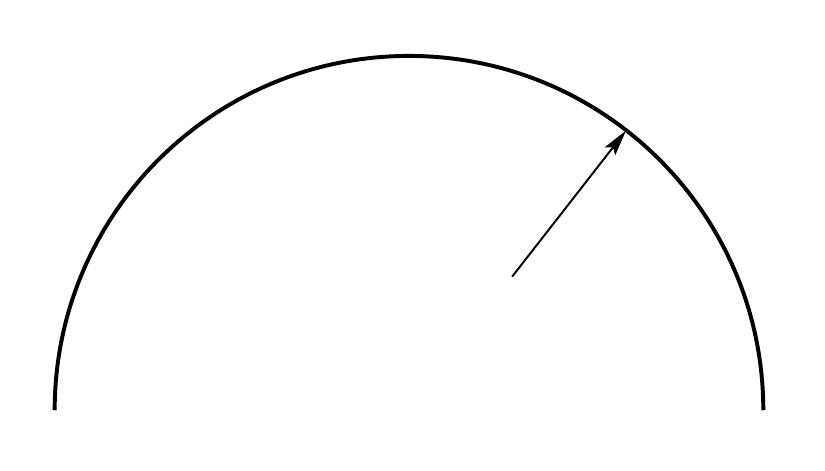}
	\caption{One of our main goals is to explore possible relations between $\N_r (M,\omega)$ and the limiting behavior of $\hat{Z}_{\mathfrak{s}} (M;q)$ at roots of unity.}
\end{figure}

Establishing the relation~\eqref{CGPZhatviac} -- and, in particular, determining the precise form of the coefficients $c^{\mathrm{CGP}}_{\omega,\mathfrak{s}}$ that relate additional structures on $M$ that enter the two sets of invariants -- is one of our main goals. We find\pagebreak

\begin{conj}\label{conj:cgpz}
Let $M$ be a rational homology sphere.
\begin{enumerate}\itemsep=0pt
\item[$(a)$] There exist a collection of series $W_\mathfrak{s}(M)\in q^{\Delta_\mathfrak{s}}\Q[[q]]\big[q^{-1}\big]$, $\mathfrak{s}\in \mathrm{Spin}^c(M)$ $($for some $\Delta_\mathfrak{s}\in \Q)$ such that the relation holds:%
\footnote{Note that $W_\mathfrak{s}(M)$, due to the overall factor $q^{\Delta_\mathfrak{s}}, \Delta_\mathfrak{s}\in \Q$, are multivalued functions of $q$ (the values on different branches differ by an overall phase). Throughout the paper, when writing the limit in the form ``$q\rightarrow {\rm e}^{{\rm i}\phi}$'' with $\phi\in [0,2\pi)$, it is assumed that the branch of $q^\delta$ for some $\delta\in
\Q$ is such that $q^{\delta} \rightarrow {\rm e}^{{\rm i}\phi \delta}$. Equivalently, one can consider the corresponding single-valued functions of $\tau$, with $q={\rm e}^{2\pi {\rm i}\tau}$. Then ``$q\rightarrow {\rm e}^{{\rm i}\phi}$'' should be understood as ``$\tau\rightarrow \phi/(2\pi)$'' (taken from the upper half plane).}
\begin{equation}
\N_r (M,\omega) = \bigg(\sum_{\mathfrak{s}\in \mathrm{Spin}^c(M)} c^{\mathrm{CGP}}_{\omega,\mathfrak{s}} W_\mathfrak{s} (M)\bigg) \bigg|_{q \to {\rm e}^{\frac{2 \pi {\rm i}}{r}}}
\label{CGPZhatviac3}
\end{equation}
with
\begin{equation}
\!c^{\mathrm{CGP}}_{\omega,\sigma(b,s)} \!=\! \frac{\mathcal{T}(M,[\omega])}{{|H_1(M;\Z)|}}
\!\begin{cases}
-{\rm e}^{-\frac{\pi {\rm i}}{2}\mu(M,s)}\!\sum_{a,f} {\rm e}^{2\pi {\rm i}\big(
 -\frac{r-1}{4}\lk(a,a) +\lk(a,f-b)-\frac{1}{2}\omega(a) +\lk(f,f)\big)}
\\[-1mm]
\hspace{60mm}\text{if}\quad r=1\bmod 4,
 \\[1mm]
\! \sqrt{|H_1(M;\Z)|} \sum_{a}
 {\rm e}^{-\frac{\pi {\rm i}r}{2} q_s(a)-2\pi {\rm i} \lk(a,b)-\pi {\rm i}\omega(a)},
\\[-1mm]
\hspace{60mm}\text{if}\quad r=2\bmod 4,
 \\[1mm]
\!-{\rm e}^{\frac{\pi {\rm i}}{2}\mu(M,s)}\sum_{a,f} {\rm e}^{2\pi {\rm i} \big(-\frac{r+1}{4}\lk(a,a)
 -\lk(a,f+b)-\frac{1}{2}\omega(a) -\lk(f,f)\big)}
 \\[-1mm]
\hspace{60mm} \text{if}\quad r=3\bmod 4,
\end{cases}
\label{coeffsummary}
\end{equation}
where
\begin{itemize}\itemsep=0pt
\item $\mathcal{T}(M,[\omega])$ is a suitable version of the Reidemeister torsion $($see Appendix~\ref{app:torsion}$)$;
\item $\lk(\cdot, \cdot)$ is the linking form on $H_1(M;\Z)$, and in the sum $a,f \in H_1(M;\Z)$ $($for more details about it, see Section~$\ref{sec:combinatorial})$;
\item $\sigma$ is the canonical map
\begin{equation*}
 \sigma\colon \ H_1(M;\Z)\times\operatorname{Spin}(M) \longrightarrow \operatorname{Spin}^c(M)
\end{equation*}
producing a spin$^c$ structure $\mathfrak{s}=\sigma(b,s)$ on $M$ from a spin structure\footnote{We use the standard typeface to denote spin structures (e.g., ``$s$'') and Fraktur typeface to denote spin$^c$ structures (e.g., ``$\mathfrak{s}$'').} $s$, and $b\in H_1(M;\Z)$ $($for more details about it, see Section~$\ref{sec:combinatorial})$.
\item $\mu(M,s)$ is the Rokhlin invariant of $M$ for spin structure $s$.\footnote{Note that conjecturally $\frac{1}{4}\mu(M,s)+\frac{1}{2}-\lk(b,b)= \Delta_{\sigma(b,s)}\bmod 1$, where $\Delta_{\mathfrak{s}}$ is the overall rational shift in the powers of $q$ in $\hat{Z}_{\mathfrak{s}}$ for a given $\mathfrak{s}\in \mathrm{Spin}^c(M)$~\cite{Gukov:2020frk}.}
\end{itemize}
\item[$(b)$] Furthermore, when the invariants $\hat{Z}_\mathfrak{s}(M;q)$ are defined one can take\footnote{As in the case of knots, the equality~\eqref{CGPZhatviac3} does not by itself uniquely fix $W_\mathfrak{s}(q)$, as one in particular can always add to them multiples of $\prod_{n\geq 1}(1-q^n)$.} $W_\mathfrak{s}(M)=\hat Z_\mathfrak{s}(M;q)$.
\end{enumerate}
\end{conj}

Note, the invariants $\N_r (M,\omega)$ are not defined for $r=0$ mod 4, which is why this case is omitted in~\eqref{coeffsummary}. (For more details about this case, see Section~\ref{sec:spincase}.)

The conjecture above is an analogue of the of the conjecture relating WRT invariant with the limits of $\hat{Z}_\mathfrak{s}$ that was formulated in~\cite{Gukov:2017kmk,Gukov:2016gkn}. It was proven for certain families of 3-manifolds in~\cite{MR4400935,fuji2021witten,Gukov:2016njj,mori2022witten}. Some elements of the latter conjecture can trace its origin to the work of Lawrence and Zagier~\cite{lawrence1999modular}. It is closely related to resurgence in analytically continued Chern--Simons theory~\cite{MR4400935,Gukov:2016njj} and quantum modularity~\cite{zagier2010quantum}.

We prove Conjecture~\ref{conj:cgpz} under some technical hypotheses specified in Theorem~\ref{thm:CGP-Zhat}, which can be stated in a simpler form as follows:

\begin{Theorem}[simplified form of Theorem~\ref{thm:CGP-Zhat-torsion}]
Let $M$ be a rational homology sphere obtained by integral surgery on a framed link $L\subset S^3$ for which Conjecture~$\ref{conj:park}$ $($appropriately extended from knots to links$)$ holds true. Then, under the technical hypotheses specified in Theorem~$\ref{thm:CGP-Zhat-torsion}$, the parts $(a)$ and $(b)$ of Conjecture~$\ref{conj:cgpz}$ are true for any pair $(M,\omega)$ with $\omega\in H^1(M;\C/2\Z)\setminus H^1(M;\Z/2\Z)$.
\end{Theorem}

In the statement of the theorem though, we point out that the status of the invariant $\hat{Z}$ is the following: assuming that parts $(a)$ and $(b)$ of Conjecture~\ref{conj:park} hold for $L$ (so that in particular a series $F_L$ is defined), there is an explicit formula providing $\hat{Z}$ for $M$ but the full proof of the convergence of the formula and its invariance is not yet available.

An infinite family of cases to which Theorem~\ref{thm:CGP-Zhat-torsion} applies is provided in Example~\ref{ex:theoremworks}.
The following is a list of cases for which the numerical evidence supports this result:
\begin{enumerate}\itemsep=0pt
\item $L$ is a plumbing link (for an infinite list of cases of these plumbing links the hypotheses of Theorem~\ref{thm:CGP-Zhat-torsion} can all be verified: see Example~\ref{ex:theoremworks});
\item $L$ is a trefoil knot.
\end{enumerate}
We also provide another kind of numerical evidence for the case of surgeries on the figure eight knot by cross checking the above conjectures with similar conjectures relating $\hat{Z}$ and the WRT invariants (see Section~\ref{sec:twistknots}).

 In the rest of the paper, we will sometimes omit explicit dependence of $\hat{Z}_\mathfrak{s}$ on $\mathfrak{s}$, and/or $M$, and/or $q$.

\begin{Remark}
 Conjecture~\ref{conj:cgpz} together with the conjectural relation between $\hat{Z}_\mathfrak{s}$ and mod-2-cohomology-refined WRT invariant (see Appendix~\ref{app:wrt-refined}) can be used to calculate the limiting values of $\hat{Z}_\mathfrak{s}$ at $q\rightarrow {\rm e}^{\frac{2\pi {\rm i}}{r}}$, $r\neq 0\bmod 4$. Namely, assuming that $\mathcal{T}(M,[\omega])\neq 0$ for $\omega\notin H^1(M;\Z/2\Z)$, one can invert the linear transform~\eqref{CGPZhatviac3} formally substituting
 \begin{equation*}
 \N_r(M,\omega)/\mathcal{T}(M,[\omega])\rightsquigarrow (-1)^r\WRT_r(M,\omega)/\big({\rm i}\sqrt{8r}\big),
 \end{equation*}
 when $\omega\in H^1(M;\Z/2\Z)$. For $r=0\bmod 4$ one can similarly use the relation between $\hat{Z}_\mathfrak{s}$ and spin-version of CGP invariants (see Section~\ref{sec:spincase}) together with the spin-refined WRT invariant (see~\cite{Gukov:2020frk}). This can be especially useful since for many 3-manifolds $\hat{Z}_\mathfrak{s}$ can be computed as $q$-series, but their modular properties, needed to determine the limiting values at roots of unity, are not known. Therefore, the relation to CGP and mod-2-cohomology-refined WRT invariants can serve as a substitute for modularity properties of $\hat{Z}_\mathfrak{s}$ in problems that involve limiting values at roots of unity.

 When $\mathcal{T}(M,[\omega])= 0$ but $\N_r(M,\omega)\neq 0$ for some $\omega\notin H^1(M;\Z/2\Z)$ this indicates that $\lim_{q\rightarrow {\rm e}^\frac{2\pi {\rm i}}{r}}\hat{Z}_{\mathfrak{s}}=\infty$ at least for some $\mathfrak{s}$. (The linear combinations of $\hat{Z}_\mathfrak{s}$ producing refined WRT invariants can still have finite limits.)
\end{Remark}

The study of the case of a manifold $M$ with positive first Betti number (denoted $b_1$ from now on) brought us to notice that the invariant $\hat{Z}$ for such manifolds requires an additional structure on $M$, namely the choice of a splitting $H_1(M;\Z)$ into its torsion part and its free part: $\operatorname{Tor}(H_1(M;\Z))\oplus \operatorname{Free}(H_1(M;\Z))$, which we will denote from now on $H_1(M;\Z)=T\oplus F$; we will also denote $b'$ (resp. $b''$) the projection of $b\in H_1(M;\Z)$ in $T$ (resp. $F$). For a fixed spin$^c$ structure on $M$, changing such a choice affects $\hat{Z}_\mathfrak{s}$ by multiplying it by a power of $q$ (see formula~\eqref{splitting-anomaly} for the precise formulation) which we call the splitting anomaly.

A version of Conjecture~\ref{conj:cgpz} for the case of non-rational-homology-spheres is then the following:

\begin{conj}\label{conj:cgpz2}
Let $M$ be any closed oriented $3$-manifold. Choose a splitting $H_1(M;\Z)$ $=T\oplus F$, where $T$ is the torsion part and $F$ the free part.
\begin{enumerate}\itemsep=0pt
\item [$(a)$] There exist $W_{\mathfrak{s}}(M;T,F)\in q^{\Delta_\mathfrak{s}}\Q[[q]]\big[q^{-1}\big]$, $\mathfrak{s}\in \mathrm{Spin}^c(M)$ $($where, as explained above $H_1(M)=T\oplus F$, with $F$ the free part$)$ for some $\Delta_\mathfrak{s}\in \Q$ such that the following holds:
\begin{equation}
\N_r (M,\omega) = \bigg(\sum_{b=b'+r'm, b'\in T, m\in F} c^{\mathrm{CGP}}_{\omega,\sigma(b,s)} W_{\sigma(b,s)} (M)\bigg)\bigg|_{q \to {\rm e}^{\frac{2 \pi {\rm i}}{r}}}
\label{CGPZhatviac2}
\end{equation}
with
\begin{align*}
&c^{\mathrm{CGP}}_{\omega,\sigma(b'+ r'm,s)}= \frac{r^{b_1} \mathcal{T}(M,[\omega])}{{|\Tor H_1(M;\Z)}|}\nonumber
\\
&{}\qquad\qquad\times\!\begin{cases}
 \sum_{\substack{a',f'\in T}}
 {\rm e}^{2\pi {\rm i}\big( -\frac{r-1}{4}\lk(a',a') +\lk(a',f'-b')-\frac{1}{2}\omega(a')
 +\lk(f',f') +\omega''(m)-\frac{1}{4}\mu(M,s)+\frac{1}{2}\big)}
 \\[-1mm]
\hspace{80mm}\text{if}\quad r=1 \bmod 4,
\\[1mm]
 \sqrt{| \Tor H_1(M;\Z)|} \sum_{\substack{a'\in T}}
 {\rm e}^{-\frac{\pi {\rm i}r}{2} q_s(a')-2\pi {\rm i} \lk(a',b')-\pi {\rm i} \omega(a')+\pi {\rm i} \omega''(m)}
 \\[-1mm]
\hspace{80mm}\text{if}\quad r=2\bmod 4,
 \\[1mm]
 \sum_{\substack{a',f'\in T}} {\rm e}^{2\pi {\rm i}\bigl(-\frac{r+1}{4}\lk(a',a')
 -\lk(a',f'+b')-\frac{1}{2}\omega(a') -\lk(f',f')
 +\omega''(m)+\frac{1}{4}\mu(M,s)+\frac{1}{2}\bigr)}
 \\[-1mm]
\hspace{80mm}\text{if}\quad r=3\bmod 4,
\end{cases}\hspace{-5mm}
\end{align*}
where $r'=r$ if $r$ is odd and $\frac{r}{2}$ else and $\omega'$ $($resp. $\omega'')$ is the restriction of $\omega$ on $T$ $($resp. on~$F)$.

\item[$(b)$] If $\hat{Z}_\mathfrak{s}(M)$ is defined, then one can take\footnote{As will be made apparent in Section~\ref{sec:Zhat-def-b1}, for $b_1>0$, the $q$-series $\hat{Z}_\mathfrak{s}$ are defined only up to a certain equivalence relation. The statement is that any representative in the equivalence class can be taken.} $W_\mathfrak{s}(M)=\hat{Z}_\mathfrak{s}(M)$.
\end{enumerate}
\end{conj}

\begin{Remark}
In the above conjecture, the choice of the splitting is auxiliary: as explained above, a different choice will provide different values for $W_{\mathfrak{s}}$ but these choices are compensated by the change in the coefficients $c^{\mathrm{CGP}}$, so that the left-hand side is indeed independent on the splitting. Equivalently, the behavior of $W_{\mathfrak{s}}$ under the choice of such splitting is controlled by the behavior of the coefficients $c^{\mathrm{CGP}}$. Clearly, if $M$ is a rational homology sphere the splitting is irrelevant and the above conjecture reduces to Conjecture~\ref{conj:cgpz}.
\end{Remark}

\begin{Remark}
The conjecture above implicitly assumes that the sum over $m$ in~\eqref{CGPZhatviac2}, which is infinite for $b_1>0$, converges coefficient-wise in the space of Puiseux series in $q$ with complex coefficients and, moreover, the resulting sum is a series of a continuous function in $\{{|q|<1},\allowbreak\operatorname{arg} q\neq 0\}$ with finite radial limits $q\rightarrow {\rm e}^{\frac{2\pi {\rm i}}{r}}$.
\end{Remark}

Again, we can prove the above conjecture under suitable technical hypotheses which are omitted in the following statement:
\begin{Theorem}
Under the technical hypotheses of Theorem~$\ref{thm:CGP-Zhat-torsion-b1}$, Conjecture~$\ref{conj:cgpz2}$ holds true for any pair $(M,\omega)$ with $\omega\in H^1(M;\C/2\Z)\setminus H^1(M;\Z/2\Z)$.
\end{Theorem}

We provide the following tests:
\begin{enumerate}\itemsep=0pt
\item[1)] $M=\Sigma\times S^1$ for some closed oriented surface $\Sigma$;
\item[2)] $M$ obtained by integral surgery on a framed link $L\subset S^3$ for which Conjecture~\ref{conj:park} holds true.
\end{enumerate}

Moreover, an infinite family of cases to which Theorem~\ref{thm:CGP-Zhat-torsion-b1} applies is provided in Example~\ref{ex:theoremworks-b1}.

In~\cite{blanchet2016non}, the invariants $\N_r$ were extended to a TQFT defined on a suitable category of cobordisms decorated with (relative) cohomology classes.
Conjecturally, it should be possible to do the same for the invariant $\hat{Z}$. In Section~\ref{sec:TQFToperations}, we discuss this possibility and define (or sketch) two ``operations on TQFTs" which, if applied to the conjectural TQFT for $\hat{Z}$ would produce the TQFT built in~\cite{blanchet2016non}, thus extending (unfortunately only partially) Conjecture~\ref{conj:cgpz} to the case of cobordisms.

The rest of the paper is organized as follows. In general, we tried to be as self-contained as possible and we ascribed the physical motivations to dedicated sections so that the paper should be accessible for both mathematicians and physicists.
Section~\ref{sec:prelims} gives a self-contained review of all the relevant invariants that we wish to relate and introduces a number of technical tools that are used throughout the paper. After presenting a few families of concrete examples in Section~\ref{sec:examples}, we then proceed to a more general and systematic discussion of the relation between~$\N_r$ and~$\hat{Z}$ invariants in Section~\ref{sec:relation}. (A reader more interested in a general argument may prefer to read Section~\ref{sec:relation} first, before going through the examples in Section~\ref{sec:examples}.) In Section~\ref{sec:TQFT}, $\hat{Z}$ and $\N_r$ are considered as decorated TQFT's and we propose how a relation between this richer structure can extend the relation between numerical 3-manifold invariants in the previous sections. Since both~$\hat{Z}$ and~$\N_r$ are related to Witten--Reshetikhin--Turaev (WRT) invariants, it is natural to ask whether our proposed relation between $\hat{Z}$ and $\N_r$ is compatible with the previously known relations. This question is answered in the affirmative in Section~\ref{sec:WRT}. Finally, many useful facts about various refined invariants, gradings in CGP TQFT, and details related to the order of limits are collected in appendices.

\section{Preliminaries}
\label{sec:prelims}

\subsection{ADO invariants of links in a three-sphere}

One of the oldest non-semisimple invariants is the collection~-- labeled by integer $r \ge 2$~-- of polynomial link invariants introduced by Akutsu--Deguchi--Ohtsuki~\cite{MR1164114}, or ADO invariants for short.
Namely, let $L\subset S^3$ be an oriented framed link with $V$ components and $\xi=\exp{\frac{2\pi {\rm i}}{2r}}$ for some $r\geq 2$. A coloring of $L$ is an assignment to each component $L_I$ of $L$ of a complex number~$\alpha_i$.
Then, the ADO invariant of $L$ colored by $\{ \alpha_i \}$ is an element of $\Z[\xi^{\pm\alpha_1},\ldots , \xi^{\pm \alpha_V}]$ if $V\geq 2$ and an element of $\frac{1}{\xi^{r \alpha}-\xi^{-r\alpha}}\Z[\xi^{\pm \alpha}]$ otherwise \big(here $\xi^\alpha=\exp\big(\frac{2\pi {\rm i}\alpha}{2r}\big)$\big).

On the one hand, these polynomial link invariants are close cousins of the Alexander polynomial, and include the Alexander polynomial as a special case (corresponding to $r=2$), cf.~\cite{MR1197048}. On the other hand, the ADO polynomials can be viewed as close cousins of the quantum group invariants that play a role in the Reshetikhin--Turaev construction. Indeed, the ADO polynomials were later re-formulated by Murakami~\cite{Murakami} in terms of the $R$-matrix and quantum groups of at roots of unity. Specifically, the $R$-matrix (and its inverse) used in~\cite{Murakami} is a map $W_{\lambda} \otimes W_{\mu} \to W_{\mu} \otimes W_{\lambda}$:
\begin{gather}
 R_{kl}^{ij} = \xi^{\frac{1}{2} (\lambda - 2i-2n)(\mu - 2j+2n) + n(n-1)/2}
\frac{\{ i+n;n \} \{ \mu-j+n;n \} }{ \{ n;n \} },\nonumber
\\
\big(R^{-1}\big)_{kl}^{ij} = (-1)^n \xi^{- \frac{1}{2} (\lambda - 2i)(\mu - 2j) - n(n-1)/2}
\frac{\{ j+n;n \} \{ \lambda-i+n;n \} }{ \{ n;n \} },
\label{ADOR}
\end{gather}
where $n=l-i=j-k$, $\{ a \} = {\xi^a - \xi^{-a}}$, $\{ x;n \} = \prod_{i=0}^{n-1} \{ x-i \}$ and $H v_i^{\lambda} = (\lambda - 2i) v_i^{\lambda}$. We warn the reader that with respect to the notation we will adopt in the definition given in Section~\ref{defiADO} the the colors $\mu,\lambda$ above are highest weight rather than middle weights of the modules. This creates a shift of $r-1$ in the colors with respect to the convention adopted later on, but we mention the above expression because we will return to it in Section~\ref{sec:GPPV} and compare it to the $R$-matrix that one encounters in the study of $\hat Z$-invariants.

Also, in Section~\ref{sec:CGP} below we review how the ADO invariants were used in~\cite{costantino2014quantum} to define invariants of $3$-manifolds $M$ endowed with a cohomology class $\omega\in H^1(M;\C/2\Z)$. Namely, if $L$ is a link such that $M$ is obtained by integral surgery on $L$, then $\omega$ induces a coloring on $L$ by setting $\overline{\alpha}_I=\omega(\mathfrak{m}_I)$ where $\mathfrak{m}_I$ is the oriented meridian of $L_I$.

The study of these invariants remained detached from physics for almost 30 years. This is especially surprising given a large number of close ties that WRT invariants of knots and 3-manifolds have with quantum field theory and string theory. The situation started to change about a year ago~\cite{Gukov:2020lqm} and we hope that the present paper can be another step toward bridging this gap, see also~\cite{Brown:2020qui,Chae:2020ldd,Cheng:2018vpl,MR3611724,Costello:2018swh, toappear1,toappear2,Ferrari:2020avq,Qiu:2020mji} for closely related work.

\subsection[Combinatorics of spin and spin$^{c}$-structures]{Combinatorics of spin and spin$^{\boldsymbol c}$-structures}
\label{sec:combinatorial}

Let $M$ be obtained as an integral surgery on a framed oriented link $L$ in $S^3$ and let $\omega \in H^1(M;\C/2\Z) \setminus H^1(M;\Z/2\Z).$
Let us index the components of $L$ by a set $\vert$ and denote its $V\times V$ linking matrix by $B_{IJ}$, $I,J\in \vert$, $V:=|\vert|$. Then we have the following identifications (see~\cite{deloupmassuyeau} for details):
\begin{gather}
H_1(M;\Z)\cong \Z^\vert /B\Z^\vert,
\label{plumbed-H_1}
\\
H^1(M;\Z_2)\cong \bigg\{c\in \Z_2^\vert \bigg| \sum_{J\in \vert}B_{IJ}c_J=0\bmod 2, \forall I\bigg\},
\label{plumbed-H1-Z2}
\\
H^2(M;\Z)\cong \{h\in \Z^\vert/B\Z^\vert \},
\label{plumbed-H2-Z}
\\
\Spin(M)\cong \bigg\{s\in \Z_2^\vert \bigg| \sum_{J\in \vert}B_{IJ}s_J=B_{II}\bmod 2, \forall I\bigg\},
\label{plumbed-spin}
\\
\mathrm{Spin}^c(M) \cong \operatorname{Char}(B)/2B\Z^\vert =
\big\{ K\in \Z^\vert /2B\Z^\vert \mid K_I=B_{II}\bmod 2,\forall I \big\},\label{spinc}
\end{gather}
where $\operatorname{Char}(B):=\big\{K\in \Z^\vert \mid K^Tn=n^TBn\bmod 2,\, \forall n\in\Z^\vert\big\}$ is the space of characteristic vectors of the lattice dual to the lattice $\Z^\vert$ with the quadratic form $B$. Here and in what follows, we use $(\ldots)^T$ to denote vector transposition. The one-to-one correspondence between the elements $a$ and $K$ in~\eqref{plumbed-spinc} is given by $a=K-B\varepsilon$, where
\begin{equation}
 \varepsilon:=(1,1,1,\ldots,1) \in \Z^\vert.
 \label{epsilon-vector}
\end{equation}
The Bockstein homomorphism $\beta\colon H^1(M;\Z_2)\to H^2(M;\Z)$ writes in the above notation as follows; for $c\in \Z_2^\vert$ let $\tilde{c}\in \Z^\vert$ be such that $\tilde{c}\cong c \bmod 2$. Then
\begin{equation}
 \beta(c)=\frac{1}{2}B\tilde{c}\in \Z^\vert/B\Z^\vert,
\end{equation}
where we observe that since $B\tilde{c}$ is even, division by two is possible and its result in $\Z^\vert/B\Z^\vert$ is independent on the choice of $\tilde{c}$.
We observe that $\Spin(M)$ is affine over $H^1(M;\Z_2)$ via component-wise addition and $\mathrm{Spin}^c(M)$ is affine over $H^2(M;\Z)$ by defining $(K+[h])_I:=K_I+2h_I$, $\forall I\in \vert$.
Observe also that there is a canonical map $i\colon\Spin(M)\to \mathrm{Spin}^c(M)$ which in the above notation writes
\begin{equation}
i(s)_I=\sum_{J\in \vert} B_{IJ} \tilde{s}_J,
\end{equation}
where $\tilde{s}$ is any lift of $s\in \Z_2^\vert$ to $\tilde s \in \Z^{\vert}$ so that $s=\tilde{s} \bmod 2$.
Furthermore, it is clear that $i$ is affine over the Bockstein homomorphism
\begin{equation}
i(s+c)=i(s)+2\beta(c)\qquad \forall s\in \Spin(M), \quad c\in H^1(M;\Z_2).
\end{equation}

Later we will use the canonical map
$\sigma$
\begin{equation}
 \sigma\colon \ H_1(M;\Z)\times\operatorname{Spin}(M) \longrightarrow \operatorname{Spin}^c(M)
 \label{eq:definitionsigma}
\end{equation}
induced by the map $ B\operatorname{Spin}\times B\operatorname{U}(1) \rightarrow B\operatorname{Spin}^c$ between the corresponding classifying spaces, combined with the isomorphisms $B\operatorname{U}(1)\cong B^2\Z$, $H_1(M;\Z)\cong H^2(M;\Z)$.
In terms of the above combinatorial encodings, it is
\[
\sigma(s,b)=i(s)+2b.
\]
We will also need the linking pairing
\begin{equation}
 \lk\colon \ \Tor H_1(M;\Z)\otimes \Tor H_1(M;\Z)\longrightarrow \Q/\Z
\end{equation}
and its quadratic refinement~\cite{kirby1990p}
\begin{equation}
 q_s\colon \ \Tor H_1(M;\Z) \longrightarrow \Q/2\Z,\qquad q_s(a+b)-q_s(a)-q_s(b)=2 \lk (a,b).
\end{equation}
depending on a spin structure $s\in \Spin(M)$ as follows:
\begin{equation}
 q_{s+c}(a)-q_s(a)=c(a),\qquad c\in H^1(M;\Z_2).
 \label{qref-spin-change}
\end{equation}
In terms of the identifications~\eqref{plumbed-H_1}--\eqref{plumbed-spinc}, we have
\begin{gather}
 \lk(a,b) = a^TB^{-1}b\bmod 1,
 \label{plumbed-lk}
\\
 q_s(a) = a^TB^{-1}a+s^Ta\bmod 2.
 \label{plumbed-qref}
\end{gather}
We will also use the following expression for the mod 4 reduction of Rokhlin invariant (see, e.g.,~\cite{kirbymelvin}):
\begin{equation}
 \mu(M,s)=\sigma -s^TBs\bmod 4,
 \label{Rokhlin-mod4}
\end{equation}
where $\sigma$ is the signature of the linking matrix $B$.

A special case is when $M$ is a plumbed manifold, i.e., all the components of $L$ are unknots which are linked to each other as Hopf links according to the combinatorial structure of a~contractible graph $\Gamma$, the vertices of which are indexed by $\vert$.
In this case, letting $\deg(I)$ be the degree of $I\in \vert$ (i.e., the number of edges containing $v$), we also have the following identification for spin$^c$-structures:
\begin{equation}\label{plumbed-spinc}
\mathrm{Spin}^c(M) \cong \big\{a\in \Z^\vert /2B\Z^\vert \mid a_I=\deg(I)\bmod 2\big\}.
\end{equation}

\subsection{WRT invariants}

Let $\xi=\exp{\frac{\pi {\rm i}}{r}}$, $\CU^H_\xi(\mathfrak{sl}_2)$ be the so-called ``unrolled version'' of the quantum $\mathfrak{sl}_2$ algebra as defined in~\cite{costantino2014quantum}, given by generators $E$, $F$, $H$, $K$, $K^{-1}$
and relations
\begin{gather*}
 KK^{-1}=K^{-1}K=1, \qquad
 KEK^{-1}=\xi^2E, \qquad
 KFK^{-1}=\xi^{-2}F,\qquad
 [E,F]=\frac{K-K^{-1}}{\xi-\xi^{-1}},
 \\
 HK=KH,\qquad
 [H,E]=2E, \qquad
 [H,F]=-2F.
\end{gather*}
The algebra $\CU^H_\xi(\mathfrak{sl}_2)$ is a Hopf algebra where the coproduct, counit and
antipode are defined as follows:
\begin{alignat*}{4}
 &\Delta(E)= 1\otimes E + E\otimes K,\qquad
 &&\varepsilon(E)= 0,\qquad
 &&S(E)=-EK^{-1},
 \\
 &\Delta(F)=K^{-1} \otimes F + F\otimes 1,\qquad
&& \varepsilon(F)=0, \qquad
&& S(F)=-KF,
 \\
& \Delta(K)=K\otimes K, \qquad
&& \varepsilon(K)=1,\qquad
&& S(K)=K^{-1},
 \\
 &\Delta(H)=H\otimes 1 + 1 \otimes H,\qquad
 && \varepsilon(H)=0,\qquad
 &&S(H)=-H.
\end{alignat*}

Let $\cat$ be the category of finite dimensional weight modules (i.e., modules on which $H$ acts diagonally) and on which $K$ acts as $\xi^H$ and such that $E^r$ and $F^r$ act as $0$.
We will now recall that the category $\cat$ is a ribbon category.
 Let $V$ and $W$ be
objects of~$\cat$. Let $\{v_i\}$ be a basis of~$V$ and~$\{v_i^*\}$ be
a dual basis of $V^*=\Hom_\C(V,\C)$. Then
\begin{gather*}
 \coev_V\colon \ \C \rightarrow V\otimes V^{*}, \qquad \text{given by } 1 \mapsto \sum
 v_i\otimes v_i^*,
 \\
 \ev_V\colon \ V^*\otimes V\rightarrow \C, \qquad \text{given by }
 f\otimes w \mapsto f(w)
\end{gather*}
are duality morphisms of~$\cat$.
In~\cite{ohtsuki}, Ohtsuki truncates the usual formula of the $h$-adic
quantum~$\mathfrak{sl}_2$ $R$-matrix to define an operator on $V\otimes W$ by
\begin{equation*}
 R=\xi^{H\otimes H/2} \sum_{n=0}^{r-1} \frac{\{1\}^{2n}}{\{n\}!}\xi^{n(n-1)/2} E^n\otimes F^n,
\end{equation*}
where $\xi^{H\otimes H/2}$ is the operator given by
\begin{equation*}
\xi^{H\otimes H/2}(v\otimes v') =\xi^{\lambda \lambda'/2}v\otimes v'
\end{equation*}
for weight vectors $v$ and $v'$ of weights of $\lambda$ and
$\lambda'$. The $R$-matrix is not an element in $\CU^H_\xi(\mathfrak{sl}_2)\otimes \CU^H_\xi(\mathfrak{sl}_2)$,
however the action of $R$ on the tensor product of two objects of~$\cat$ is a well defined linear map on such a tensor
product. So, $R$ gives rise to a braiding $c_{V,W}\colon V\otimes W
\rightarrow W \otimes V$ on~$\cat$ defined by $v\otimes w \mapsto
\tau(R(v\otimes w))$, where $\tau$ is the permutation $x\otimes
y\mapsto y\otimes x$.
 Also, let $\theta$ be the operator given by
\begin{equation*}
\theta=K^{r-1}\sum_{n=0}^{r-1}
\frac{\{1\}^{2n}}{\{n\}!}\xi^{n(n-1)/2} S(F^n)\xi^{-H^2/2}E^n,
\end{equation*}
where $\xi^{-H/2}$ is an operator defined on a weight vector $v_\lambda$ by
$\xi^{-H^2/2}.v_\lambda = \xi^{-\lambda^2/2}v_\lambda$.
Ohtsuki shows that the family of maps $\theta_V\colon V\rightarrow V$ in
$\cat$ defined by $v\mapsto \theta^{-1}v$ is a twist (see
\cite{ohtsuki}).

Now the ribbon structure on $\cat$ yields right duality morphisms
\begin{equation*}
{\tev_{V}} = {\ev_{V}}c_{V,V^*}(\theta_V\otimes\Id_{V^*})\qquad\text{and}\qquad
{\tcoev_V} =(\Id_{V^*}\otimes\theta_V)c_{V,V^*}\coev_V,
\end{equation*}
which are compatible with the left duality morphisms $\{\coev_V\}_V$ and
$\{\ev_V\}_V$. These duality morphisms are given by
\begin{gather*}
 \tcoev_{V}\colon \ \C \rightarrow V^*\otimes V,\qquad \text{where}\quad
 1 \mapsto \sum K^{r-1}v_i \otimes v_i^*,
 \\
 \tev_{V}\colon \ V\otimes V^*\rightarrow \C,\qquad \text{where}\quad v\otimes f \mapsto f\big(K^{1-r}v\big).
\end{gather*}

\begin{Definition}\label{def:qdim}
The \emph{quantum dimension} $\qdim(V)$ of an object $V$ in $\cat$ is the $\qdim(V)= {\tev_V\circ \coev_V}=\sum v_i^*\big(K^{1-r}v_i\big)$.
\end{Definition}

For each $n \in \{0,\ldots,r-1\}$, let $S_n$ be the usual
$(n+1)$-dimensional simple highest weight $\mathcal{U}_\xi^H(\mathfrak{sl}_2)$-module with
highest weight $n$. The module $S_n$ is a highest weight module with a highest weight vector $s_0$ such that $Es_0=0 $ and $Hs_0=ns_0$. Then
 $\{s_0, s_1,\ldots, s_n\}$ is a basis of $S_n$, where $Fs_i=s_{i+1}$, $H.s_i=(n-2i)s_i$, $E.s_0=0=F^{n+1}.s_0$
and $E.s_i=\frac{\{i\}\{n+1-i\}}{\{1\}^2}s_{i-1}$. The quantum
dimension of $S_n$ is
$\qdim(S_n)=(-1)^n\frac{\{n+1\}}{\{1\}}$.
Next we consider a larger class of finite dimensional highest weight modules:
for each $\alpha\in \C$, we let $V_\alpha$ be the $r$-dimensional
highest weight $\mathcal{U}^H_\xi(\mathfrak{sl}_2)$-module of highest weight $\alpha + r-1$ (we stress here that $\alpha$ is then the ``mid-weight'' of $V_\alpha$ as opposed to its highest weight which is $\alpha+r-1$). The
modules $V_\alpha$ has a basis $\{v_0,\ldots,v_{r-1}\}$ action on which is
given by
\begin{equation*}
H.v_i=(\alpha + r-1-2i) v_i,\qquad
E.v_i= \frac{\{ i\}\{i-\alpha\}}{\{1\}^2}v_{i-1} ,\qquad
F.v_i=v_{i+1}.
\end{equation*}
For all $\alpha\in \C$, the quantum dimension of $V_\alpha$ is zero:
\begin{equation*}
\qdim(V_\alpha)= \sum_{i=0}^{r-1} v_i^*\big(K^{1-r}v_i\big)=
 \sum_{i=0}^{r-1} \xi^{(1-r)(\alpha + r-1-2i)} =
 \xi^{(1-r)(\alpha + r-1)}\frac{1-\xi^{2r}}{1-\xi^{2}}=0.
\end{equation*}

As shown in~\cite{costantino2015relations}, if $L$ is a framed oriented link in $S^3$ colored by modules $S_{c_I}$, where $c_I$ denotes the coloring of the $I^{\rm th}$ component of the link, then
\begin{equation*}
F(L)=J_{{c}}(L),
\end{equation*}
where $F$ is the Reshetikhin--Turaev functor and $J_{{c}}$ is the (unnormalised) skein theoretical colored Jones polynomial of $L$, so that in particular the value for the $S_i$-colored unknot is $(-1)^i[i+1]$ (as usual, $[n]:=\{n\}/\{1\}$). Another well known version of the colored Jones polynomial, which we shall call ``representation theoretical'', is obtained by considering the full subcategory $\cat'$ of~$\cat$ generated by the modules $S_n$ with the same braiding and ribbon structure as above, with the only difference in the following morphisms:
\begin{gather*}
\tcoev_{V}\colon \ \C \rightarrow V^*\otimes V,\qquad \text{where}\quad 1 \mapsto
 \sum K^{-1}v_i \otimes v_i^*,
 \\
 \tev_{V}\colon \ V\otimes V^*\rightarrow \C, \qquad\text{where}\quad v\otimes f \mapsto f(Kv).
\end{gather*}
Then the image of $L$ via the Reshetikhin--Turaev functor associated to the category $\cat'$ is the version of the colored Jones considered for instance in~\cite{kirbymelvin}, let us denote it $V_c(L)$.
\begin{Remark}
The definition of representation theoretical Jones polynomial can be extended to generic values of $q$ as opposed to $\xi=\exp\big(\frac{{\rm i}\pi}{r}\big)$.
\end{Remark}

\begin{Lemma}\label{lem:comparison}
It holds $V_{{c}}(L)=(-1)^{\sum_{I\in \vert}(B_{II}+1)c_I}J_{{c}}(L).$
\end{Lemma}
\begin{proof}
It is sufficient to compare when $L$ is the closure of a braid so that for each component of the resulting link we have $B_{II}=0$ for all $I$. Then the only difference between the two Reshetikhin--Turaev functors associated to $\cat'$ and $\cat$ is coming from the $m$ maxima: a maximum colored by $c_I$ will contribute to $J_{{c}}(L)$ via $K^{1-r}$ and to $V_c(L)$ via $K^{}$ and since $K^r$ acts as $\xi^{rc_I}=(-1)^{c_I}$ the overall difference is $(-1)^{\sum c_I}$ where the sum is taken over all the $m$ strands of the braid. But now we claim that this sign coincides with $(-1)^{\sum_{I}(B_{II}+1)c_I}$. Indeed, observe that $B_{II}$ is not changed (modulo $2$) if one switches the crossings of the braid so that we can prove the claim when the link is actually a disjoint union of unknots. Now for each such unknot colored by odd $c_I$ we observe that each Reidemeister $1$ move changes by one both the number of maxima and $B_{II}$, while the other Reidemeister moves do not change these values. Finally, for the standard diagram of the unknot the statement is true.
\end{proof}

\begin{Lemma}[symmetry principle for skein theoretical Jones polynomials]\label{lem:symmetryprinciple}
Let $\xi=\exp(\frac{\pi {\rm i}}{r})$ and~$L$ be a framed oriented link colored by $S_{c_1},\ldots, S_{c_V}$ with $0\leq c_I\leq r-2$. For each $a_I\in \{0,1\}$ and $c_I\in \{0,\ldots, r-2\}$, let $a_I*c_I:=a_I(r-2-c_I)+(1-a_I)c_I$, and for ${a}\in \{0,1\}^V$, let $a*c:=(a_1*c_1,\ldots, a_V*c_V)$.
Then it holds:
\begin{equation*}
J_{a*c}(L)={\rm i}^{(r-2)\sum_{I,J\in \vert}B_{IJ}a_Ia_J}(-1)^{\sum_{I,J\in \vert}B_{IJ}a_Ic_J}(-1)^{\sum_{I\in \vert}(B_{II}+1)(r-2)a_I}J_c(L).
\end{equation*}
\end{Lemma}

\begin{proof}
The symmetry principle for the representation theoretical Jones polynomial stated in~\cite[formula~(4.20)]{kirbymelvin} is
\begin{equation*}
V_{a*c}(L)={\rm i}^{(r-2)\sum_{I,J\in \vert}B_{IJ}a_Ia_J}(-1)^{\sum_{I,J\in \vert}B_{IJ}a_Ic_J}V_c(L).
\end{equation*}
By Lemma~\ref{lem:comparison}, for each component where $a_I\neq 0$ we acquire an additional factor
\begin{equation*}
 (-1)^{(B_{II}+1)(c_I-(r-2-c_I))}=(-1)^{(B_{II}+1)(c_I-(r-c_I))}. \tag*{\qed}
\end{equation*}
\renewcommand{\qed}{}
\end{proof}

We end this section by recalling the definition of the Witten--Reshetikhin--Turaev invariants in their TQFT normalisation, according to Blanchet--Habegger--Masbaum--Vogel~\cite{blanchetetal}:
\begin{Definition}
Let $D=\sqrt{\frac{r}{2}}\big(\sin\big(\frac{\pi}{r}\big)\big)^{-1}$ and $M$ be the three-manifold obtained by surgery on a framed oriented link $L$. Then
\begin{equation*}
\WRT_r(M)=D^{-b_0(M)-b_1(M)}U_+^{-b_+}U_-^{-b_-}\sum_{{c}}\qdim(c)J_{{c}}(L),
\end{equation*}
where $U_{\pm}=\sum_{c=0}^{r-2}\qdim(c)J_c(u_{\pm})$, $u_\pm$ is the unknot with framing $\pm$, ${c}$ runs over all the colorings of $L$ with colors in $\{0,1,\ldots, r-2\}$, $\qdim({c})=\prod_{I=1}^V \qdim(c_I)$ (see Definition~\ref{def:qdim}), $b_i(M)$ are the Betti numbers of $M$ and $b_\pm$ are the number of positive (or negative) eigenvalues of the linking matrix of $L$.
\end{Definition}

We also recall that there exists a cohomological refinement of $\WRT_r(M)$ defined for cohomology classes $\omega\in H^1(M;\Z/2\Z)$ (see~\cite{costantino2015relations,kirbymelvin}).
Given such a class and an oriented link $L\subset S^3$ such that $M$ is obtained by surgery over $L$, then $\omega$ induces a parity on the components of $L$ via $\omega(L_I):=\omega(\mathfrak{m}_I)\in \Z/2\Z$, where $\mathfrak{m}_I$ is homology class of the meridian of $L_I$. Let
\begin{equation*}
 \Delta_-^{{\rm SO}(3)}:=\frac{\Delta_-}{(-1)^{r-1}(\xi-\xi^{-1})},\qquad
 \Delta_+^{{\rm SO}(3)}:=-\frac{\Delta_+}{(-1)^{r-1}(\xi-\xi^{-1})},
\end{equation*}
where $\Delta_\pm$ are defined in~\eqref{delta-minus} and~\eqref{delta-plus}. Then one defines (see also Appendix~\ref{app:wrt-refined}):
\begin{Definition}
\label{def-wrt-refined}
\begin{equation*}
\WRT_r(M,\omega)=\frac{D^{-b_0(M)-b_1(M)}}{\big(\Delta_+^{{\rm SO}(3)}\big)^{b_+}\big(\Delta_-^{{\rm SO}(3)}\big)^{b_-}}\sum_{{c}= \omega\bmod 2}\qdim(c)J_{{c}}(L),
\end{equation*}
where ${c}$ runs over all the colorings of $L$ with values in $\{0,1,\ldots, r-2\}$ which are congruent $\bmod~2$ to $\omega(L_I)$ for each component of $L$.
\end{Definition}

\begin{Remark}
With respect to~\cite{kirbymelvin}, we index modules by highest weight vectors instead of by their dimension, this causes a shift by one in the above summation range.
\end{Remark}

\subsection[$\N_r$ invariants]{$\boldsymbol{\N_r}$ invariants}
\label{sec:CGP}

With the above notation, to calculate $\N_r(M,\omega)$ in this case it is enough to introduce the following definitions:\footnote{In~\cite{costantino2014quantum}, $q$ was used instead of $\xi:={\rm e}^{\frac{\pi {\rm i}}{r}}$ and, mainly, the definition of $d(\alpha)$ differs by a factor $(-1)^{r-1}r$: this only affects $\N_r$ by an overall scalar.}
\begin{gather}
 d(\alpha):=\frac{\sin\frac{\pi\alpha}{r}}{\sin\pi\alpha} =\frac{\xi^\alpha-\xi^{-\alpha}}{\xi^{r\alpha}-\xi^{-r\alpha}},
 \label{modified-q-dim-formula}
\qquad
 S(\alpha,\beta):=\xi^{\alpha\beta},
\qquad
 T(\alpha):=\xi^\frac{\alpha^2-(r-1)^2}{2},
\\
 \Delta_-:= \begin{cases}
 0, & r=0\bmod 4,\\
{\rm i}\xi^\frac{3}{2}r^\frac{1}{2}, & r=1\bmod 4, \\
(1-{\rm i})\xi^\frac{3}{2}r^\frac{1}{2}, & r=2\bmod 4, \\
 -\xi^\frac{3}{2}r^\frac{1}{2}, & r=3\bmod 4,
 \end{cases}
 \label{delta-minus}
\\
 \Delta_+=\overline{\Delta_-}.
 \label{delta-plus}
\end{gather}
As before, let us denote by $\mathfrak{m}_I\in H_1(M;\Z)$ the homology class of the oriented meridian of the component of $L$ indexed by $I\in \vert$. Then define
\begin{equation*}
 \mu_I:=\omega(\mathfrak{m}_I)\in \C/2\Z.
\end{equation*}
We will assume from now on that $\mu_I\notin \Z/2\Z$, $\forall I$ (it was proved in~\cite{costantino2014quantum} that one can always find~$L$ presenting $M$ such that this condition is satisfied if $\omega\in H^1(M;\C/2\Z)\setminus H^1(M;\Z/2\Z)$).
Note that
\begin{equation}
 \sum_{J}B_{IJ}\mu_J=0\bmod 2.
 \label{mu-condition}
\end{equation}

\subsubsection[$\N_r$ for links, ADO and Conway--Alexander polynomials]{$\boldsymbol{\N_r}$ for links, ADO and Conway--Alexander polynomials}\label{defiADO}

As proved in~\cite{MR1164114} (but in what follows we use the notation used in~\cite{costantino2014quantum}), there is an invariant $\N_r(L_\mu)$ associated to each framed oriented link with a $\C$-coloring $\mu_1,\ldots, \mu_V$ of its components, which, up to a factor depending on the linking form, is valued in $\C[\xi^{\pm\mu_1},\ldots, \xi^{\pm\mu_V}]$ if $V>1$ and in $\frac{\xi^{\mu_1}-\xi^{-\mu_1}}{\xi^{r\mu_1}-\xi^{-r\mu_1}}\C[\xi^{\pm\mu_1}]$ in the knot case.

This invariant is computed by first cutting open $L$ on one of its components (say the $I^{\rm th}$) to get a $(1,1)$-tangle $T$, then computing the Reshetikhin--Turaev functor $F$ applied to $T$ by considering a color $\mu_J$ as the $r$-dimensional simple projective module $V_{\mu_J}$ over $\CU^H_\xi(\mathfrak{sl}_2)$ thus getting
\begin{equation*}
F(T)=T(V_{\mu_I})\Id_{V_{\mu_I}}
\end{equation*}
for some scalar $T(V_{\mu_I})$; then defining $\N_r(L_\mu)=T(V_{\mu_I})d(\mu_I)$. It can be proved that this value does not depend on the way $L$ was cut to obtain $T$ and is then an invariant of the colored oriented framed link $L_\mu$.

Comparing with the original definition of the ADO polynomial of links~\cite{MR1164114}, the results of~\cite{Murakami} and taking into account the differences of notations, we remark the following: for each framed oriented link $L\subset S^3$ with $V\geq 2$ components we have
\begin{equation}
\mathrm{ADO}_r(L_\mu)\big(x_1=\xi^{-2}\xi^{2\mu_1},\ldots ,x_V=\xi^{-2}\xi^{2\mu_V}\big)=\N_r(L_\mu)\xi^{-\frac{\mu^T B\mu}{2}+\frac{(r-1)^2 \Tr B}{2}}.
\label{ADO-Nr-link}
\end{equation}
For knots, we remark that $\N_r$ is \textit{not} a polynomial and the relation with ADO polynomial is as follows:
\begin{equation}
\mathrm{ADO}_r(K_\mu)\big(x=\xi^{-2} \xi^{2\mu}\big)=\frac{\xi^{r\mu}-\xi^{-r\mu}}{\xi^{\mu}-\xi^{-\mu}}\N_r(K_\mu)\xi^{-\frac{f\mu^2 }{2}+\frac{f(r-1)^2}{2}},
\label{ADO-Nr-knot}
\end{equation}
where $f$ is the framing of $K$.

A case of special interest is when $r=2$ where, as proved in~\cite{blanchet2016non} $\N_r$ is equivalent to the Alexander--Conway function (see Corollary 6.19, taking into account the difference in the definition of the modified quantum dimension used here). The precise statement is the following, letting $\xi={\rm i}$:
{\samepage\begin{align*}
\begin{split}
 \N_2(L_\mu)&={\rm i}\nabla_L\big({\rm i}\xi^{-\mu_1},{\rm i}\xi^{-\mu_1},\ldots,{\rm i}\xi^{-\mu_V}\big) \xi^{\frac{\mu^TB\mu}{2}-\frac{{\varepsilon}^TB{\varepsilon}}{2}}
 \\
&= (-1)^V{\rm i}\nabla_L\big({-}{\rm i}\xi^{\mu_1},-{\rm i}\xi^{\mu_2},\ldots,-{\rm i}\xi^{\mu_V}\big) \xi^{\frac{\mu^TB\mu}{2}-\frac{{\varepsilon}^TB{\varepsilon}}{2}},
\end{split}
\end{align*}
where $\varepsilon$ is the vector $(1,1,\ldots, 1)$ as in~\eqref{epsilon-vector}.}

\begin{Example}\label{ex:plumbedlink}
Let $L$ be a plumbing link the components of which are colored by $\mu_I$, $I\in \vert$, and let $B$ be the linking matrix. Using equations (Nf) and (Nh) of~\cite{costantino2014quantum} one gets
\begin{equation*}
\N_r(L_\mu)=\xi^{-\frac{\Tr B (r-1)^2}{2}}
\xi^{\frac{1}{2}({\mu}^TB{\mu})}\prod_{I\in \vert} d(\mu_I)^{1-\deg(I)}.
\end{equation*}
In the case $r=2$, so that $\xi={\rm i}$ we get
\begin{equation*}
\N_2(L_\mu)=\xi^{-\frac{\Tr B}{2}}
\xi^{\frac{1}{2}({\mu}^TB{\mu})}\prod_{I\in \vert} \big(\xi^{\mu_I}+\xi^{-\mu_I}\big)^{\deg(I)-1}
\end{equation*}
and
\begin{equation*}
\nabla_L(x_1,\ldots, x_V)=\prod_{I}\big(x_I-x_I^{-1}\big)^{\deg(I)-1}.
\end{equation*}
Indeed, remark that $d(\mu)=\frac{\xi^\mu-\xi^{-\mu}}{\xi^{2\mu}-\xi^{-2\mu}}=(\xi^\mu+\xi^{-\mu})^{-1}$ and replacing $x_I$ by $\xi \xi^{-\mu_I}$ we get
\begin{equation*}
\nabla_L\big(\xi \xi^{-\mu_1},\ldots, \xi \xi^{-\mu_V}\big)={\rm i}^{(-V+\sum_I \deg(I))}\prod_{I}\big(\xi^{\mu_I}+\xi^{-\mu_I}\big)^{\deg(I)-1}
\end{equation*}
and remarking that if $E$ is the number of edges in the plumbing graph then $\sum_I \deg(I)=2E$ and that the Euler characteristic of the plumbing graph is $1$ we get
\begin{equation*}
\nabla_L\big(\xi \xi^{-\mu_1},\ldots, \xi \xi^{-\mu_V}\big)={\rm i}^{(-1+\frac{\sum_I \deg(I))}{2}}\prod_{I}\big(\xi^{\mu_I}+\xi^{-\mu_I}\big)^{\deg(I)-1}.
\end{equation*}
Finally, observe that since $\epsilon$ was defined as the vector of whose components are all $1$, and we are considering a plumbing link, we have $\epsilon^T B\epsilon=\sum_{I\neq J\in \vert} B_{IJ}+\operatorname{tr}(B)=\sum_{I\in \vert} \deg(I)+\operatorname{tr}(B)$.
\end{Example}

\subsubsection[$\N_r$ for manifolds]{$\boldsymbol{\N_r}$ for manifolds}

Let $M$ be an oriented $3$-manifold, presented by integral surgery on a link $L$ and endowed with a cohomology class $\omega\in H^1(M;\mathbb{C}/2\mathbb{Z})$ and assume that for the periods of $\omega$ on the meridians of~$L$ are represented by $\mu_k\in (\mathbb{C}\setminus \mathbb{Z})\cup r\mathbb{Z}$, $k\in \vert$.

\begin{Definition}
\begin{equation}\N_r(M,\omega)= \frac{1}{\Delta_+^{b_+}\Delta_-^{b_-}}
 \sum_{k\in H_r^\vert}
 \prod_{I\in \vert} d(\alpha_{k_I}) \N_r(L_\alpha),
 \label{CGP-surgery-formula}
\end{equation}
where
\begin{equation*}
 \alpha_{k_I}:=\mu_I+k_I, \qquad H_r=\{-(r-1),-(r-3),\ldots,(r-3),r-1\}
\end{equation*}
and $b_\pm$ are the number of positive/negative eigenvalues of $B$.
\end{Definition}

$N_r(M,\omega)$ was proved in~\cite{costantino2014quantum} to be an invariant of pairs $(M,\omega)$ up to positive diffeomorphism.

\begin{Remark}
The invariant $\N_r$ defined above differs by a constant with respect to the invariant~$N'_r$ defined in~\cite{costantino2014quantum} and its renormalisation $Z_r$ introduced in~\cite{blanchet2014non} because of the different definition of $d(\alpha)$ and consequently of $\Delta_\pm$ and of the link invariant.
The general relation for a~closed connected manifold is then:
\begin{equation*}
\N_r(M,\omega)=\frac{\N'_r(M,\omega)}{((-1)^{r-1}r)^{b_1(M)+1}} =\frac{r\sqrt{r'}\big(\frac{r^2}{\sqrt{r'}}\big)^{b_1(M)}Z_r(M,\omega)}{((-1)^{r-1}r)^{b_1(M)+1}},
\end{equation*}
where $r'=r$ if $r$ is odd and $\frac{r}{2}$ else.
Also notice that we recalled the definition when all the periods are non integral but actually the definition can be extended to the case at least one period is non-integral (see~\cite{costantino2014quantum}).
\end{Remark}
In the special case of a plumbed $M$, using the notation introduced in Section~\ref{sec:combinatorial} it reads
\begin{equation}
 \N_r(M,\omega)=
 \frac{1}{\Delta_+^{b_+}\Delta_-^{b_-}}
 \sum_{k\in H_r^\vert}
 \prod_{I\in \vert} d(\alpha_{k_I})^{2-\deg(I)} T(\alpha_{k_I})^{B_{II}}
 \prod_{(I,J)\in \text{Edges}} S(\alpha_{k_I},\alpha_{k_J}).\label{eq:NrM}
\end{equation}

\begin{Example}[$\Sigma\times S^1$]
Let $M=\Sigma_g\times S^1$ for a closed oriented surface $\Sigma_g$ of genus $g$ and let $\omega\in H^1(M;\C/2\Z)\setminus H^1(M;\Z/2\Z)$. If $\beta=\omega\big(\{pt\}\times S^1\big)$,
 then as shown in~\cite[Theorem~5.9]{blanchet2014non} it holds
\begin{equation*}
\N_r(M,\omega)=r^{2g} \sum_{k\in H_r} \bigg(\frac{\{r\beta\}}{\{\beta+k\}}\bigg)^{2g-2}.
\end{equation*}
\end{Example}

\subsubsection[A relation between ${\rm N}_r$ and ${\rm WRT}_r$]{A relation between $\boldsymbol{{\rm N}_r}$ and $\boldsymbol{{\rm WRT}_r}$}

The following result outlines a partial direct relation between the invariants $\N_r$ and $\WRT_r$ (cf.~\cite{beliakova2020nonsemisimple}):
\begin{Theorem}
\label{thm-CGP-knot-surgery-limit}
Let $r\geq 2$ be an integer non divisible by $4$ and $K\subset S^3$ be an oriented zero framed knot and $M$ be the surgery on it. Let $\alpha\in \C$ be a color on $K$ and $\omega_\alpha\in H^1(M;\C/2\Z)$ be the unique cohomology class the value of which on the positive meridian of $K$ is $\alpha\ {\rm mod}\ 2\Z$. Then the following holds:
\begin{alignat*}{3}
&\text{if $r$ is odd:}\quad &&\lim_{\alpha\to 0} [r\alpha]^2\N_r(M,\omega_\alpha)=D^2 \WRT_r(M),
\\
&\text{else:} &&\lim_{\alpha\to 0} [r\alpha]^2\N_r(M,\omega_\alpha) = 2D^2 \WRT_r(M,\omega_0),\quad\text{and}
\\
& && \lim_{\alpha\to 1} [r\alpha]^2\N_r(M,\omega_\alpha) = 2D^2 \WRT_r(M,\omega_1),
\end{alignat*}
where $D=\sqrt{\frac{r}{2}}\big(\sin\big(\frac{\pi}{r}\big)\big)^{-1}$, $\WRT_r(M)$ is the standard WRT invariant of $M$, $\WRT_r(M,\omega_i)$ are the cohomology refined invariants of $M$ and $\omega_i\in H^1(M;\Z/2\Z)$ is the cohomology class on $M$ the value of which on the meridian of $K$ is $i=0,1$.
\end{Theorem}
\begin{proof}
Present $K$ as the closure of a $(1,1)$-tangle $T$ and for each absolutely simple module $V$ (i.e., such that $\End(V)=\C$) over $U^H_q(\mathfrak{sl}_2)$ let $T(V)\in \C$ be the scalar such that $F(T)=T(V)\Id_V$, where $F$ is the Reshetikhin--Turaev functor. In particular, we shall use the $(n+1)$-dimensional highest weight simple module $S_n$ and the absolutely simple $r$-dimensional module $V_\alpha$ with highest weight $\alpha+r-1$.

\textbf{Case $\boldsymbol r$ is odd.}
By definition, for $r$ odd we have
\begin{gather*}
[r\alpha]^2\N_r(M,\omega)=\sum_{k=-(r-1),\ \text{by}\ 2}^{r-1} \frac{\{\alpha+k\}^2}{\{1\}^2} T(V_{\alpha+k}),
\\
\WRT_r(M)=D^{-2}\sum_{j=0}^{r-2} \frac{\{j+1\}^2}{\{1\}^2} T(S_{j}).
\end{gather*}
By the symmetry principle for the colored Jones polynomials of knots (Lemma~\ref{lem:symmetryprinciple}, see~\cite[formula~(4.20)]{kirbymelvin}) and by the equality $\{r-1-j\}=\{j+1\}$, we have also
\begin{equation*}
\WRT_r(M)=2D^{-2}\sum_{j=0\ by\ 2}^{r-3} \frac{\{j+1\}^2}{\{1\}^2} T(S_{j}).
\end{equation*}
As shown in~\cite[Proposition~4]{costantino2015relations},
\begin{equation*}
T(S_{r-1-k})=T(V_{k})=T(V_{-k})\qquad \forall k\in \{0,\ldots ,r-1\},
\end{equation*}
so, since by~\cite[Corollary~15]{costantino2015relations} $T(V_\alpha)$ is a continuous function of $\alpha$, we have
\begin{align*}
\lim_{\alpha\to 0} [r\alpha]^2\N_r(M,\omega)&= \sum_{k=-(r-1),\ \text{by}\ 2}^{r-1} \frac{\{k\}^2}{\{1\}^2} T(V_{k})= 2\sum_{k=2,\ by\ 2}^{r-1} \frac{\{k\}^2}{\{1\}^2} T(V_{k})
\\
&=2\sum_{k=2,\ \text{by}\ 2}^{r-1} \frac{\{k\}^2}{\{1\}^2} T(S_{r-1-k})=2\sum_{j=0,\ \text{by}\ 2}^{r-3} \frac{\{r-1-j\}}{\{1\}} T(S_{j})
\\
&=2\sum_{j=0, \ \text{by}\ 2}^{r-3} \frac{\{j+1\}^2}{\{1\}^2} T(S_{j}).
\end{align*}

\textbf{Case $\boldsymbol r$ is even.}
By definition, for $r$ even we have
\begin{gather*}
[r\alpha]^2\N_r(M,\omega_{\alpha})=\sum_{k=-(r-1),\ \text{by}\ 2}^{r-1} \frac{\{\alpha+k\}^2}{\{1\}^2} T(V_{\alpha+k}),
\\
\WRT_r(M,\omega_i)=D^{-2}\sum_{j=i,\ \text{by}\ 2}^{r-2+i} \frac{\{j+1\}^2}{\{1\}^2} T(S_{j}).
\end{gather*}
So
\begin{align*}
\lim_{\alpha\to 0}[r\alpha]^2\N_r(M,\omega_{\alpha})&=\sum_{k=1,\ \text{by}\ 2}^{r-1} 2\frac{\{k\}^2}{\{1\}^2} T(V_{k})=\sum_{k=1\ by\ 2}^{r-1} 2\frac{\{k\}^2}{\{1\}^2} T(S_{r-1-k})
\\
&=\sum_{j=0,\ \text{by}\ 2}^{r-2} 2\frac{\{r-1-j\}^2}{\{1\}^2} T(S_{j})=\sum_{j=0,\ \text{by}\ 2}^{r-2} 2\frac{\{1+j\}^2}{\{1\}^2} T(S_{j})
\\
&=2D^2 \WRT_r(M,\omega_0).
\end{align*}
Similarly,
\begin{align*}
\lim_{\alpha\to 1}[r\alpha]^2\N_r(M,\omega_{\alpha})&=\sum_{k=-(r-1),\ \text{by}\ 2}^{r-1} 2\frac{\{k+1\}^2}{\{1\}^2} T(V_{k+1})=\sum_{s=2,\ \text{by}\ 2}^{r-2} 2\frac{\{s\}^2}{\{1\}^2} T(V_{s})
\\
&=\sum_{s=2,\ \text{by}\ 2}^{r-2} 2\frac{\{s\}^2}{\{1\}^2} T(S_{r-1-s})=\sum_{j=1,\ \text{by}\ 2}^{r-3} 2\frac{\{r-1-j\}^2}{\{1\}^2} T(S_{j})
\\
&=2D^2 \WRT_r(M,\omega_1)
\end{align*}
where in the second equality we used that $\{r\}=\{0\}=0$.
\end{proof}

\subsection[$\hat{Z}$ invariants]{$\boldsymbol{\hat{Z}}$ invariants}\label{sec:GPPV}
In this section, we recall some general facts about the invariants $\hat{Z}_\mathfrak{s}$ of 3-manifolds and their version $F_L (x,q) := \hat Z \big(S^3 \setminus L\big)$ for links. We formulate some conjectures about the link invariant and outline how the $3$-manifold invariant is built out of the invariant for links. We postpone examples to a later section.

Before diving into a more detailed discussion of the $q$-series invariants $\hat{Z}_\mathfrak{s} (M;q)$, we start with a few general comments regarding mathematical definition(s) of these invariants and compare them to other invariants of a similar nature. One important feature that we should stress is that many such invariants and the corresponding TQFTs require a regularization because~-- manifestly or in disguise~-- their space of states $\mathcal{H} (\Sigma)$ on a general 2-manifold $\Sigma$ is infinite-dimensional. This includes Vafa--Witten invariants~\cite{Vafa:1994tf}, the $q$-series invariants $\hat{Z}_\mathfrak{s} (M;q)$, the Teichm\"uller TQFT~\cite{MR3204520}, and even BCGP theory of our interest here.

So currently the value of the invariants is known only for some (infinite) family of surgery presentations of some family of $3$-manifolds.
In a smaller set of cases, it can be shown that these values are invariants of the $3$-manifolds: it is the case of surgeries over plumbing links where the invariance was proved in~\cite[Proposition~4.6]{Gukov:2019mnk} by showing that two different surgery presentations via negative plumbing links yield the same invariant.

Therefore, a mathematical definition that applies to arbitrary manifolds is not available and this can be traced to a need of a suitable regularization. For example, in the case of the Teichm\"uller TQFT, the space of states $\mathcal{H} (\Sigma)$ is expected to be a quantization of a particular component (namely, the Teichm\"uller component) in the space of real ${\rm SL}(2,\mathbb{R})$ flat connections on~$\Sigma$. Note, for this reason, the Teichm\"uller TQFT at best can be described as a ``real Chern--Simons theory,'' unlike $\hat{Z}_\mathfrak{s} (M;q)$, which provides a non-perturbative definition to complex Chern--Simons theory. Indeed, even at the perturbative level, only certain ${\rm SL}(2,\mathbb{R})$ flat connections contribute to the Teichm\"uller TQFT, whereas $\hat{Z}_\mathfrak{s} (M;q)$ generically includes contributions of arbitrary ${\rm SL}(2,\mathbb{C})$ flat connections on $M$. We will return to the TQFT aspects and discuss $\mathcal{H} (\Sigma)$ in more detail
in Section~\ref{sec:TQFT}.

Compared to the above-mentioned similar invariants, the $q$-series $\hat{Z}_\mathfrak{s} (M;q)$ are relatively young, introduced only a few years ago. Nevertheless, they are quickly developing, in part thanks to several complementary approaches that we briefly summarize here:
\begin{itemize}
\item First, we briefly mention a \textit{physics approach} based on 3d-3d correspondence because, at least in principle, it allows to turn the problem of computing $\hat{Z}_\mathfrak{s} (M;q)$ into a concrete computation in a certain QFT associated with $M$. Aside from serving as a bridge between math and physics, this sometimes is capable of producing concrete expressions for $\hat{Z}_\mathfrak{s} (M;q)$ in certain infinite families of 3-manifolds~\cite{Gukov:2017kmk}. A notable example is the expression~\eqref{Zhat-b1-def} for plumbed manifolds that came out of this approach and will be extensively used in this paper.

\item One of the first mathematical approaches to the $q$-series $\hat{Z}_\mathfrak{s} (M;q)$ is based on \textit{quantum groups at generic} $\mathbf{q}$, following the Reshetikhin--Turaev construction at roots of unity. In~this approach, one first needs to construct invariants for links and then, via a surgery formula, produce invariants for closed 3-manifolds. Among all approaches listed here, this one is by far the most user-friendly and efficient as it comes to computations. In particular, it produces invariants for many infinite families of knots, such as all torus knots, positive braid knots, homogeneous braid links, and many other examples, including e.g., all fibered knots up to 10 crossings~\cite{Park:2020edg,Park2021}. The computations in all these examples have been shown to be invariant under Reidemeister moves, but the challenge in this approach comes at the level of 3-manifolds: while the invariants of closed 3-manifodls are very easy to compute in each case, their invariance under Kirby moves has not been demonstrated in general. Moreover, the surgery formula which is used, in order to be well defined, requires the link framings to satisfy certain inequalities. Needless to say, this is one of the good problems for future research.

\item A very different approach to the mathematical formulation of the $q$-series $\hat{Z}_\mathfrak{s} (M;q)$ is based on the \textit{geometry of affine Grassmannians}. Namely, it was proposed in~\cite{Gukov:2020lqm} that $\hat{Z}_\mathfrak{s} (M;q)$ can be defined as a Rozansky--Witten TQFT with the target space given by the space of BFN triples, a particular model for the total space of the cotangent bundle of the affine Grassmannian Gr$_G$. Since Rozansky--Witten theory admits a rigorous mathematical definition~\cite{MR1671737,MR1671725}, this formulation is mathematically rigorous and applies to all closed 3-manifolds, but in practice is rather hard to compute.\footnote{Another, related challenge in this approach is that the variable $q$ appears as a formal parameter (associated with the loop rotation) and, therefore, produces only a formal power series in $q$. This definition would not be sufficient e.g., for the purposes of the present paper since it does not allow evaluating $\hat{Z}_\mathfrak{s} (M;q)$ at generic points in $|q|<1$ that we need for approaching the roots of unity and studying the connection with the BCGP invariants.\looseness=1}

\item There are candidates for several other, even more computationally-challenging, ap\-proach\-es~\cite{Gukov:2017kmk} based on analysis and moduli spaces in \textit{gauge theory}~\cite{Witten:2011zz} and in \textit{enumerative geometry}~\cite{Ekholm:2020lqy}. Potentially, they can provide alternative mathematical definitions of the $q$-series $\hat{Z}_\mathfrak{s} (M;q)$ and produce important connections to other branches of mathematics.

\item Yet another approach, based on \textit{resurgent analysis}, is in many ways the ``middle of the road.'' It is computationally not as efficient as the approach based on quantum groups at generic~$q$. On the other hand, even though this approach is also analytic in nature, it is much easier than the analysis involved in studying the moduli spaces of gauge theory PDEs or curve counting. In particular, using this approach the expressions for $\hat{Z}_\mathfrak{s} (M;q)$ computed via other methods listed here were reproduced in large families of 3-manifolds, see, e.g.,~\cite{MR4400935,Chung:2020efy,Gukov:2016njj}.
\end{itemize}

Following this brief summary, it may be worth emphasizing the distinction between ability to demonstrate invariance under Kirby moves in complete generality versus providing a certificate in each individual case. This is similar to a Sudoku puzzle where, despite the lack of a general fast algorithm, verifying correctness of solutions can be done quickly (in polynomial time). For example, as mentioned in this summary, the approach based on quantum groups at generic~$q$ provides a quick and efficient way to compute $\hat{Z}_\mathfrak{s} (M;q)$ based on a surgery presentation. Therefore, using this approach one can easily {\it check} Kirby moves in lots of examples, one-by-one, even before the general proof is available.

Therefore, in the above summary we find one approach (based on quantum groups at gener\-ic~$q$) which is easy to compute, but which is not defined in general, and another approach (based on the geometry of affine Grassmannians) which is well defined but is very hard to compute. An obvious conjecture is that these two approaches define the same invariants $\hat{Z}_\mathfrak{s} (M;q)$. In~other words, it is expected that the two approaches~-- one, which is hard to compute and the other which is hard to define~-- are equivalent. Due to the difficulty of computing $\hat{Z}_\mathfrak{s} (M;q)$ via Rozansky--Witten theory, the validity of this conjecture has been established only for very simple 3-manifolds, such as Lens spaces and $M = S^1 \times S^2$.

\subsubsection[$\hat{Z}$ for links]{$\boldsymbol{\hat{Z}}$ for links}

A physical construction of new $q$-series invariants of 3-manifolds was proposed in~\cite{Gukov:2017kmk,Gukov:2016gkn}. Much like invariants $\N_r (M,\omega)$ reviewed above, these $q$-series invariants are labeled by extra data, which originally was interpreted as the choice of abelian flat connection or, equivalently,\footnote{In this paper we only consider $\operatorname{SU}(2)$ version of these invariants. The higher-rank version is also available, see, e.g.,~\cite{Chung:2018rea,Park:2019xey}.} an element of $H_1 (M)$.
Soon, it was realized~\cite{Gukov:2020lqm,Gukov:2019mnk} that the extra data $\mathfrak{s}$ which labels $\hat{Z}_\mathfrak{s}(M)$ should be understood as a spin$^c$ structure on $M$. Recall, that as a set $H_1 (M)$ is isomorphic to Spin$^c (M)$, but the isomorphism is non-canonical. The difference between the two becomes apparent in cutting and gluing operations, which in part explains why it was not noticed until the invariants were extended to 3-manifolds with toral boundaries and link complements, where~$\mathfrak{s}$ is a relative spin$^c$-structure. In particular, if $M$ is the complement of a knot $K$ in rational homology sphere, the invariant is expected to be read off from a single two-variable series:\footnote{The notation $R[q^{-1},q]]$ means the ring of formal Laurent power series in $q$, with coefficients in $R$, where one can have infinitely many terms with positive powers of $q$ but only finitely many terms with negative powers of $q$.}
\begin{equation*}
F_K(x,q)\in 2^{-c} q^{\Delta} \Z\big[q^{-1},q\big]\big]\big[\big[x^{\pm\frac{1}{2}}\big]\big]
\end{equation*}
for some $c\in \Z_+$ and $\Delta\in \Q$. Apart from the original physics formulation, several approaches toward rigorous constructions of this invariant have been developed. For example, the approach based on recursion and resurgence, although in principle can be applied to any knot (or link) is rather laborious and was used in~\cite{Gukov:2019mnk} to produce $F_K (x,q) := \hat Z \big(S^3 \setminus K\big)$ only for torus knots and a single hyperbolic knot, the figure-eight knot $K = {\bf 4_1}$.

Another, much more efficient diagrammatic approach based on the $R$-matrix for Verma modules and quantum groups at generic $q$ was proposed by Park~\cite{Park:2020edg,Park2021}. Using this approach, one can quickly compute $F_K(x,q)$ for many hyperbolic knots up to roughly 10 crossings and also for many infinite families. Specifically, the $R$-matrix used by Park~\cite{Park:2020edg} is the universal $R$-matrix~\cite{MR1026957} applied to the lowest weight Verma modules with complex weights (determined by complex variables $x$ and $y$ in this expression):
\begin{equation}
R^{ij}_{kl} = \bigl[\!\begin{smallmatrix} j \\ n \end{smallmatrix}\!\bigr] \prod_{l=1}^n \big(1 - y^{-1} q^{i+l}\big) x^{-\frac{\rm i}{2}-\frac{n}{4}} y^{-\frac{j-n}{2}+\frac{n}{4}} q^{(j-n)i + \frac{(j-n)n}{2} + \frac{j-n+i+1}{2}},
\label{ZhatR}
\end{equation}
where $n = l-i = j-k$, $x$ and $y$ are the exponentiated values of the two complex weights (i.e., the variable $y$ has the same meaning as the variable $x$ that we already encountered in the discussion of $F_K (x,q)$), and
\begin{equation}
\bigl[\!\begin{smallmatrix} j \\ n \end{smallmatrix}\!\bigr] = \frac{[j]!}{[n]! [j-n]!} = \frac{ \{ 1 \} \cdots \{ j \} }{ \{ 1 \} \cdots \{ n \} \{ 1 \} \cdots \{ j-n \} } = \frac{ \{ j - n + 1 \} \cdots \{ j \} }{\{ 1 \} \cdots \{ n \}} = \frac{ \{ j;n \} }{ \{ n;n \} }
\label{qfactident}
\end{equation}
One of the main difficulties in Park's constructions is to deal with closures of braids and making sense of the resulting infinite sums. It is this delicate aspect that leads to some of the above mentioned conditions on knots and links currently covered by this construction.

Here, and in relating it to the $R$-matrix~\eqref{ADOR} we are a little cavalier with the overall powers of $q$. Indeed, the fact that~\eqref{ADOR} and~\eqref{ZhatR} are related by sending $q$ to a root of unity is the first indication for the relation~\eqref{FKADOnorm} between the corresponding knot invariants. This also clarifies the quantum group origin of the $\hat Z$-invariants. Since they basically provide a non-perturbative formulation\footnote{The all-order perturbative formulation was known for quite some time, see, e.g.,~\cite{Dimofte:2009yn} and references therein.} of the ${\rm SL}(2,\C)$ Chern--Simons TQFT~-- where $q$ is a continuous complex variable and so are the highest weights of representations by which Wilson lines are colored~-- it was expected for a long time that ${\rm SL}(2,\C)$ Chern--Simons theory should be described by $U_q ({\mathfrak g})$ with generic $q$.

In order to relate the $R$-matrices~\eqref{ADOR} and~\eqref{ZhatR}, we first need to replace $q \to q^{-2}$ and then take the root of unity limit. As in the generalized volume conjecture~\cite{Gukov:2003na}, all combinations of the form $q^{\text{weight}}$ must be treated as independent variables and kept fixed. In particular, we need to use the identifications $x = q^{\mu}$ and $y = q^{\lambda}$ and treat them as independent complex parameters. Then, using identities of the form~\eqref{qfactident} we arrive at the desired relation between~\eqref{ZhatR} and~\eqref{ADOR}. Based on this relation between the $R$-matrices, it is not surprising to expect the following conjectural relation between the corresponding knot invariants~\cite{Chae:2020ldd,Gukov:2020lqm} (cf.~Conjecture~\ref{conj:park}):
\begin{equation*}
\lim_{q\to \exp(2{\rm i}\pi/r)}F_K(x_1,q)\Delta_K(x_1^r) =\mathrm{ADO}_r(K)\bigg(\frac{x_1}{q}\bigg)\Big|_{q=\exp{\frac{2\pi {\rm i} }{r}}},
\end{equation*}
where $\Delta_K$ is the Alexander polynomial of $K$. The $r=1$ case of the conjecture holds automatically if one assumes the definition of $F_K$ for a class of knots in~\cite{Park2021}. For some other values of $r$, it has been previously verified for certain knots.

To properly deal with the case of links, we consider the set
\begin{equation*}
\C\big[q^{-1},q\big]\big]\big[\big[x_1^{\pm\frac{1}{2}},\ldots, x_V^{\pm\frac{1}{2}}\big]\big]
\end{equation*}
of formal power series in $\big\{x_I^{\pm \frac{1}{2}}\big\}$ with coefficients in Laurent series in $q$.
Such a set is not a ring but is a module over the ring:
\begin{equation*}
\C\big[q^{-1},q\big]\big]\big[x_1^{\pm\frac{1}{2}},\ldots, x_V^{\pm\frac{1}{2}}\big]
\end{equation*}
of Laurent polynomials in $\big\{x_I^{\pm \frac{1}{2}}\big\}$ with coefficients in Laurent series in $q$.

Taking into account the different normalisation between $\mathrm{ADO}_r(L)$ and $\N_r(L)$, we propose the following conjecture for the link case:

\begin{conj}\label{conj:FADO2}\quad
\begin{enumerate}\itemsep=0pt
\item[$(a)$] For each framed link $L$ in $S^3$ there exists a non-zero formal power series
\[
W_L\in \Z\big[q^{-1},q\big]\big]\big[\big[x_1^{\pm \frac{1}{2}},\ldots,x_V^{\pm \frac{1}{2}}\big]\big]
\]
such that the following holds for every $r\geq 2$ and for every $\vec{\alpha}\in \mathbb{R}^n$:
\begin{equation}
\lim_{t\to 1} \lim_{q\to \exp(\frac{2\pi {\rm i}}{r})} \frac{W^t_L(q^{\alpha_1},\dots, q^{\alpha_V},q)}{W^t_L(q^{r\alpha_1},\dots, q^{r\alpha_V},q^r)}=\N_r(L_{\alpha})q^{{-\frac{1}{4}}({\alpha}^TB\alpha-(r-1)^2 \Tr B)}, \label{CGP-Zhat-knot}
\end{equation}
where $B$ is the linking matrix of $L$ and $W^t_L$ is the formal power series in $t$ obtained by replacing each $x_I^{n/2}$, $n>0$ by $t^nq^{\alpha_I/2}$ and each $x_I^{-n/2}$ by $t^nq^{-\alpha_I/2}$.
In general, the radius of convergence of $W^t_L$ might be less than one but we conjecture that $\lim_{q\to \exp(\frac{2\pi {\rm i}}{r})}W^t_L$ is actually a rational function of $t$ so that the limit $t\to 1$ makes sense via analytic continuation.

\item[$(b)$] When $F_L\equiv \hat{Z}\big(S^3\setminus L\big)$ is defined, one can take $W_L=F_L$.
\end{enumerate}
\end{conj}

The relation of Conjecture~\ref{conj:FADO2} with the previous ones is clarified by the following:

\begin{conj}\label{conj:FADO3}\quad
\begin{enumerate}\itemsep=0pt
\item[$(a)$] If $L\subset S^3$ is a framed oriented link multivariable Alexander polynomial $\nabla_L$ of which is non-zero then one can choose $W_L(x,q)$ in Conjecture~$\ref{conj:FADO2}$ such that the following holds:
\begin{equation*}
\lim_{t\to 1}\lim_{q\to 1} W^t_L\big(q^{\alpha_1},\ldots ,q^{\alpha_V},q\big)\nabla_L\big(q^{\alpha_1},\ldots ,q^{\alpha_V}\big)=1.
\end{equation*}

\item[$(b)$] When $F_L(x,q)$ is defined it satisfies the above property.
\end{enumerate}
\end{conj}

Indeed, in the case of a knot $K$ endowed with zero framing and colored by $\alpha$, Conjecture~\ref{conj:FADO2} becomes:
\begin{equation*}
\lim_{t\to 1}\lim_{q\to \exp(\frac{2\pi {\rm i}}{r})}\frac{F^t_K(q^{\alpha},q)}{F^t_K(q^{r\alpha};q^r)}=\N_r(K_\alpha)
\end{equation*}
so that using Conjecture~\ref{conj:FADO3} we get
\begin{equation*}
\lim_{t\to 1} \lim_{q\to \exp(\frac{2\pi {\rm i}}{r})} F^t_K(q^\alpha,q)\nabla_K.
\end{equation*}
There are, however, links with vanishing Alexander polynomial, e.g.,
\[
{\bf 9_{27}},\ {\bf 10_{32}},\ {\bf 10_{36}},\ {\bf 10_{107}},\ {\bf 11_{244}},\ {\bf 11_{247}},\
{\bf 11_{334}},\ {\bf 11_{381}},\ {\bf 11_{396}},\ {\bf 11_{404}},\ {\bf 11_{406}},\ \ldots\,.
\]
We then set the following definition:

\begin{Definition}[good links]
A link $L$ is very good if Conjectures~\ref{conj:FADO2} and~\ref{conj:FADO3} are verified for $L$ and is good if only Conjecture~\ref{conj:FADO2} is satisfied for it.
\end{Definition}

We remark that the links with zero multivariable Alexander polynomial cannot be very good. Also, Conjecture~\ref{conj:FADO2} does not completely fix $F_L(x,q)$. In particular, one can always multiply it by a function in $q$ that has the same radial limit when $q\rightarrow {\rm e}^{\frac{2\pi {\rm i}}{\ell}}$ for all positive integral $\ell$ (including $\ell=1$).

When $L$ is a knot, a version of Conjecture~\ref{conj:FADO2} was proven by S.~Willets~\cite{willets} in which $F_L$ belongs to a suitable ring obtained as completion of the Laurent polynomials in two variables.

\subsubsection{Examples of good links}\label{sub:goodlinks}
Let $L\subset S^3$ be a plumbing link, with the notation introduced in Example~\ref{ex:plumbedlink}.
In~\cite{Gukov:2020frk}, the following formula was provided for $F_L$:
\begin{equation*}
F_L \big(\big\{x_I^2\big\}_I,q\big) =\prod_I \big(x_I-x_I^{-1}\big)^{1-\deg(I)}.
\end{equation*}
If we use this formula for $F_L$, then it is clear that
\begin{equation*}
\lim_{q\to \exp(2{\rm i}\pi/r)} \frac{F_L\big(\big\{x_I^2\big\}_I,q\big)}{F_L\big(\big\{x_I^{2r}\big\}_I,q^r\big)}=\prod_I \frac{\big(x_I-x_I^{-1}\big)^{1-\deg(I)}}{\big(x_I^r-x_I^{-r}\big)^{1-\deg(I)}}
=\xi^{-\frac{{\alpha}^TB{\alpha}}{2}+\frac{(r-1)^2}{2}\Tr(B)} \N_r(L_\alpha),
\end{equation*}
where we used Example~\ref{ex:plumbedlink} and we set $x_I=\xi^{\alpha_I}$, so that Conjecture~\ref{conj:FADO2} holds for these links.

In general, $F_L$ is to be considered as a formal power series and in the above case we do it as follows. Let $t$ be a regularization parameter and let $F^t_L$ be the power series development in $t$ of
\begin{align*}
F^t_L \big(\big\{x_I^2\big\}_I,q\big) = {}&\prod_{I\colon \deg(I)>1} \frac{1}{2}\big(\big(x_I/t-tx_I^{-1}\big)^{1-\deg(I)}+\big(tx_I-x_I^{-1}/t\big)^{1-\deg(I)}\big)
\\
&{}\times \prod_{I\colon \deg(I)=0} t\big(x_I-x_I^{-1}\big).
\end{align*}
For the factors with $\deg(I)>1$, we used the fact that the expansion of $(x_I-1/x_I)^{1-\deg{I}}$ at $x_I=0$ (respectively, $x_I=\infty$) contains only positive (respectively, negative) powers of $x_I$. Since~$F^t_L$ is a rational function in $t$ the limit $F_L=\lim_{t\to 1} F^t_L$ exists even though the power series $F^t_L$ might have radius of convergence in $t$ less than $1$. So that Conjecture~\ref{conj:FADO2} holds.

Similarly the Alexander--Conway, polynomial of $L$ can be computed directly in these cases as $\nabla_L=F_L^{-1}$ (see Example~\ref{ex:plumbedlink}) so that also Conjecture~\ref{conj:FADO3} holds.
The plumbing links are then very good links.

\subsubsection{Whitehead link}

Starting with the expression for $F_L (x_1,x_2,q)$ obtained by the $R$-matrix technique~\cite{Park:2020edg,Park2021}:
\begin{equation*}
F_L (x_1,x_2,q) =
\sum_{n \ge 0} \frac{(-1)^n q^{- \frac{n(n+1)}{2}} \big(q^{n+1}\big)_n \big(x_1^{1/2} - x_1^{-1/2}\big) \big(x_2^{1/2} - x_2^{-1/2}\big) }{\prod_{j=0}^n \big(x_1 + x_1^{-1} - q^j - q^{-j}\big) \big(x_2 + x_2^{-1} - q^j - q^{-j}\big)},
\end{equation*}
where $(a)_n:=\prod_{i=0}^{n-1}(1-aq^i)$, we get
\begin{equation*}
F_L (x_1,x_2,q) \big|_{q \to -1} =
\frac{1}{\big(x_1^{1/2} - x_1^{-1/2}\big) \big(x_2^{1/2} - x_2^{-1/2}\big)}.
\end{equation*}
According to the identification of the parameters in~\cite{Gukov:2020lqm}, we should compare this with $\operatorname{ADO}_p$ with $p=2$, evaluated at $x {\rm e}^{- 2 \pi {\rm i}/p}$. In the case of knots, $\operatorname{ADO}_p$ is also multiplied by $\frac{x^{1/2} - x^{-1/2}}{\Delta (x^p)}$.

Murakami~\cite{Murakami} works in conventions such that $q={\rm e}^{\pi {\rm i} /p}$ is the $2p$-th root of unity and his $\operatorname{ADO}_p$ with $p=2$ gives the Alexander polynomial
\begin{gather*}
\operatorname{ADO}_1 = 1,\qquad
\operatorname{ADO}_2 = \big(z_1-z_1^{-1}\big)\big(z_2-z_2^{-1}\big),
\end{gather*}
where we used $z_1 = q^{\lambda}$ and $z_2 = q^{\mu}$ in Murakami's notations~\cite{Murakami}. Comparing this with the multivariable Alexander polynomial
\begin{equation*}
\Delta = \big(x_1^{1/2}-x_1^{-1/2}\big)\big(x_2^{1/2}-x_2^{-1/2}\big),
\end{equation*}
we see that Murakami's $z_i = x_i^{1/2}$ and the relation between $F_L (\vec x,q)$ and $\operatorname{ADO}_p (\vec x)$ for links must be more complicated; it should convert powers of $x_i^{1/2} - x_i^{1/2}$ in the denominator to the powers of $x_i^{1/2} - x_i^{1/2}$ in the numerator.

Note, if as in case of knots we evaluate the ADO invariant at $x {\rm e}^{- 2 \pi {\rm i}/p}$, we would have $\big(x_1^{1/2}+x_1^{-1/2}\big)\big(x_2^{1/2}+x_2^{-1/2}\big)$ in the numerator. Then, dividing it by $\Delta (x_i^2)$ we would get $\big(x_1^{1/2}-x_1^{-1/2}\big)\big(x_2^{1/2}-x_2^{-1/2}\big)$ in the denominator.
For $p=3$, we get
\begin{align}
&F_L (x_1,x_2,q) \big|_{q \to {\rm e}^{2\pi {\rm i} /3}} \nonumber
\\
&\qquad{}=
\frac{x_1^2 x_2^2+x_1^2 x_2+x_1^2+x_1 x_2^2+x_1 x_2+{\rm i}\sqrt{3} x_1 x_2+x_1+x_2^2+x_2+1}{\big(x_1^{1/2} - x_1^{-1/2}\big) \big(x_2^{1/2} - x_2^{-1/2}\big)\big(1 + x_1 + x_1^2\big) \big(1 + x_2 + x_2^2\big)}.
\label{WhiteheadFp3}
\end{align}
Apart from the familiar $\Delta (x_i)$, in the denominator we have a factor of
\begin{equation*}
\big(x_1 + 1 + x_1^{-1}\big) \big(x_2 + 1 + x_2^{-1}\big) = \frac{\Delta (x_i^p)}{\prod_i \big(x_i^{1/2} - x_i^{-1/2}\big)},
\end{equation*}
so that the entire denominator is basically $\Delta (x_i^p)$.

On the other hand, the numerator of~\eqref{WhiteheadFp3} is precisely Murakami's ADO$_3$ evaluated at $x_i {\rm e}^{- 2 \pi {\rm i}/p}$, just as in the case of knots.
Therefore, we can write this relation as
\begin{equation}
F_L (x_1,x_2,q) \big|_{q \to {\rm e}^{2\pi {\rm i} /p}} =
\frac{\operatorname{ADO}_p (x_i {\rm e}^{- 2 \pi {\rm i}/p})}{\Delta (x_i^p)}.
\label{FLADOreln}
\end{equation}
For reasons mentioned above, this also checks out in the $p=2$ case.
For $p=4$ we get
\begin{align*}
F_L (x_1,x_2,q) \big|_{q \to {\rm e}^{2\pi {\rm i} /4}} =
\frac{x_1^{-3/2} x_2^{-3/2} (x_1+1) (x_2+1) \big(x_1^2 \big(x_2^2+1\big)+2 i x_1 x_2+x_2^2+1\big)}{\Delta \big(x_i^4\big)},
\end{align*}
again, in perfect agreement with~\eqref{FLADOreln}.

As a side remark, we also note that, in general, ADO polynomials have many coefficients that are algebraic numbers; the coefficients of $\operatorname{ADO}_p (x_i {\rm e}^{- 2 \pi {\rm i}/p})$ are also algebraic integers, but typically much simpler.

\subsection[$\hat{Z}$ for 3-manifolds]{$\boldsymbol{\hat{Z}}$ for 3-manifolds}\label{sec:Zhat-def-b1}

Let $M=S^3(L)$ where $L$ is a link with linking matrix $B$ and the set of components $\vert$. Let
\begin{equation}
F(x,q) :=
F_L\big(x^2,q\big) \prod_{I\in \vert}\big(x_I-x^{-1}_I\big)
= \sum_{\ell \in \Z^\vert}F_{\ell}\prod_{I\in \vert} x_I^{\ell_I}
\label{FL-expansion}
\end{equation}
be a somewhat differently normalized $\hat{Z}$-invariant of the link $L$ or, more precisely, the link complement $S^3 \setminus L$. For example, $F(x,q)=F_K\big(x^2,q\big)\big(x-x^{-1}\big)$ for a knot (complement), and $F(x,q)=\prod_{I\in\vert}\big(x_I-x_I^{-1}\big)^{2-\deg(I)}$ for a plumbing link (complement). Note, the two notations, $F (x,q)$ and $F_L (x,q)$, are very similar and we hope it will not cause a confusion. In fact, $F_L (x,q)$ is used in most of the paper, and in a few places where $F (x,q)$ is used we try to remind the reader about the relation between the two normalizations.

\begin{Remark}\label{rem:whichlinks}
By default here and below, we assume that $F_L$ is defined using the $R$-matrix at general $q$ (and that the link $L$ is of the type for which that definition works, see~\cite[Section~1.3]{Park2021}), but other approaches reviewed above can be also in principle considered.
\end{Remark}

By an ${\rm SL}(V,\Z)$ transform, we can bring the integral quadratic form $B$ to $B'\oplus \mathbf{0}_{b_1}$, where $\mathbf{0}_{b_1}$ is the trivial quadratic form on $\mathbb{Z}^{b_1}\subset \Z^{V}$~\cite{kyle1954branched}. Namely, there exists $U\in {\rm SL}(V,\Z)$ such that
\begin{equation*}
 UBU^{T}=
 \begin{pmatrix}
 B' & 0 \\
 0 & 0
 \end{pmatrix},
\end{equation*}
where the right-hand side shows the block decomposition corresponding to the partition $V=(V-b_1)+b_1$. With $\varepsilon=(1,1,\ldots,1)$ and $s$ being mod 2 vectors defined as before ($\varepsilon=(1,1,\ldots,1)$, $\sum_JB_{IJ}s_J=s_I\bmod 2$)
\begin{equation*}
\begin{pmatrix}
 s' \\ s''
 \end{pmatrix} :=\big(U^T\big)^{-1}s,\qquad
 \begin{pmatrix}
 \varepsilon' \\
 \varepsilon''
 \end{pmatrix} :=\big(U^T\big)^{-1}\varepsilon.
\end{equation*}
Then we can define $\hat{Z}$ by the following surgery formula:
\begin{equation}
 \hat{Z}_{\sigma(b'\oplus b'',s)}(M):=
 (-1)^{b_+}q^{\frac{3\sigma-\Tr B}{4}+\sum_{I}|{\varepsilon''}_I||b''_I|}\sum_{\substack{\ell'=2b'+B'(s'-\varepsilon')\\\bmod 2B'\Z^{V-b_1}}}
F_{U^{-1}
 \begin{psmallmatrix}
 \ell' \\
 2b''
 \end{psmallmatrix}}
 q^{-\frac{\ell'^TB'^{-1}\ell'}{4}}, \label{Zhat-b1-def}
\end{equation}
where $b'\in \mathrm{Coker}B'\cong \Tor H_1(M;\Z)$, $b''\in \Z^{b_1}$ and $\sigma$ is the canonical map defined in~\eqref{eq:definitionsigma}. Here we assume that the sum over $\ell'$ in the right-hand side is convergent in the space of formal power series $2^{-c}q^{\Delta}\Z\big[q^{-1},q\big]\big]$. This means that any given power of $q$ gets contributions only from a finite number of terms. The definition of $\hat{Z}$ via the surgery formula~\eqref{Zhat-b1-def} also relies on the conjectural invariance of the right-hand side under Kirby moves. It has been verified in the literature for the plumbing links~\cite{Gukov:2019mnk} and some other cases.

If $b_1=0$, then $B=B'$, and the above formula simplifies to
\begin{equation}
\hat{Z}_{\sigma(b,s)}(M)=(-1)^{b_+} q^{\frac{3\sigma- \Tr B}{4}}\sum_{\substack{\ell=2b+B(s-\varepsilon)\\\bmod 2B\Z^{V}}}F_{\ell}q^{-\frac{\ell^tB^{-1}\ell}{4}}.
\label{Zhat-surgery-RHS}
\end{equation}
If $b_1>0$, the result is independent of the choice of $U$ preserving both $b'$ and $b''$, and is conjecturally invariant under the Kirby moves only up to the following equivalence relation:
\begin{equation}
 1\sim q^{\text{LCM}(2,\operatorname{GCD}(b''))},
 \label{H1-splitting-equivalence}
\end{equation}
where $\operatorname{GCD}(b''):=\operatorname{GCD}\big(\{b_i''\}_{i=1}^{b_1}\big)$ (which is invariant under ${\rm SL}(b_1,\Z)$ transformations). Na\-me\-ly, the invariants should be considered as equivalence
\begin{equation}
 \hat{Z}_{\sigma(b'\oplus b'',s)} \in 2^{-c} q^{\Delta}\Z[[q]]/\big(1-q^{\text{LCM}(2,\operatorname{GCD}(b''))}\big)\Z[q].
 \label{non-torsion-zhat-space}
\end{equation}

It is easy to see that $\hat{Z}_{\sigma( \cdot ,s)}$ transform covariantly (as functions on $H_1(M;\Z)\cong \Tor H_1(M;\Z)\allowbreak\oplus \Z^{b_1}$) under the automorphisms preserving the splitting. They correspond to changes of the matrix $U$ of the following form:
\begin{equation*}
 U\rightsquigarrow
 \begin{pmatrix}
 \tilde{\gamma} & 0 \\
 0 & \nu
 \end{pmatrix} U,
\end{equation*}
where $\tilde{\gamma}\in {\rm SL}(V-b_1,\Z)$ and $\nu\in {\rm SL}(b_1,\Z)$. Then from the definition~\eqref{Zhat-b1-def}, it follows that
\begin{equation*}
 \hat{Z}_{\sigma(b'\oplus b'',s)}\rightsquigarrow \hat{Z}_{\sigma(\gamma^{-1}b'\oplus \nu^{-1}(b''),s)},
\end{equation*}
where $\gamma$ is the automorphism of $\Tor H_1(M;\Z)$ represented by $\tilde{\gamma}$. Namely,
\begin{equation*}
 \gamma (a'):=\tilde{\gamma}(a') \bmod B'\Z^{V-b_1},\qquad
 a'\in \Z^{V-b_1}.
\end{equation*}

However, there is an anomaly under the automorphisms changing the splitting. Modulo the automorphisms preserving the splitting, they are of the form
\begin{equation}
 \Tor H_1(M;\Z)\oplus \Z^{b_1}
 \stackrel{\begin{psmallmatrix}
 \text{id}_{\Tor H_1} & \mu \\
 0 & \text{id}_{\Z^{b_1}}
 \end{psmallmatrix}}{\longrightarrow}
 \Tor H_1(M;\Z)\oplus \Z^{b_1},
 \label{splitting-auto}
\end{equation}
where $\mu$ is a non-trivial homomorphism
\begin{equation}
 \mu\colon \ \Z^{b_1}\longrightarrow \Tor H_1(M;\Z).
 \label{free-to-tor}
\end{equation}
This automorphism is realized by the replacement
\begin{equation}
 U\rightsquigarrow
 \begin{pmatrix}
 \mathbf{1}_{V-b_1} & \tilde{\mu} \\
 0 & \mathbf{1}_{b_1}
 \end{pmatrix} U,
\end{equation}
where $\tilde{\mu}$ is a $(V-b_1)\times b_1$ matrix. Under the identification $\Tor H_1(M;\Z)=\operatorname{Coker} B'$, we have
\begin{equation*}
 \mu (a''):=\tilde{\mu}(a'') \bmod B'\Z^{V-b_1},\qquad
 a''\in \Z^{b_1}.
\end{equation*}
The corresponding change of $\hat{Z}$ depends only on $\mu$, and not the representative $\tilde{\mu}$, if one takes into account the equivalence relation~\eqref{H1-splitting-equivalence}.

However, it is \textit{not} true that simply $\hat{Z}_\sigma(b'\oplus b'',s)\rightsquigarrow \hat{Z}_\sigma((b'-\mu(b''))\oplus b'',s)$. The covariant transformation is corrected by an anomalous factor
\begin{equation}
 \hat{Z}_{\sigma(b'\oplus b'',s)}\rightsquigarrow q^{\mathcal{E}(b',b'')}\hat{Z}_{\sigma((b'-\mu(b''))\oplus b'',s)},
 \label{splitting-anomaly}
\end{equation}
where
\begin{equation*}
 \mathcal{E}(b',b'')=\operatorname{GCD}(b'')\cdot
 \begin{cases}
 \operatorname{GCD}(b'')q_s\big(\mu\hat{b}''\big)-2\lk\big(b',\mu\hat{b}''\big)
 ,& \operatorname{GCD}(b'')\text{ is odd}, \\
 \operatorname{GCD}(b'')\lk\big(\mu\hat{b}'',\mu\hat{b}''\big)-2\lk\big(b',\mu\hat{b}''\big)
 ,& \operatorname{GCD}(b'')\text{ is even}
 \end{cases}
\end{equation*}
and
\begin{equation*}
 \hat{b}'':=\frac{b''}{\operatorname{GCD}(b'')} \in \Z^{b_1}.
\end{equation*}
The factor $q^{\mathcal{E}(b',b'')}$ is well defined modulo the equivalence relation~\eqref{H1-splitting-equivalence}.

\subsection{Physics of non-torsion fluxes}

For a given spin$^c$ structure $b$, let $\CB_b$ be the corresponding boundary condition of the 6d fivebrane theory on $M \times D^2 \times_q S^1$. This defines the boundary condition in 4d gauge theory on $M \times \R_+$ (obtained by projecting $D^2 \times_q S^1 \to \R_+$) as well as the boundary condition in 3d theory $T[M]$ on $D^2 \times_q S^1$ (obtained by reducing on $M$):
\begin{equation}
\begin{tikzcd}
& \begin{array}{c} \text{6d $(0,2)$ theory on} \\
M \times D^2 \times_q S^1
\end{array} \ar[ld] \ar[rd, "\text{on $M$}"] & \\
\begin{array}{c} \text{4d super-Yang--Mills} \\
\text{on $M \times \R_+$}
\end{array} & &
\begin{array}{c} \text{3d theory}~T[M] \\
\text{on $D^2 \times_q S^1$.}
\end{array}
\end{tikzcd}
\label{dualitycascade}
\end{equation}
In each of these descriptions, including the original 6d system viewed from the enumerative perspective of Calabi--Yau 3-fold and M2-branes, the ambiguity is naturally associated with the partition function on $M \times T^2 \times I$, where both $M \times T^2$ boundaries are colored by $\CB_b$.

The effect of $\CB_b$ is two-fold: $(i)$ first, it effectively abelianizes the theory, and $(ii)$ it also puts it in a non-trivial background, so that even abelian fluxes with one ``leg'' along $M$ and another ``leg'' along $I$ now carry a non-trivial $q$-degree. Note, this latter effect is absent when $b$ is torsion. The sum over such fluxes gives a $q$-series unbounded in both directions
\begin{equation*}
\sum_{m \in H^1 (M)} q^{m \cdot b''}
\end{equation*}
that we would like to remove or factor out, in order to make the partition function on $M \times D^2 \times_q S^1$ well-defined. It would be interesting to explore various ways to do this. Relegating a more systematic study of this question for future work, here we merely sidestep the issue by imposing the identification $q^{\operatorname{GCD}(b'')} \sim 1$ that leads to~\eqref{H1-splitting-equivalence} and~\eqref{non-torsion-zhat-space}.

Let us make a few comments on this interesting phenomenon by examining it from various perspectives. Reduction to 4d gauge theory~\eqref{dualitycascade} yields a system of Kapustin--Witten PDE's on $M \times \R_+$, with a Nahm pole boundary condition at $y=0$, where $y$ is a coordinate along $\R_+$~\cite{Witten:2011zz}. The boundary condition at $y = \infty$ breaks the gauge group $G$ to a Levi subgroup $\mathbb{L} \subseteq G$, which in applications to $\hat Z$-invariants is a maximal torus of $G$. In particular, $\mathbb{L} = \operatorname{U}(1)$ for $G = \operatorname{SU}(2)$. Other choices of $\mathbb{L}$ are also interesting, and lead to a generalization of $\hat Z$-invariants labeled by complex coadjoint orbits of $G_{\C}$ or, equivalently, by $\rho\colon \mathfrak{sl} (2) \to \mathfrak{g}$~\cite{Gukov:2006jk}.

When $\mathbb{L} = \mathbb{T}$, there are infinitely many different topological sectors labeled by monopole numbers $b_i'' \in \Lambda_{\text{cochar}} = \operatorname{Hom} (\operatorname{U}(1), \mathbb{T})$ or, more precisely, by spin$^c$ structures. In order to keep track of dependence on $b$, one can introduce a topological term $\exp (2\pi {\rm i} \eta b)$ in the action of 4d gauge theory. Then, using Pontryagin duality,
\begin{equation*}
x = {\rm e}^{2\pi {\rm i} \eta} \in \operatorname{Hom} (H_1 (M), \mathbb{T})
\end{equation*}
can be identified with the variable by the same name in $F_K (x,q)$ and in the integrand of $\hat Z$-invariants. To summarize, on a closed 3-manifold with $b_1 > 0$ the boundary condition $\CB_b$ breaks the gauge group $G$ to $\mathbb{L} = \mathbb{T}$ and creates a ``flux'' $b$. When this flux is non-torsion, the solutions to Kapustin--Witten PDE's on $M \times \R_+$ can not approach a constant field configuration at $y = + \infty$. At best one can require solutions to approach a field configuration periodic in the $y$-direction, which leads us to conclude that the anomaly in question is controlled by the moduli space of solutions on $M \times S^1$ with gauge group $\mathbb{L}$. Below we give another interpretation of this claim from the perspective of 3d-3d correspondence.

From the point of view of 3d $\CN=2$ theory $T[M]$, the boundary condition $\CB_b$ labels the background momentum / charge sectors of the 2d boundary theory, cf.~\cite{Dedushenko:2017tdw}. In the partition function on $D^2 \times_q S^1$, the parameter $q$ keeps track of the spin with respect to the rotation symmetry of $D^2$. It is defined up to spins of BPS states in 2d boundary theory or, equivalently, 3d theory $T[M, \mathbb{T}]$ on a slab $T^2 \times I$ with boundary conditions $\CB_b$ on both sides. This theory is a~close cousin of $T\big[M \times S^1, \mathbb{T}\big]$ in the background of spin$^c$ structure $b$ on $M$. Indeed, both theories exhibit a qualitative change in behavior depending on whether $b_1 = 0$ or $b_1 > 0$. In the latter case, the BPS states come in towers infinite in both direction, with spins in each tower being multiples of $b''$, which leads again to the ambiguity~\eqref{H1-splitting-equivalence}.

Similarly, from the curve counting perspective on $T^* M$, when $b_1 > 0$ in addition to open BPS states one also has a non-trivial ``closed sector.''

\section{Families of examples}\label{sec:examples}

In this section, we illustrate the proposed relation with a number of instructive examples for which explicit formulas for $\hat{Z}$ can be provided (recall Remark~\ref{rem:whichlinks}).

\subsection{Plumbed 3-manifolds}
\label{sec:plumbed}

Here we consider the case when $M$ is a rational homology sphere given by a weakly negative definite plumbing graph $\Gamma$~\cite{Gukov:2019mnk, Gukov:2017kmk}. We can then assume that
\begin{equation*}
 \omega \in H^1(M;\Q/2\Z) \setminus H^1(M;\Z/2\Z).
\end{equation*}
With the notation of Section~\ref{sec:combinatorial}, formula~\eqref{eq:NrM} becomes
\begin{equation*}
 \N_r(M,\omega)=
 \frac{1}{\Delta_+^{b_+}\Delta_-^{b_-}}
 \sum_{k\in H_r^\vert}
 \prod_{I\in \vert} d(\alpha_{k_I})^{2-\deg(I)} T(\alpha_{k_I})^{B_{II}}
 \prod_{(I,J)\in \text{Edges}} S(\alpha_{k_I},\alpha_{k_J}),
\end{equation*}
where
\begin{equation*}
 \alpha_{k_I}:=\mu_I+k_I
\end{equation*}
and $b_\pm$ are the number of positive/negative eigenvalues of $B$. In this section, we will be somewhat cavalier with taking the limits and about convergence of infinite series. Such technical details will be properly addressed in Section~\ref{sec:gauss-vs-laplace} and Appendix~\ref{app:commutlimits}.

It is instructive to separate 3 factors:
\begin{gather*}
 \N_r(M,\omega)=\mathcal{A}\cdot \mathcal{B}\cdot \mathcal{C},
\\
 \mathcal{A}=r^{-V/2} \xi^{\frac{3\sigma-\Tr B}{2}}\cdot
 \begin{cases}
 {\rm e}^{\frac{\pi {\rm i}(\sigma+\Tr B)}{2}}, & r=1\bmod 4, \\
 2^{-V/2} {\rm e}^{-\frac{\pi {\rm i}\sigma}{4}}, & r=2\bmod 4, \\
 {\rm e}^{-\frac{\pi {\rm i}}{2}\Tr B} (-1)^\sigma, & r=3\bmod 4,
 \end{cases}
\end{gather*}
where $\sigma$ is the signature of $B$,
\begin{gather}
 \mathcal{B}=F\big(\big\{{\rm e}^{\pi {\rm i}\mu_I}\big\}_{I\in \vert}\big)^{-1},
 \label{factor-B}
\\
 \mathcal{C}=\sum_{k\in H_r^\vert} F\big(\big\{\xi^{\mu_I+k_I}\big\}_{I\in \vert}\big)\cdot
 \xi^{\frac{1}{2}(\mu+k)^TB(\mu+k)},
 \label{factor-C}
\end{gather}
and
\begin{equation}
 F(x):=\prod_{I\in \vert }(x_I-1/x_{I})^{2-\deg(I)}=\sum_{\ell \in\Z^\vert }F_\ell \prod_I x_I^{\ell_I}
 \label{F-function}
\end{equation}
is a slightly different normalization of the invariant $F_L(x,q)$ for the case of the plumbing link $L$. Consider a contribution of a monomial $\prod_I x_I^{\ell_I}$ from $F(x)$ into~\eqref{factor-C}:
\begin{equation}
 \mathcal{C}_\ell:=\sum_{k\in H_r^\vert} \xi^{\ell^T(\mu+k)}\cdot
 \xi^{\frac{1}{2}(\mu+k)^TB(\mu+k)}=
 \sum_{n\in \Z^\vert/r\Z^\vert}
 {\rm e}^{\frac{\pi {\rm i} }{2r}(\tilde\mu+2n)^TB(\tilde\mu+2n)+\frac{\pi {\rm i}}{r}\ell^T(\tilde\mu+2n)},
 \label{factor-C-monomialnew}
\end{equation}
where
\begin{equation*}
 \tilde\mu:=\mu+(r-1)\varepsilon
\end{equation*}
and $\varepsilon$ is the vector $(1,1,\ldots, 1,1)$.
We can now use the following version of Gauss reciprocity formula~\cite{deloup2007reciprocity,jeffrey1992chern}:
\begin{gather}
\sum_{n \in \Z^\vert/r\Z^\vert}
\exp\bigg(\frac{2\pi {\rm i}}{r} n^TBn+\frac{2\pi {\rm i}}{r} p^Tn\bigg)\nonumber
\\ \qquad
{}=\frac{{\rm e}^{\frac{\pi {\rm i}\sigma}{4}} (r/2)^{V/2}}{|\det{B}|^{1/2}}
\sum_{\tilde{a} \in \Z^\vert/2B\Z^\vert}
\exp\biggl(-\frac{\pi {\rm i} r}{2}\bigg(\tilde{a}+\frac{p}{r}\bigg)^TB^{-1}\bigg(\tilde{a}+\frac{p}{r}\bigg)\biggr).
\label{reciprocity-alt}
\end{gather}
Applying it to~\eqref{factor-C-monomialnew}, we have
\begin{equation}
 \mathcal{C}_\ell= \xi^{-\frac{\ell^TB^{-1}\ell}{2}}
 \frac{{\rm e}^{\frac{\pi {\rm i}\sigma}{4}}(r/2)^{V/2}}{|\det B|^{1/2}}
 \underbrace{
 \sum_{\tilde{a}\in \Z^\vert /2B\Z^{\vert}}
 {\rm e}^{-\frac{\pi {\rm i} r}{2} \tilde{a}^T B^{-1}\tilde{a}-\pi {\rm i} \tilde{a}^T B^{-1}(\ell+B\tilde\mu)}
 }_{=:\mathcal{C}_\ell'}.
 \label{C-ell-expr}
\end{equation}
Let us make the change of variables $\tilde{a}=BA+a$, $A\in \Z^\vert/2\Z^\vert$, $a\in\Z^\vert/B\Z^\vert$:
\begin{align}
 \mathcal{C}_\ell'&= \sum_{{a}\in \Z^\vert /B\Z^{\vert}}\sum_{A\in \Z^\vert/2\Z^\vert}
 {\rm e}^{-\frac{\pi {\rm i} r}{2} {a}^T B^{-1}{a}-\pi {\rm i}r A^Ta-\frac{\pi {\rm i}r}{2} A^TBA
 -\pi {\rm i} a^TB^{-1}(\ell+B\tilde{\mu})-\pi {\rm i} A^T(\ell+B\tilde\mu)} \nonumber
 \\
&=\sum_{{a}\in \Z^\vert /B\Z^{\vert}}\sum_{A\in \Z^\vert/2\Z^\vert}
 \exp\biggl\{-\frac{\pi {\rm i} r}{2} {a}^T B^{-1}{a}-\pi {\rm i}r A^Ta-\frac{\pi {\rm i}r}{2} A^TBA
\nonumber
 \\
&\hspace{45mm} -2\pi {\rm i} a^TB^{-1}b-\pi {\rm i} a^T(s+\mu-r\varepsilon)
 -\pi {\rm i} A^TB(s-r\varepsilon)\biggr\},
 \label{C-prime-sum}
\end{align}
where in the last line we used the fact that~\eqref{F-function} only contains powers $\prod_{I}x_I^{\ell_I}$ satisfying $\ell_I=\deg(I)\bmod 2$, and therefore one can introduce $b\in \Z^\vert/B\Z^\vert$, $s\in \Z^\vert/2\Z^\vert$, $\sum_{J}B_{IJ}s_J=B_{II}\bmod 2$, such that
\begin{equation}
 \ell=2b+B(s-\varepsilon)\bmod 2B\Z^\vert.
 \label{ell-H1-spin}
\end{equation}
We also used the property~\eqref{mu-condition}. At this point we will need to consider the cases with different~$r\bmod 4$ values separately.

\subsection[Level $r=2\bmod 4$]{Level $\boldsymbol{r=2\bmod 4}$}

Using the fact that $r$ is even, while $r/2$ is an odd integer, and condition on $s$ the sum~\eqref{C-prime-sum} simplifies to
\begin{equation*}
 \mathcal{C}_\ell'= 2^V\sum_{{a}\in \Z^\vert /B\Z^{\vert}}
 \exp\biggl\{-\frac{\pi {\rm i} r}{2} {a}^T B^{-1}{a}
 -2\pi {\rm i} a^TB^{-1}b-\pi {\rm i} a^T(s+\mu)\biggr\}.
\end{equation*}
Combining everything together we then have
\begin{align*}
 \N_r(M,\omega)={}& \frac{F\big(\big\{{\rm e}^{\pi {\rm i}\mu_I}\big\}_{I\in \vert}\big)^{-1}} {|\det B|^{1/2}} \xi^{\frac{3\sigma-\Tr B}{2}}
 \\
&{}\times \sum_{\ell \in \Z^\vert }\sum_{{a}\in \Z^\vert /B\Z^{\vert}}
 F_\ell \xi^{-\frac{\ell^TB^{-1}\ell}{2}}
 {\rm e}^{-\frac{\pi {\rm i} r}{2} {a}^T B^{-1}{a}
 -2\pi {\rm i} a^TB^{-1}b-\pi {\rm i} a^T(s+\mu)}.
\end{align*}
Using the following expressions for $\hat{Z}$ and Reidemeister torsion\footnote{In principle, the torsion $\mathcal{T}(M,\alpha)$, defined for $\alpha \in H^1(M;\Q/\Z)\stackrel{\lk}{\cong} H_1(M;\Z)$, has sign ambiguity, if no additional structures on $M$ are introduced. One can fix the sign for example by introducing a spin structure $s\in \Spin(M)$ on $M$, cf.~\cite{Mikhaylov:2015nsa}. The change of the spin structure $s\rightarrow s+c$, $c\in H^2(M;\Z_2)$ then changes the sign by $(-1)^{c(\tilde{\alpha})}$ where $\tilde{\alpha}\in H_1(M;\Z)$ is dual to $\alpha$. However, since in our case $\alpha =\omega\bmod 1$, $\tilde{\alpha}$ is even, and the dependence on spin structure drops out.} $\mathcal{T}$ for plumbed manifolds (see formula~\eqref{torsion-plumbed-spin} in Appendix~\ref{app:torsion})
\begin{gather}
\hat{Z}_{a}(M)=(-1)^{b_+} q^{\frac{3\sigma-\Tr B}{4}}
\sum_{\ell =a\bmod 2B\Z^\vert}F_\ell q^{-\frac{\ell^T B^{-1}\ell}{4}},\qquad
 a_I=\deg(I)\bmod 2,\nonumber
\\
 \mathcal{T}(M,[\omega])=(-1)^{b_+}\prod_{I\in \vert}\big({\rm e}^{\pi {\rm i}\mu_I}-{\rm e}^{-\pi {\rm i} \mu_I}\big)^{\deg(I)-2},\qquad
 [\omega]:=\omega \bmod H^1(M;\Z/2\Z)
\label{eq:zhatplumbing}
\end{gather}
and the identifications~\eqref{plumbed-H_1}--\eqref{plumbed-spinc}, we can conjecture the following general relation for a rational homology sphere $M$ and $r=2\bmod 4$:
\begin{equation}
 \N_{r}(M,\omega)=
 \frac{\mathcal{T}(M,[\omega])}{\sqrt{|H_1(M;\Z)|}}
 \sum_{a,b\in H_1(M;\Z)}
 {\rm e}^{-\frac{\pi {\rm i}r}{2} q_s(a)-2\pi {\rm i} \lk(a,b)-\pi {\rm i} \omega(a)}
 \hat{Z}_{\sigma(b,s)}\Big|_{q\rightarrow {\rm e}^{\frac{2\pi {\rm i}}{r}}},
\label{CGP-Zhat-2mod4}
\end{equation}
where $\sigma$ is the canonical map
\begin{equation}
 \sigma\colon \ H_1(M;\Z)\times \operatorname{Spin}(M) \longrightarrow \operatorname{Spin}^c(M)
 \label{spin-spinc-map}
\end{equation}
producing a spin$^c$ structure on $M$ from a spin structure $c$ and $\tilde{b}\in H_1(M;\Z)$. It is induced by the map $ B\operatorname{Spin}\times B\operatorname{U}(1) \rightarrow B\operatorname{Spin}^c$ between the corresponding classifying spaces, combined with the isomorphisms $B\operatorname{U}(1)\cong B^2\Z$, $H_1(M;\Z)\cong H^2(M;\Z)$. In~\eqref{CGP-Zhat-2mod4}, we have introduced an auxiliary spin structure $s\in \Spin(M)$ (see also formula~\eqref{ell-H1-spin}). The result is independent of it due to~\eqref{qref-spin-change}, so that the simultaneous change of $b\in H_1(M;\Z)$ and $s\in \Spin(M)$ leaving $\sigma(b,s)$ invariant also leaves invariant the exponent in the sum in~\eqref{CGP-Zhat-2mod4}.

\subsection[Level $r=1\bmod 4$]{Level $\boldsymbol{r=1\bmod 4}$}

The sum~\eqref{C-prime-sum} reads
\begin{align*}
 \mathcal{C}_\ell'=\sum_{a\in \Z^\vert/B\Z^\vert}
 \sum_{A\in \Z^\vert/2\Z^\vert}
 \exp\biggl\{-&\frac{\pi {\rm i} r}{2} a^TB^{-1}a
 -2\pi {\rm i} a^TB^{-1}b-\pi {\rm i} a^T\mu
\\
 -&\pi {\rm i} a^T(s-\varepsilon)-\frac{\pi {\rm i}}{2} A^TBA+\pi {\rm i} A^T(a+B(s-\varepsilon))
 \biggr\}.
\end{align*}
Applying a version of the Gauss reciprocity formula to the sum over $A$, we can rewrite it as follows:
\begin{equation*}
\begin{split}
 \mathcal{C}_\ell'=
 \frac{{\rm e}^{-\frac{\pi {\rm i}\sigma}{4}}2^{V/2}}{|\det B|^{1/2}}
\sum_{a,f\in \Z^\vert/B\Z^\vert}
\exp\biggl\{{-} &\frac{\pi {\rm i} (r-1)}{2} a^T B^{-1}a
 -2\pi {\rm i} a^TB^{-1}b-\pi {\rm i} a^T\mu\nonumber
\\
 +&2\pi {\rm i} f^TB^{-1}f
 +2\pi {\rm i} f^TB^{-1}a
 +\frac{\pi {\rm i}}{2} (s-\varepsilon)^TB(s-\varepsilon)
\biggr\}.
\end{split}
\end{equation*}
Taking into account that
\begin{gather*}
 \frac{1}{4}\varepsilon^T B \varepsilon = \frac{V-1}{2}+\frac{1}{4} \Tr B \bmod 1,
\qquad
 \frac{1}{2}\varepsilon^T B s = \frac{1}{2} \Tr B \bmod 1,
\end{gather*}
and combining everything together we then have
\begin{align*}
 \N_r(M,\omega)={}&
 \frac{F(\{{\rm e}^{\pi {\rm i}\mu_I}\}_{I\in \vert})^{-1}}{|\det B|}
 {\rm e}^{\frac{\pi {\rm i}}{2}(2-\sigma+s^TBs )}
 \xi^{\frac{3\sigma-\Tr B}{2}} \sum_{\ell \in \Z^\vert }\sum_{a,f\in \Z^\vert /B\Z^{\vert}}
 F_\ell \xi^{-\frac{\ell^TB^{-1}\ell}{2}}
 \\
&{}\times
 {\rm e}^{-\frac{\pi {\rm i} (r-1)}{2} a^TB^{-1}a
 -2\pi {\rm i} a^TB^{-1}b-\pi {\rm i} a^T\mu
 +2\pi {\rm i} f^TB^{-1}f
 +2\pi {\rm i} f^TB^{-1}a }.
\end{align*}
As in the case $r=2\bmod 4$, we can then conjecture the following general relation for a rational homology sphere $M$ and $r=1\bmod 4$:
\begin{align}
 \N_{r}(M,\omega)={}&
 \frac{-{\rm e}^{-\frac{\pi {\rm i}}{2}\mu(M,s)} \mathcal{T}(M,[\omega])}{{|H_1(M;\Z)|}} \nonumber
 \\
 &\times{} \sum_{a,b,f\in H_1(M;\Z)} {\rm e}^{2\pi {\rm i}\left(-\frac{r-1}{4}\lk(a,a)
 +\lk(a,f-b)-\frac{1}{2}\omega(a) +\lk(f,f)\right)}
 \hat{Z}_{\sigma(b,s)}\Big|_{q\rightarrow {\rm e}^{\frac{2\pi {\rm i}}{r}}},
 \label{CGP-Zhat-1mod4}
\end{align}
where we have used the surgery formula~\eqref{Rokhlin-mod4} for the $\pmod 4$ reduction of the Rokhlin invariant $\mu(M,s)$.

It is interesting to remark that unlike in the case $r=2\pmod 4$ here the relation between~$\N_r$ and~$\hat{Z}$ is based on a \emph{triple} summation (instead of double).

\subsection[Level $r=3\bmod 4$]{Level $\boldsymbol{r=3\bmod 4}$}

This case is analogous to the case $r=1\bmod 4$ considered above. When $r=3\bmod 4$, the sum~\eqref{C-prime-sum} reads
\begin{align*}
 \mathcal{C}_\ell'=\sum_{a\in \Z^\vert/B\Z^\vert}
 \sum_{A\in \Z^\vert/2\Z^\vert}
 \exp\biggl\{-&\frac{\pi {\rm i} r}{2} a^TB^{-1}a
 -2\pi {\rm i} a^TB^{-1}b-\pi {\rm i} a^T\mu
\\
 -&\pi {\rm i} a^T(s-\varepsilon)+\frac{\pi {\rm i}}{2} A^TBA+\pi {\rm i} A^T(a+B(s-\varepsilon))
 \biggr\}.
\end{align*}
Applying again the Gauss reciprocity formula to the sum over $A$ we have
\begin{align*}
 \mathcal{C}_\ell'=
 \frac{{\rm e}^{\frac{\pi {\rm i}\sigma}{4}}2^{V/2}}{|\det B|^{1/2}}
 \sum_{a,f\in \Z^\vert/B\Z^\vert}
 \exp\biggl\{-&\frac{\pi {\rm i} (r+1)}{2} a^TB^{-1}a
 -2\pi {\rm i} a^TB^{-1}b-\pi {\rm i} a^T\mu
 \\
 -&2\pi {\rm i} f^TB^{-1}f -2\pi {\rm i} f^TB^{-1}a
 -\frac{\pi {\rm i}}{2} (s-\varepsilon)^TB(s-\varepsilon) \biggr\}.
\end{align*}
Combining everything together we then have
\begin{align*}
 \N_r(M,\omega)={}&
 \frac{F(\{{\rm e}^{\pi {\rm i}\mu_I}\}_{I\in \vert})^{-1}}{|\det B|}
 {\rm e}^{\frac{\pi {\rm i}}{2}(2+\sigma-s^TB^{-1}s )}
 \xi^{\frac{3\sigma-\Tr B}{2}} \sum_{\ell \in \Z^\vert }\sum_{a,f\in \Z^\vert /B\Z^{\vert}}
 F_\ell \xi^{-\frac{\ell^TB^{-1}\ell}{2}}
 \\
 &\times
 {\rm e}^{-\frac{\pi {\rm i} (r-1)}{2} a^TB^{-1}a
 -2\pi {\rm i} a^TB^{-1}b-\pi {\rm i} a^T\mu
 +2\pi {\rm i} f^TB^{-1}f
 +2\pi {\rm i} f^TB^{-1}a }.
\end{align*}
We can then conjecture the following general relation for a rational homology sphere $M$ and $r=3\bmod 4$:
\begin{equation}
\begin{aligned}[b]
\N_{r}(M,\omega)= &
 \frac{-{\rm e}^{\frac{\pi {\rm i}}{2}\mu(M,s)} \mathcal{T}(M,[\omega])}{{|H_1(M;\Z)|}}
 \\
 &\times {} \sum_{a,b,f\in H_1(M;\Z)}
 {\rm e}^{2\pi {\rm i}\left(-\frac{r+1}{4}\lk(a,a)
 -\lk(a,f+b)-\frac{1}{2}\omega(a) -\lk(f,f) \right)}
 \hat{Z}_{\sigma(b,s)}\Big|_{q\rightarrow {\rm e}^{\frac{2\pi {\rm i}}{r}}}.
\end{aligned}
\label{CGP-Zhat-3mod4}
\end{equation}

\subsection[Generalization to $b_1>0$]{Generalization to $\boldsymbol{b_1>0}$}
\label{sec:b1-positive}

Let $M$ be, as before, obtained by a surgery on a link with the linking matrix $B$. However, now we will allow $B$ to be degenerate. By an ${\rm SL}(V,\Z)$ transform we can bring the quadratic form $B$ to $B'\oplus \mathbf{0}_{b_1}$ where $\mathbf{0}_{b_1}$ is the trivial quadratic form on $\mathbb{Z}^{b_1}\subset \Z^{V}$~\cite{kyle1954branched}. The expression~\eqref{C-ell-expr} then will be modified to
\begin{equation*}
\begin{aligned}[b]
 \mathcal{C}_\ell={}&
 \xi^{-\frac{{\ell'}^T{B'}^{-1}{\ell'}}{2}}
 \frac{{\rm e}^{\frac{\pi {\rm i}\sigma}{4}}(r/2)^{(V-b_1)/2}}{|\det B'|^{1/2}}
 \\
 &\times{}
 \sum_{\tilde{a}\in \Z^{V-b_1} /2B'\Z^{V-b_1}}
 {\rm e}^{-\frac{\pi {\rm i} r}{2} \tilde{a}^T (B')^{-1}\tilde{a}-\pi {\rm i} \tilde{a}^T (B')^{-1}(\ell'+B'\tilde{\mu}')}
 r^{b_1} {\rm e}^{\frac{\pi {\rm i}{\ell''}^T \tilde{\mu}''}{r}}
 \delta_{\ell''=0\bmod r},
 \end{aligned}
\end{equation*}
where $\ell=\ell'\oplus \ell''$ and $\tilde{\mu}=\tilde{\mu}'\oplus \tilde{\mu}''$ according to the splitting of $B$ above.

Consider first the case of $r=2\bmod 4$. We have
\begin{equation*}
\mathcal{A}=\xi^{\frac{3\sigma-\Tr B}{2}}
r^{-V/2+b_1/2} 2^{b_1/2-V/2}
{\rm e}^{-\frac{\pi {\rm i}\sigma}{4}}.
\end{equation*}
The relation~\eqref{CGP-Zhat-2mod4} generalizes to
\begin{gather}
 \N_{r}(M,\omega)\nonumber
 \\
 {}=\frac{r^{b_1}\mathcal{T}(M,[\omega])}{\sqrt{| \Tor H_1(M;\Z)|}}\sum_{\substack{a',b'\in \Tor H_1(M;\Z)\\ m \in \Z^{b_1}}}\!\!\!\!\!
 {\rm e}^{-\frac{\pi {\rm i}r}{2}q_s(a')-2\pi {\rm i} \lk(a',b')-\pi {\rm i}\omega(a')+\pi {\rm i} \omega''(m)}
 \hat{Z}_{\sigma(b'\oplus rm/2,s)}\Big|_{q\rightarrow {\rm e}^{\frac{2\pi {\rm i}}{r}}}\nonumber
 \\
{} =\frac{r^{b_1}\mathcal{T}(M,[\omega])}{\sqrt{| \Tor H_1(M;\Z)|}}
 \int_{H^1(M;\R/\Z)}\!\!\!\!\mu(\alpha)\!\!\!\!\sum_{b\in H_1(M;\Z)}
 {\rm e}^{-\frac{\pi {\rm i}r}{2}q_s(\alpha')-2\pi {\rm i} \lk(\alpha',b')-\pi {\rm i}\omega(\alpha')+2\pi {\rm i} \alpha''(b'')}\nonumber
 \\[-2mm]
 \hspace{70mm}{}\times \delta(r\alpha''-\omega''/2)
 \hat{Z}_{\sigma(b,s)}\Big|_{q\rightarrow {\rm e}^{\frac{2\pi {\rm i}}{r}}},
 \label{CGP-Zhat-2mod4-mod}
\end{gather}
where we chose explicit splittings $\omega =\omega'\oplus \omega''\in H^1(M;\C/2\Z)\cong \Tor{H_1(M;\Z)} \oplus (\C/2\Z)^{b_1}$ and $b=b'\oplus b''\in H_1(M;\Z)\cong \Tor{H_1(M;\Z)} \oplus \Z^{b_1}$ according to the splitting of the linking matrix $B$ above.
It is straightforward to see that the right-hand side is independent of the choice of representative of the equivalence class~\eqref{H1-splitting-equivalence}. The coefficients in the relation~\eqref{CGP-Zhat-2mod4-mod} are \textit{not} invariant under the automorphisms~\eqref{splitting-auto}. However, this compensated by the non-covariance of~$\hat{Z}$. Taking into account~\eqref{splitting-anomaly}, one can show that the total sum in~\eqref{CGP-Zhat-2mod4-mod} transforms covariantly (i.e., as a function on $H^1(M;\C/2\Z)\cong \Hom(\Tor H_1(M;\Z)\oplus \Z^{b_1},\C/2\Z)$).
Namely, considering that
\begin{equation*}
 q^{\mathcal{E}(b',b'')}\big|_{b''=\frac{mr}{2}}=
 {\rm e}^{\frac{\pi {\rm i}r}{2}q_s(\mu m)-2\pi {\rm i} \lk(\mu m,b')}
\end{equation*}
and shifting the summation variables $b'\rightarrow b'+\frac{r}{2} \mu m$, $a'\rightarrow a'-\mu m$, we have indeed
\begin{equation*}
 \N_{r}(M,\omega'\oplus \omega'')
 \rightsquigarrow \N_{r}(M,\omega'\oplus (\omega''+2\mu^*\omega)),
\end{equation*}
where
$
 \omega'\oplus \omega''
 \in \Tor H_1(M;\Z)\oplus (\C/2\Z)^{b_1}
$
and
$
 \mu^*\colon \Tor H_1(M;\Z)\rightarrow (\C/\Z)^{b_1}
$
is the map dual to~$\mu$ in~\eqref{free-to-tor}.

For $r=1\bmod 4$, we have
\begin{equation*}
 \mathcal{A}=r^{-V/2+b_1/2}
 \xi^{\frac{3\sigma-\Tr B}{2}}
 {\rm e}^{\frac{\pi {\rm i}}{2}(\Tr B+\sigma)}
\end{equation*}
and
\begin{align}
 \N_{r}(M,\omega)&=
 \frac{r^{b_1} \mathcal{T}(M,[\omega])}{{|\Tor H_1(M;\Z)|}}
 \sum_{\substack{a',b',f'\in \Tor H_1(M;\Z)\\m\in \Z^{b_1}}}
 \exp\bigg\{2\pi {\rm i}\biggl(
 -\frac{r-1}{4}\lk(a',a')
 +\lk(a',f'-b')\nonumber
 \\
 &-\frac{1}{2}\omega(a')
 +\lk(f',f')
 -\Delta_{\sigma(b',s)}-\lk(b',b')
 +\omega''(m)\biggr)\bigg\}
\hat{Z}_{\sigma(b'\oplus rm,s)}\Big|_{q\rightarrow {\rm e}^{\frac{2\pi {\rm i}}{r}}}.
\label{CGP-Zhat-1mod4-mod}
\end{align}

\begin{Example}
An interesting example is $M=\Sigma_g\times S^1$.
In this case $\lk=0$ and $q_s=0$.
Furthermore, we have the identity:
\begin{equation*}
\mathcal{T}(M,[\omega])=(-2)^{b_1+1}\bigg(\frac{\rm i}{4}\bigg)^{b_1}\frac{\rm i}{2}\N_2(M,2\omega).
\end{equation*}
Thus replacing in the above formula we get
\begin{equation*}
 \frac{\N_{r}\big(\Sigma_g\times S^1,\omega\big)}{\N_{2}\big(\Sigma_g\times S^1,2\omega\big)}={r^{b_1}(-2)^{1+b_1}\bigg(\frac{\rm i}{4}\bigg)^{b_1}\frac{\rm i}{2}}
 \sum_{m \in \Z^{2g+1}} {\rm e}^{+\pi {\rm i} \omega(m)} \hat{Z}_{\sigma(\frac{rm}{2},s)}\Big|_{q\rightarrow {\rm e}^{\frac{2\pi {\rm i}}{r}}}.
\end{equation*}
If we now recall that
\begin{equation*}
\N_r\big(\Sigma_g\times S^1,\omega\big)=r^{2g}\sum_{k\in H_r}\bigg( \frac{\{r\beta\}}{\{\beta+k\}}\bigg)^{2g-2}
\end{equation*}
\big(so in particular,
$\N_2\big(\Sigma_g\times S^1,2\omega\big)=2^{2g+1}\frac{1}{{\rm i}^{2g-2}}\big({\rm i}^{2\beta}-{\rm i}^{-2\beta}\big)^{2g-2}$\big), where $\beta=\omega\big(S^1\big)$ (note the formula does not depend on the orientation of $S^1$), then we get
\begin{equation}\label{eq:sigmas1}
 \frac{1}{r}\sum_{k\in H_r}\frac{\{r\beta\}^{2g-2}}{(\{\beta+k\}({\rm i}^{2\beta}-{\rm i}^{-2\beta}))^{2g-2}} =\sum_{m \in \Z^{2g+1}} {\rm e}^{+\pi {\rm i} \omega(m)} \hat{Z}_{\sigma(\frac{rm}{2},s)}\Big|_{q\rightarrow {\rm e}^{\frac{2\pi {\rm i}}{r}}}.
\end{equation}

In order to compute $\hat{Z}_\mathfrak{s}$, we observe that the formula~\eqref{eq:zhatplumbing} providing the value of $\hat{Z}$ for surgeries over plumbing links can be generalised to manifolds which are the boundaries of a~plumbing of surfaces in a tree-like fashion. In particular, for $M=\Sigma_g\times S^1$,
\begin{equation}
\big(x-x^{-1}\big)^{2-2g}=\sum_{\ell \in \Z} F_\ell x^\ell.
\end{equation}
Since $\Tor( H_1(M))=0$, the first Chern class provides a bijection between the set of spin$^c$ structures on $M$ and $H^2(M)=H_1(M)$, but $\hat{Z}_\mathfrak{s}$ is zero for all spin$^c$ structures $\mathfrak{s}$ such that $c_1(\mathfrak{s})\neq PD\big(\ell \big[\{pt\}\times S^1\big]\big)$ for some $\ell \in \Z$. So letting $\ell$ be the spin$^c$ structure the first Chern class of which is Poincar\'e dual to $\ell\big[\{pt\}\times S^1\big]$ and using~\eqref{Zhat-b1-def} with $U={\rm Id}$, $B=0$, $\sigma=0=b_+$, we have
\smash{$\hat{Z}_\ell=q^{\ell} F_\ell$}.

Therefore, equation~\eqref{eq:sigmas1} becomes
\begin{equation}\label{eq:proofverlinde}
 \sum_{k\in H_r}\frac{1}{\big({\rm e}^{\frac{\pi {\rm i}}{r} (\beta+k)}-{\rm e}^{-\frac{\pi {\rm i}}{r}(\beta+k)}\big)^{2g-2}}=r\sum_{m \in \Z}{\rm e}^{\pi {\rm i} m\beta} \hat{Z}_{\sigma(\frac{rm}{2},s)}\big|_{q\rightarrow {\rm e}^{\frac{2\pi {\rm i}}{r}}}.
\end{equation}
Then we have that the left-hand side of~\eqref{eq:proofverlinde} equals
\begin{equation*}
 \sum_{\ell\in \Z}F_\ell\sum_{k\in H_r}{\rm e}^{\ell \frac{\pi {\rm i}}{r} (\beta+k)}=r\sum_{l\in \Z} F_{rl} {\rm e}^{\pi {\rm i} l\beta}(-1)^{l(r-1)}
\end{equation*}
and the right-hand side equals
\begin{equation*}
r\sum_{m\in \Z}{\rm e}^{\pi {\rm i} m\beta} F_{rm}{\rm e}^{\pi {\rm i} m}.
\end{equation*}
So if $r$ is even then~\eqref{eq:proofverlinde} is verified directly and if $r$ is odd, then the equality is true because $F_{rl}=0$ for odd $l$.
\end{Example}

\subsection{Surgeries on knots}

Consider $M=S^3_p(K)$ and assume that $-p\in \Z_+$ for concreteness. Then recalling (see Section~\ref{sec:combinatorial}) that a spin$^c$ structure on $S^3_p(K)$ can be encoded by an integer congruent to $p\mod 2$ and that two integers give the same structure iff they differ by a multiple of $2p$, the general surgery formula~\eqref{Zhat-b1-def} reads:
\begin{equation}
 \hat{Z}_a\big[S^3_p(K)\big]= q^{-\frac{3+p}{4}}
 \sum_{\ell =a \bmod 2p} F_\ell q^{-\frac{\ell^2}{4p}}, \label{Zhat-p-surgery}
\end{equation}
where the coefficients $F_\ell$ appear in the expansion of
\begin{equation}
 F(x,q):=F_K\big(x^2,q\big)\big(x-x^{-1}\big)= \sum_{\ell} F_\ell x^\ell
 \label{F-function-knot}
\end{equation}
with $F_K$ introduced in~\cite{Gukov:2019mnk}. We will use the facts that \smash{$F_K\big(x^{-2},q\big)=-F_K\big(x^2,q\big)$} and that \smash{$F_K\big(x^2,q\big)\in x\Z\big[\big[x^{\pm 2},q\big]\big]$}. In particular, it follows that $\hat{Z}_a\equiv 0$ for $a=1\bmod 2$.
As argued in~\cite{Gukov:2020lqm}, the series $F_K (x,q)$ gives ADO polynomials at roots of unity (which, in turn, are related to the CGP invariants for knot complements). Note, this already establishes a relation between $F_K (x,q) := \hat{Z} \big(S^3 \setminus K\big)$ and $\N_r (K)$ for knots, and tells us the relation between the parameters: in order to obtain $\N_r$ on the CGP side this large class of examples shows that on the GPPV side we need to take $q = {\rm e}^{2 \pi {\rm i} /r}$ (not $q = {\rm e}^{\pi {\rm i} /r}$).

With our choice of normalization, Conjecture~\ref{conj:park}\,$(b)$ states that
\begin{equation}
\left. F_K (x,q) \right|_{q = \xi^2} =
\frac{\operatorname{ADO}_r(K) \big(x/\xi^2\big)}{\Delta_K (x^{r})} \cdot \big( x^{1/2} - x^{-1/2} \big), \qquad
\xi := {\rm e}^{\pi {\rm i} /r},
\label{FKADO}
\end{equation}
where $\Delta_K(t)$ denotes the Alexander polynomial of $K$.
We wish to compose this with the relation between ADO polynomials and CGP invariants for knot complements
\begin{equation*}
\operatorname{ADO}_r (K)\big(x^2/\xi^2\big) =
\frac{x^r - x^{-r}}{x - x^{-1}} \N_r (K_{\alpha}), \qquad
\text{where}\quad x = {\rm e}^{\frac{\pi {\rm i} \alpha}{r}}.
\end{equation*}
Eliminating the ADO polynomial from the above two relations we get a more direct relation between GPPV and CGP invariants for knot complements
\begin{equation}
 \mathrm{N}_r(K_\alpha) =
 \frac{F_K\big({\rm e}^{\frac{2\pi {\rm i}\alpha}{r}},{\rm e}^{\frac{2\pi {\rm i}}{r}}\big)\Delta_{K}\big({\rm e}^{2\pi {\rm i}\alpha}\big)}{\big({\rm e}^{\pi {\rm i} \alpha}-{\rm e}^{-\pi {\rm i}\alpha }\big)}.
 \label{CGP-Zhat-knot2}
\end{equation}
This is the particular case of the more general Conjecture~\ref{conj:FADO2} $(b)$ combined with Conjecture~\ref{conj:FADO3}, in the case of the knot with zero framing.
On the other hand, we have
\begin{equation}
 \mathrm{N}_r\big(S^3_p(K),\omega\big)
 = \frac{1}{\Delta_-} \sum_{k \in H_r} d(\alpha_k)
 \N_r(K_{\alpha_k})T(\alpha_k)^p,
 \label{CGP-p-surgery}
\end{equation}
where $\mathrm{N}_r(K_{\alpha})$ denotes CGP invariant of a (zero framing) knot in $S^3$ colored by $\alpha$ (e.g., for unknot~$U$ with zero framing $\mathrm{N}_r(U_\alpha)=d(\alpha)$). The other notations are the same as in Section~\ref{sec:CGP}. In~particular, $\alpha_{k}:=\alpha+\mu$, where $\mu=\omega(\mathfrak{m})\in \frac{2}{p}\Z/\Z$ and $\mathfrak{m}\in H_1\big(S^3_p(K);\Z\big)\cong \Z/p\Z$ is the generator represented by the meridian of the knot $K$.

We want to check that the surgery formulas for homological blocks $\hat{Z}$ and CGP invariant $\mathrm{N}_r$ are consistent with the conjectural relations between these invariants for knot complements and closed manifolds. In other words, we want to check commutativity of the following schematic diagram:
	\begin{equation*}
	\begin{tikzcd}[row sep=huge, column sep=huge]
	 \hat{Z}_b\big[S^3_{p}(K)\big](q) \ar[r,"q\rightarrow {\rm e}^{\frac{2\pi {\rm i}}{r}}"] &
	 \mathrm{N}_r\big(S^3_{p}(K),\omega\big) \\
	 F_K(x,q)
	 \ar[u,"\substack{\text{Laplace}\\\text{transform}}"]
	 \ar[r,"q\rightarrow {\rm e}^{\frac{2\pi {\rm i}}{r}}"]
	 &
	 \mathrm{N}_r(K_\alpha),
	 \ar[u,"\substack{\text{Kirby}\\\text{color}}"]
\end{tikzcd}
	\end{equation*}
where, more concretely, the left vertical arrow is given by the equations~\eqref{Zhat-p-surgery}--\eqref{F-function-knot}, the right vertical arrow is given by~\eqref{CGP-p-surgery}, the top horizontal arrow is given by~\eqref{CGP-Zhat-1mod4},~\eqref{CGP-Zhat-2mod4} and~\eqref{CGP-Zhat-3mod4}, and the bottom horizontal arrow is given by~\eqref{CGP-Zhat-knot2}.

This check can be done by essentially repeating the analysis done for plumbings in case of a~single vertex (with $F$ in~\eqref{F-function} replaced by~\eqref{F-function-knot}) and using the following well known relation between the Reidemeister torsion of $M=S^3_p(K)$ and the Alexander polynomial of the knot~$K$:
\begin{equation}
 \mathcal{T}\big(S^3_p(K),t\big)= \frac{t\Delta_K(t)}{(1-t)^2}\bigg|_{t\in \Z_p\subset \operatorname{U}(1)},
 \label{torsion-surgery}
\end{equation}
where $t$ is the holonomy of $\operatorname{U}(1)$ flat connection along the the meridian $\mathfrak{m}$ of the knot $K$. But let us write it explicitly anyway.

As for plumbings, after plugging~\eqref{CGP-Zhat-knot} into~\eqref{CGP-p-surgery}, it is instructive to separate the result into three factors:
\begin{equation*}
 \N_r\big(S^3_p(K),\omega\big)=\mathcal{A}\cdot \mathcal{B}\cdot \mathcal{C},
\end{equation*}
where
\begin{gather}
 \mathcal{A}=r^{-1/2} \xi^{-\frac{3+p}{2}}\cdot
 \begin{cases}
 {\rm e}^{\frac{\pi {\rm i} p}{2}} {\rm e}^{-\frac{\pi {\rm i}}{2}}, & r=1\bmod 4, \\
 2^{-1/2} {\rm e}^{\frac{\pi {\rm i}}{4}}, & r=2\bmod 4, \\
 -{\rm e}^{-\frac{\pi {\rm i} p}{2}}, & r=3\bmod 4,
 \end{cases}
 \label{factor-A-surgery}
\\
 \mathcal{B}=\frac{\Delta_K \big({\rm e}^{2\pi {\rm i}\mu}\big)}{\big({\rm e}^{\pi {\rm i}\mu}-{\rm e}^{-\pi {\rm i}\mu}\big)^2},
 \label{factor-B-surgery}
\\
 \mathcal{C}=\sum_{k\in H_r} F\big(\xi^{\mu+k},\xi^2\big)\cdot
 \xi^{\frac{p}{2}(\mu+k)^2},
 \label{factor-C-surgery}
\end{gather}
where, as before, $\xi:={\rm e}^{\frac{\pi {\rm i}}{r}}$.
Consider a contribution of a monomial $x^\ell$ from $F_K(x)$ into~\eqref{factor-C-surgery}
\begin{equation}
 \mathcal{C}_\ell:=\sum_{k\in H_r} \xi^{\ell(\mu+k)}\cdot
 \xi^{\frac{p}{2}(\mu+k)^2}=
 \sum_{n\in \Z/r\Z}
 {\rm e}^{\frac{\pi {\rm i} p}{2r}(\tilde\mu+2n)^2+\frac{\pi {\rm i}}{r}\ell(\tilde\mu+2n)},
 \label{factor-C-monomial}
\end{equation}
where
\begin{equation*}
 \tilde\mu:=\mu+(r-1).
\end{equation*}
We can now use the following one-dimensional Gauss reciprocity formula:
\begin{equation*}
\sum_{n \in \Z/r\Z}
\exp\bigg(\frac{2\pi {\rm i}}{r}\big(p n^2+\ell n\big)\bigg)
= {\rm e}^{\frac{\pi {\rm i} \operatorname{sign}(p)}{4}} \sqrt{\frac{r}{2|p|}}
\sum_{\tilde{a} \in \Z/2p\Z}
\exp\biggl(-\frac{\pi {\rm i} r}{2p}\bigg(\tilde{a}+\frac{\ell}{r}\bigg)^2\biggr).
\end{equation*}
Applying it to~\eqref{factor-C-monomial} we have
\begin{equation*}
 \mathcal{C}_\ell= \xi^{-\frac{\ell^2}{2p}}
 {\rm e}^{-\frac{\pi {\rm i}}{4}} \sqrt{\frac{r}{2|p|}}
 \sum_{\tilde{a}\in \Z /2p\Z}
 {\rm e}^{-\frac{\pi {\rm i} r}{2p} \tilde{a}^2-\pi {\rm i} \tilde{a} (\ell/p+\tilde\mu)}.
\end{equation*}
Therefore, taking into account~\eqref{Zhat-p-surgery} we can write
\begin{equation*}
 \mathcal{C}=
 \sum_{\ell}\mathcal{C}_\ell F_\ell =
 {\rm e}^{-\frac{\pi {\rm i}}{4}} \xi^{\frac{3+p}{2}}
 \sqrt{\frac{r}{2|p|}}
 \sum_{\substack{b\in \Z/p\Z \\ \tilde{a}\in \Z/2p\Z}}
 {\rm e}^{-\frac{\pi {\rm i} r}{2p}\tilde{a}^2
 -\frac{2\pi {\rm i} \tilde{a}b}{p}-\pi {\rm i}\tilde{\mu}\tilde{a}}
 \hat{Z}_{2b}\big[S^3_p(K)\big]\Big|_{q\rightarrow \xi^2}.
\end{equation*}
Combining this together with~\eqref{factor-A-surgery} and~\eqref{factor-B-surgery}, we get
\begin{align*}
 \mathrm{N}_r\big(S^3_p(K),\omega\big) ={}&
 \frac{1}{\sqrt{{|p|}}}
 \frac{\Delta_K ({\rm e}^{2\pi {\rm i}\mu})}{({\rm e}^{\pi {\rm i}\mu}-{\rm e}^{-\pi {\rm i}\mu})^2}
 \left\{\!\!
 \begin{array}{ll}
 2^{-1}, & r=2\bmod 4 \\
 2^{-1/2} {\rm e}^{\mp\frac{\pi {\rm i} p}{2}}
 {\rm e}^{\mp\frac{3\pi {\rm i}}{4}}, &
 r=\pm 1\bmod 4 \end{array}\!\! \right\}
 \\& \times \sum_{\substack{b\in \Z/p\Z \\ \tilde{a}\in \Z/2p\Z}}
 {\rm e}^{-\frac{\pi {\rm i} r}{2p}\tilde{a}^2
 -\pi {\rm i} (r-1)\tilde{a}
 -\frac{2\pi {\rm i} \tilde{a}b}{p}-\pi {\rm i}{\mu}\tilde{a}}
 \hat{Z}_{2b}\big[S^3_p(K)\big]\Big|_{q\rightarrow {\rm e}^{\frac{2\pi {\rm i}}{r}}}.
\end{align*}

Using the formula~\eqref{torsion-surgery} with $t={\rm e}^{2\pi {\rm i}\mu}$, we indeed arrive at the conjectural formula~\eqref{CGP-Zhat-2mod4}, \eqref{CGP-Zhat-1mod4} or~\eqref{CGP-Zhat-3mod4}, depending on the value $r\bmod 4$, in the case of $M=S^3_p(K)$. For example, when $r=2\bmod 4$ the sum over $\tilde{a}$ reduces to the sum over $a=\tilde{a}\bmod p$:
\begin{align}
 \mathrm{N}_r\big(S^3_p(K),\omega\big) =
 \frac{1}{\sqrt{{|p|}}}
 \frac{\Delta_K \big({\rm e}^{2\pi {\rm i}\mu}\big)}{\big({\rm e}^{\pi {\rm i}\mu}-{\rm e}^{-\pi {\rm i}\mu}\big)^2}
 \!\!\sum_{\substack{b\in \Z/p\Z \\ a\in \Z/p\Z}}\!\!
 {\rm e}^{-\frac{\pi {\rm i} r}{2p}{a}^2
 -\pi {\rm i} a
 -\frac{2\pi {\rm i} ab}{p}-\pi {\rm i}{\mu}a}
 \hat{Z}_{2b}\big[S^3_p(K)\big]\Big|_{q\rightarrow {\rm e}^{\frac{2\pi {\rm i}}{r}}}.\!\!\!
 \label{p-surgery-CGP-Zhat-2mod4}
\end{align}
Similarly, for $r=1 \bmod 4$,
\begin{equation*}
c^{\mathrm{CGP}}_{a,b} =
\frac{1}{|p|} \frac{\Delta_K \big({\rm e}^{4\pi {\rm i} a/p}\big)}{({\rm e}^{2\pi {\rm i} a/ p} - {\rm e}^{- 2\pi {\rm i} a/ p})^2}
{\rm e}^{- 2\pi {\rm i} \frac{3\operatorname{sign} (p) - p}{4}}
\sum_{c,f=0}^{p-1} {\rm e}^{- \frac{2\pi {\rm i}}{p} ( - f^2 + ac + (b-f)c + (r-1) \frac{c^2}{4} )}
\end{equation*}
and for $r=3$ mod~4,
\begin{equation*}
c^{\mathrm{CGP}}_{a,b} =
\frac{1}{|p|} \frac{\Delta_K \big({\rm e}^{4\pi {\rm i} a/p}\big)}{\big({\rm e}^{2\pi {\rm i} a/ p} - {\rm e}^{- 2\pi {\rm i} a/ p}\big)^2}
{\rm e}^{2\pi {\rm i} \frac{3\operatorname{sign} (p) - p}{4}}
\sum_{c,f=0}^{p-1} {\rm e}^{- \frac{2\pi {\rm i}}{p} ( f^2 + ac + (b+f)c + (r+1) \frac{c^2}{4} )},
\end{equation*}
where we use the notation of~\eqref{coeffsummary}.

\begin{Example}
As a concrete example, consider $p=-3$ surgery on the right-handed trefoil $K = {\bf 3_1^r}$. In this case, $H_1 (S^3_p (K)) = \Z_3$ and, therefore, there are two independent $\hat Z$-invariants, which can be expressed in terms of the false theta-functions (cf.~\cite{Cheng:2018vpl,Gukov:2019mnk}):
\begin{gather*}
\hat Z_0 = q^{\frac{71}{72}}
\big( \tilde \Psi_{18}^{(1)} + \tilde \Psi_{18}^{(17)} \big)
= q + q^5 - q^6 - q^{18} + q^{20} + \cdots,
\\
\hat Z_{\pm 1} = - \frac{1}{2} q^{\frac{71}{72}}
\big( \tilde \Psi_{18}^{(5)} + \tilde \Psi_{18}^{(13)} \big)
= - \frac{1}{2} q^{4/3} \big( 1 + q^2 - q^7 - q^{13} + q^{23} + \cdots \big),
\end{gather*}
where the factor $\frac{1}{2}$ in the latter expression and ``$\pm$'' in its label appear precisely because we write these expressions in the {\it unfolded} form.
Evaluating the right-hand side of~\eqref{p-surgery-CGP-Zhat-2mod4} for various values of $r$, we find
{\samepage\begin{alignat*}{6}
&r=5\colon &&\phantom{-2}4.85591-6.4514 \mathrm{i},
&&r=13\colon &&-37.2754-1.28057 \mathrm{i},
\\
&r=6\colon &&\phantom{-2}5.30731-4.45336 \mathrm{i},
&&r=14\colon &&-11.4885+28.4093 \mathrm{i},
\\
&r=7\colon &&\phantom{-2}1.89035+3.49675 \mathrm{i},
&&r=15\colon &&-15.3891+13.8158 \mathrm{i},
\\
&r=9\colon &&-6.77162-0.394402 \mathrm{i},
&&r=17\colon &&-11.2632+37.6555 \mathrm{i},
\\
&r=10\colon \quad&&-24.2779+7.76375 \mathrm{i},\qquad\quad
&&r=18\colon \quad&&\phantom{-2}17.4965+18.5452 \mathrm{i},
\\
&r=11\colon &&-6.01733+3.60533 \mathrm{i},
&&r=19\colon &&\phantom{-2}59.3259+18.3538 \mathrm{i},
\\
& &&&&\cdots\cdots
\end{alignat*}
which match the corresponding values of $\mathrm{N}_r(S^3_{-3} ({\bf 3_1^r}))$.}
\end{Example}

\subsection{0-surgeries on knots}\label{sec:0surgknot}

In general, the invariant of a knot complement is usually written as
\begin{equation}
F_K (x,q) = \frac{1}{2} \sum_{\substack{m \ge 1 \\ \text{odd}}} f_m (q) \cdot \big(x^{\frac{m}{2}} - x^{- \frac{m}{2}}\big),
\label{FKgeneral}
\end{equation}
which after multiplying by $\big(x^{\frac{1}{2}} - x^{- \frac{1}{2}}\big)$ in the surgery formula gives
\begin{equation*}
\big(x^{\frac{1}{2}} - x^{- \frac{1}{2}}\big) F_K (x,q) = \frac{1}{2} \sum_{\substack{m \ge 1 \\ \text{odd}}} f_m (q) \cdot \big(x^{\frac{m}{2} + \frac{1}{2}} - x^{\frac{m}{2} - \frac{1}{2}} - x^{- \frac{m}{2} + \frac{1}{2}} + x^{- \frac{m}{2} - \frac{1}{2}} \big),
\end{equation*}
which means (with $n \in \Z$):
\begin{equation}
\text{Coeff}_{x^n} \Big[ \big(x^{\frac{1}{2}} - x^{- \frac{1}{2}}\big) F_K (x,q) \Big] =
\begin{cases}
f_{2n-1} - f_{2n+1} & \text{if}\quad n \ge 1, \\
- 2 f_{1} & \text{if}\quad n = 0, \\
f_{2|n|-1} - f_{2|n|+1} & \text{if}\quad n \le -1.
\end{cases}
\label{ZnFm}
\end{equation}
For the unknot we have $f_m (q) = \delta_{m,1}$, so that $\hat Z_n = \{ \ldots, 0,0,1,-2,1,0,0,\ldots \}$, where by $\hat{Z}_n$ we denote the invariant $\hat{Z}_\mathfrak{s}\big(S^3(K)\big)$ associated to the spin$^c$ structure encoded by the integer $2n$ on the knot (see Section~\ref{sec:combinatorial}).

\begin{Example}
The right-handed trefoil knot $K = {\bf 3_1^r}$: consider the $0$-framed trefoil and let $H_r=\{-(r-1),-(r-3),\ldots, (r-1)\}$ we have
\begin{equation*}
\N_r \big(({\bf 3_1^r})_\alpha\big) =
\frac{(-1)^{r-1}\xi^{9 (r-1)^2/4}}{\{ 2r \alpha \}}
\sum_{n \in H_r} \xi^{3n \alpha + \frac{3n^2}{4}}
\{ 2 \alpha + n \}
\end{equation*}
Therefore, we get
\begin{align*}
\N_r (S^3_0 ({\bf 3_1^r})) &=
\sum_{k \in H_r} d (\alpha + k)
\frac{(-1)^{r-1} \xi^{9 (r-1)^2/4}}{\{ 2r \alpha + 2r k \}}
\sum_{n \in H_r} \xi^{3n \alpha + 3nk + \frac{3n^2}{4}}\{ 2 \alpha + 2k + n \}
\\
&= (-1)^{r-1}\frac{ \xi^{9 (r-1)^2/4}}{ \{ r \alpha \} \{ 2r \alpha \}}
\sum_{k \in H_r} \sum_{n \in H_r}
(-1)^k \{ \alpha + k \} \{ 2 \alpha + 2k + n \} \xi^{3n \alpha + 3nk + \frac{3n^2}{4}}.
\end{align*}
On the other hand, for the right-handed trefoil knot $K = {\bf 3_1^r}$ we have $f_m = \epsilon_m q^{\frac{m^2 + 23}{24}}$, where
\begin{equation*}
\epsilon_m =
\begin{cases}
-1 & \text{if}\quad m \equiv 1\text{ or } 11\!\!\!\pmod {12},
\\
+1 & \text{if}\quad m \equiv 5\text{ or } 7\!\!\!\pmod {12},
\\
\phantom{-} 0 & \text{otherwise}.
\end{cases}
\end{equation*}
and so
\begin{equation*}
\hat Z_n =
\begin{cases}
\epsilon_{2n-1} q^{\frac{(2n-1)^2 + 23}{24}}
- \epsilon_{2n+1} q^{\frac{(2n+1)^2 + 23}{24}} & \text{if}\quad n \ge 1,
\\[1mm]
- 2 \epsilon_1 q &\text{if}\quad n = 0,
\\[1mm]
\epsilon_{2|n|-1} q^{\frac{(2|n|-1)^2 + 23}{24}}
- \epsilon_{2|n|+1} q^{\frac{(2|n|+1)^2 + 23}{24}} & \text{if}\quad n \le -1.
\end{cases}
\end{equation*}
Again, we verified numerically and analytically that the proposed relations hold.
\end{Example}

\subsection{Surgery on a link}

Let $L$ be the link $L_927$ in the Thistlethwaite table of links (see also Knot Atlas).
Let $M$ be obtained by $a,b,c\in \mathbb{N}$ integral surgery on the three components of $L$, where $a$ corresponds to the blue component, $b$ to the purple and $c$ to the green one in Figure~\ref{fig:L927}.
The linking matrix of $L$ is diagonal with entries $a$, $b$, $c$ so that a cohomology class $\omega\in H^1(M;\C/2\Z)$ is described by a three-uple $(\alpha,\beta,\gamma)$ with $\alpha\in \big\{\frac{2k}{a}, k=0,\ldots, a-1\big\}$, $\beta\in \{\frac{2k}{b}, \, k=0,\ldots, b-1\}$ and $\gamma\in \big\{\frac{2k}{c},\, k=0,\ldots, c-1\big\}$.
\begin{figure}
 \centering
 \includegraphics[width=4cm]{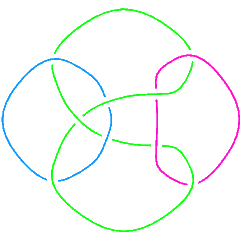}
 \caption{The $L_927$ link (Knotscape image).}
 \label{fig:L927}
\end{figure}

$L$ is one of the first links with the peculiar property that its multicolored Alexander polynomial is zero.
Thus we have $\N_2(M,\omega)=0$ for any $\omega\in H^1(M;\C/2\Z)$.

On the other hand, a direct computer-based calculation gives
\begin{equation*}
\mathrm{ADO}_3(L)(\alpha,\beta,\gamma)=\xi^{2\gamma}+\big(1-{\rm i}\sqrt{3}\big)+\xi^{-2\gamma}
\end{equation*}
so that in particular it does not depend on $\alpha$, $\beta$ and we will denote it $P(\gamma)$.
As a consequence, we have
\begin{align*}
\N_3(M,\omega)={}&\Delta_+^{-3}\bigg(\sum_{k_1\in \{-2,0,2\}} \xi^{\frac{a}{2}((\alpha+k_1)^2-(r-1)^2)}d(\alpha+k_1)\bigg)
\\
&\times \bigg(\sum_{k_2\in \{-2,0,2\}}
\xi^{\frac{b}{2}((\beta+k_2)^2-(r-1)^2)}d(\beta+k_2)\bigg)
\\
&\times \bigg(\sum_{k_3\in \{-2,0,2\}}
\xi^{\frac{c-1}{2}((\gamma+k_3)^2-(r-1)^2)}d(\gamma+k_3)P(\gamma+k_3)\bigg)
\end{align*}
(in the last factor of the above expression we used $c-1$ as the self linking of the green component is $1$ in the given diagram).
Choosing for instance $a=b=c=4$ and $\alpha=\beta=\gamma=\frac{1}{2}$ one can compute directly that $\N_3(M,\omega)\neq 0$. Since the torsion is $0$ as the multivariable Alexander polynomial is, then in this case the invariants $\hat{Z}$ should have infinite limits when $q\to \exp(2{\rm i}\frac\pi3)$.

\section[The relation between $\N_r$ and $\hat{Z}$]{The relation between $\boldsymbol{\N_r}$ and $\boldsymbol{\hat{Z}}$}
\label{sec:relation}

In this section, we present a more general and systematic discussion of the proposed relation.

\subsection{Gauss sum vs Laplace transform}
\label{sec:gauss-vs-laplace}
Let $M$ be the rational homology sphere obtained by a surgery on a framed link $L$ in $S^3$ with a linking matrix $B$. We will use the shorthand notations $x\equiv \{x_I\}_{I\in \vert}$, $x^\pm\equiv \big\{x_I^{\pm 1}\big\}_{I\in \vert}$. In particular, $K[[x^\pm]]$ denotes the space of formal Laurent series in $x_I^\pm$ with coefficients in $K$, considered as a module over the ring of Laurent polynomials $K[x^\pm]$.

As before, $\omega\in H^1(M;\C/2\Z)$ and $\mu_I=\omega(\mathfrak{m}_I)$, which we assume to be fixed. We also use notation
\begin{equation*}
 \xi\equiv {\rm e}^{\frac{\pi {\rm i}}{r}}
\end{equation*}
in what follows.

Let $K$ be a field extension over $\C$ (some of the relevant cases are $K=\C$, $K=\C\big(\big(q^{1/p}\big)\big)$, $K=\C\big(\big(q^{1/p}\big)\big)(x)$). By $K(t)'$ we will denote the localization of $K[t]$ by the subset of polynomials in $t$ non-vanishing at $t=1$. This is a subset of the field $K(t)$ of rational functions in $t$ closed under multiplication. One can then consider the following ring homomorphism:
\begin{Definition}
\label{def:t-limit-rational}
\begin{align*}
 \lim_{t\rightarrow 1}\colon \ K(t)'&\longrightarrow K,
 \\
 \frac{P(t)}{Q(t)}&\longmapsto \frac{P(1)}{Q(1)}.
\end{align*}
\end{Definition}

Similarly, by $K(x)'$ we will denote the localization of $K[x^\pm]$ by the subset of Laurent polynomials non-vanishing at $x=\xi^{\mu+k}$ for any $k\in H_r^\vert$. This is a subset of the field $K(x)$ of rational functions in $x$ closed under multiplication. We can then define the following $K$-linear ``$\omega$-twisted Gauss sum'' operation:
\begin{Definition}[$\omega$-twisted Gauss sum]
\label{def:gauss-omega}
\begin{align*}
 \Go_\omega\colon \ K(x)' & \longrightarrow K,
 \\
 \frac{P(x)}{Q(x)} & \longmapsto
 \sum\limits_{k\in H_r^\vert} \frac{P(\xi^{\mu+k})}{Q(\xi^{\mu+k})}
 \xi^{\frac{1}{2} (\mu+k)^TB(\mu+k)}.
\end{align*}
\end{Definition}

Denote $p:=4|\det B|$ and let $\C((q^{1/p}))\equiv \C\big[\big[q^{1/p},q^{-1/p}\big]$ be the field of fractions of $\C\big[\big[q^{1/p}\big]\big]$.
Then define the $\C\big(\big(q^{1/p}\big)\big)$-linear ``$\omega$-twisted Laplace transform'' on a subspace of $\C\big(\big(q^{1/p}\big)\big)[[x^\pm]]$ by the following formula:
\begin{Definition}[$\omega$-twisted Laplace transform]
\label{def:Laplace-omega}
\begin{equation}
\begin{aligned}[c]
 \Lo_\omega\colon\ \ \C\big(\big(q^{1/p}\big)\big)\big[\big[x^\pm\big]\big] & \longrightarrow \C\big(\big(q^{1/p}\big)\big),
 \\[2mm]
 \sum\limits_{\substack{\ell\in \Z^\vert \\ m\in \Z}} A_{\ell,m} q^{\frac{m}{p}} x^\ell &
 \longmapsto \sum\limits_{n\in \Z} q^{\frac{n}{p}} \sum\limits_{\substack{\ell\in \Z^\vert \\ m\in \Z \\ 4m-p\ell^TB^{-1}\ell=4n}} A_{\ell,m}\cdot \mathcal{C}_\ell^\omega,
\end{aligned}
\label{Laplace-omega-formula}
\end{equation}
where
\begin{equation*}
 \mathcal{C}_\ell^\omega=\frac{{\rm e}^{\frac{\pi {\rm i}\sigma}{4}}(r/2)^{V/2}}{|\det B|^{1/2}}
 \sum_{\tilde{a}\in \Z^\vert /2B\Z^{\vert}}
 {\rm e}^{-\frac{\pi {\rm i} r}{2} \tilde{a}^T B^{-1}\tilde{a}-\pi {\rm i} \tilde{a}^T B^{-1}(\ell+B(\mu+(r-1)\varepsilon))}.
\end{equation*}
We say that $\Lo_\omega$ is well defined if the interior sum in the right-hand side of~\eqref{Laplace-omega-formula} has a finite number of non-zero terms.
\end{Definition}
\begin{Remark}
 $\Lo_\omega$ restricted on $\C\big(\big(q^{1/p}\big)\big)\big[x^\pm\big]$ is well-defined.
\end{Remark}
\begin{Remark}
 Apart from $\omega\in H^1(M;\C/2\Z)$ (equivalently, $\mu\in (\C/2\Z)^\vert$ s.t.\ $B\mu=0\bmod 2\Z^\vert$), the operations $\Lo_\omega$ and $\Go_\omega$ also (implicitly) depend on $B$ and $r$.
\end{Remark}

Next we introduce the following limit operation (a morphism of $\C$-algebras).
\begin{Definition}
\begin{align*}
 \lim_{q\rightarrow {\rm e}^{\frac{2\pi {\rm i}}{r}}}\colon\ \C\big(\big(q^{1/p}\big)\big)
 & \longrightarrow \C,
 \\
 \sum\limits_{m} A_m q^{\frac{m}{p}} & \longmapsto\
 \lim_{q\rightarrow {\rm e}^{\frac{2\pi {\rm i}}{r}}}\sum\limits_{m} A_m q^{\frac{m}{p}}.
\end{align*}
We say that the operation is well defined if the power series in $q^{1/p}$ are convergent for $0<|q^{1/p}|<1$ and the limit, taken to the root of unity along the radial direction, exists and is finite.
\label{def:q-limit}
\end{Definition}

\begin{Remark}
 This operation can be extended to the polynomials/series with coefficients in $\C\big(\big(q^{1/p}\big)\big)$ by applying it coefficient-wise.
\end{Remark}

\begin{Proposition}
\label{prop:Go-Lo-in-limit}
Let $A(x,q)\in \C\big(\big(q^{1/p}\big)\big)\big[x^\pm\big]$ such that $\lim_{q\rightarrow {\rm e}^\frac{2\pi {\rm i}}{r}}A(x,q)\in \C\big[x^\pm\big]$ exists. Then
\begin{equation*}
 \Go_\omega \lim_{q\rightarrow {\rm e}^\frac{2\pi {\rm i}}{r}} A(x,q)
 =\lim_{q\rightarrow {\rm e}^\frac{2\pi {\rm i}}{r}}\Go_\omega A(x,q)
 =\lim_{q\rightarrow {\rm e}^\frac{2\pi {\rm i}}{r}}\Lo_\omega A(x,q).
\end{equation*}
\end{Proposition}

\begin{proof}
The first equality follows from the definition of $\Go_\omega$, which involves taking a finite linear combination of the evaluations at certain values of $x$. The second equality
is shown by applying the Gauss reciprocity formula~\eqref{reciprocity-alt} to individual monomials in $x$, the number of which is finite.\looseness=1
\end{proof}

\begin{Proposition}
\label{prop:gauss-sum-factor-out}
Let $A(x)\in K(x)'$, $B(x^r)\in K(x)'\cap K(x^r)$. Then
\begin{equation*}
 \Go_\omega(B(x^r)A(x))=
 B((-1)^{r-1} {\rm e}^{\pi {\rm i}\mu})\Go_\omega(A(x)).
\end{equation*}
\end{Proposition}

\begin{proof}
 Follows from the definition of $\Go_\omega$.
\end{proof}

\begin{Remark}
 The statement of the Proposition~\ref{prop:gauss-sum-factor-out} is directly extended to the series in $K'(x)((t))$ by applying it coefficient-wise.
\end{Remark}

We also define the following limiting operation on a subalgebra of $\C\big(\big(q^{1/p}\big)\big)((t))$:
\begin{Definition}
\begin{align*}
 \lim_{t\rightarrow 1}\colon\qquad\qquad \C\big(\big(q^{1/p}\big)\big)((t)) & \longrightarrow \C\big(\big(q^{1/p}\big)\big),
 \\
 \sum\limits_{n}\bigg(\sum\limits_{m} A_{n,m} q^{\frac{m}{p}}\bigg)t^n &
 \longmapsto \sum\limits_{m}\bigg(\sum\limits_{n}A_{n,m}\bigg)q^{\frac{m}{p}}.
\end{align*}
We say that the operation is well defined if there is a finite number of non-vanishing coefficients~$A_{n,m}$ for any fixed $m$.
\label{def:t-limit}
\end{Definition}

\begin{Remark}
When restricted on the subspace $C\big(\big(q^{1/p}\big)\big)(t)'\subset C\big(\big(q^{1/p}\big)\big)((t))$ and is well-defined, the result of the operation in Definition~\ref{def:t-limit} coincides with the result of the operation in Definition~\ref{def:t-limit-rational}.
\end{Remark}

With such definitions, consider the following diagram of (partially defined) algebra homomorphisms, where the dotted arrow means that we make the hypothesis that in the cases of interest the image of $\lim_{q\to {\rm e}^{2\pi {\rm i}/r}}$ map is contained in the subspace $\C(t)'$:
\begin{equation}
	\begin{tikzcd}[row sep=huge, column sep=huge]
	 \C\big(\big(q^{1/p}\big)\big)((t)) \ar[r,"\lim_{q\rightarrow {\rm e}^{\frac{2\pi {\rm i}}{r}}}"]
	 \ar[d,"\lim_{t\rightarrow 1}"]
	 &
	 \C((t)) \ar[r,dashed,bend left=30]
	 &
	 \C(t)'
	 \ar[d,"\lim_{t\rightarrow 1}"]
	 \ar[l,hook]
	 \\
	 \C\big(\big(q^{1/p}\big)\big)
	 \ar[rr,"\lim_{q\rightarrow {\rm e}^{\frac{2\pi {\rm i}}{r}}}"]
	 &
	 &
	\C.
\end{tikzcd}
\label{lim-transform-diagram}
\end{equation}

\begin{Proposition}
 When $r\neq 0\bmod 4$ and
 \begin{equation*}
 \ell =2b+B(s-\varepsilon)\bmod 2B\Z^\vert,
 \end{equation*}
 where $b$ and $s$ represent elements of $H_1(M;\Z)$ and $\operatorname{Spin}(M)$ respectively, the coefficients $\mathcal{C}^\omega_\ell$ in the Definition~$\ref{def:Laplace-omega}$ admit the following expression:
 \begin{align*}
 \mathcal{C}_\ell^\omega={}&\frac{{\rm e}^{\frac{\pi {\rm i}\sigma}{4}}r^{V/2}}{|H_1(M;\Z)|}
 \\
 &\times\begin{cases}
 {\rm e}^{\frac{\pi {\rm i}}{4}(5\sigma-2\Tr B)}
 \sum\limits_{a,f\in H_1(M;\Z)}
 {\rm e}^{2\pi {\rm i}\left(-\frac{r-1}{4}\lk(a,a)
 +\lk(a,f-b)-\frac{1}{2}\omega(a) +\lk(f,f)
 -\frac{1}{4}\mu(M,s)+\frac{1}{2}\right)},
 \\[-3mm]
 \hspace{100mm}
 r=1 \bmod 4,
 \\
 2^{V/2} |H_1(M;\Z)|^{1/2}\sum\limits_{a\in H_1(M;\Z)}
 {\rm e}^{-\frac{\pi {\rm i}r}{2} q_s(a)-2\pi {\rm i} \lk(a,b)-\pi {\rm i} \omega(a)},
 \qquad \
 r=2 \bmod 4,
 \\[3mm]
 {\rm e}^{\frac{\pi {\rm i}}{4}(-5\sigma+2\Tr B)} \sum\limits_{a,f\in H_1(M;\Z)}
 {\rm e}^{2\pi {\rm i}\left( -\frac{r+1}{4}\lk(a,a)
 -\lk(a,f+b)-\frac{1}{2}\omega(a) -\lk(f,f)
 +\frac{1}{4}\mu(M,s)+\frac{1}{2}\right)},
 \\[-3mm]
 \hspace{100mm}
 r=3 \bmod 4.
 \end{cases}
 \end{align*}
 \label{prop:Laplace-omega-explicit}
\end{Proposition}
\begin{proof}

 Is contained in Section~\ref{sec:plumbed}.
\end{proof}

Assuming the identification $\xi^{\alpha_I} = x_I$, one can consider (see \eqref{ADO-Nr-knot} and~\eqref{ADO-Nr-link}):
\begin{equation*}
 \xi^{-\frac{1}{2}\alpha^TB\alpha} \N_r(L_\alpha)
 \in \C(x)' \subset \C(x).
\end{equation*}

\begin{Proposition}
\label{prop:Gauss-omega-CGP}
\begin{equation}
\frac{1}{\Delta_+^{b_+}\Delta_-^{b_-}}
\Go_\omega\bigg(\xi^{-\frac{1}{2}\alpha^TB\alpha} \N_r(L_\alpha)\cdot \frac{\prod_{I\in \vert}\big(x_I-x^{-1}_I\big)}{\prod_{I\in \vert}\big(x_I^r-x^{-r}_I\big)}\bigg)=
\N_r(M,\omega).
\label{eq:Gauss-omega-CGP}
\end{equation}
\end{Proposition}

\begin{proof}
After using the identification $x_I=\xi^{\alpha_I}$, taking into accound the formula~\eqref{modified-q-dim-formula} for the modified quantum dimensions, and writing explicitly the action of $\Go_\omega$ according to the formula in the Definition~\ref{def:gauss-omega}, the left-hand side of~\eqref{eq:Gauss-omega-CGP} becomes the surgery formula~\eqref{CGP-surgery-formula} for~$\N_r(M,\omega)$.
\end{proof}

\begin{Proposition}
\label{prop:Laplace-omega-Zhat}
\begin{align*}
&\Lo_\omega\bigg(F_L\big(x^2,q\big)\cdot \prod_{I\in \vert}\big(x_I-x^{-1}_I\big)\bigg)={}
\frac{{\rm e}^{\frac{\pi {\rm i}\sigma}{4}}r^{V/2}}{|H_1(M;\Z)|}
\\
&\times\sum_{b\in H_1(M;\Z)}
 \begin{cases}
 {\rm e}^{\frac{\pi {\rm i}}{4}(5\sigma-2\Tr B)}
 \!\!\sum\limits_{a,f\in H_1(M;\Z)}\!\!\!
 {\rm e}^{2\pi {\rm i}\left(-\frac{r-1}{4}\lk(a,a)
 +\lk(a,f-b)-\frac{1}{2}\omega(a) +\lk(f,f)
 -\frac{1}{4}\mu(M,s)+\frac{1}{2} \right)},
 \\[-3mm]\hspace{100mm}
 r=1 \bmod 4,
 \\[1mm]
 2^{V/2} |H_1(M;\Z)|^{1/2}\sum\limits_{a\in H_1(M;\Z)}
 {\rm e}^{-\frac{\pi {\rm i}r}{2} q_s(a)-2\pi {\rm i} \lk(a,b)-\pi {\rm i} \omega(a)},\qquad
\  r=2 \bmod 4,
 \\[2mm]
 {\rm e}^{\frac{\pi {\rm i}}{4}(-5\sigma+2\Tr B)}
\!\! \sum\limits_{a,f\in H_1(M;\Z)}\!\!\!
 {\rm e}^{2\pi {\rm i}\left( -\frac{r+1}{4}\lk(a,a)
 -\lk(a,f+b)-\frac{1}{2}\omega(a)
 -\lk(f,f) +\frac{1}{4}\mu(M,s)+\frac{1}{2}\right)},
 \\[-3mm]\hspace{100mm}
 r=3 \bmod 4, \\
 \end{cases}
 \\[1mm]
& \times (-1)^{b_+}q^{\frac{\Tr B-3\sigma}{4}} \hat{Z}_{\sigma(b,s)}
\end{align*}
when the left-hand side is well defined.
\end{Proposition}

\begin{proof}
Using the expansion~\eqref{FL-expansion}, the left-hand side reads:
\begin{equation}
 \Lo_\omega\bigg(\sum_{\ell\in \Z^\vert}F_\ell x^\ell\bigg)=
 \sum_\ell C_\ell^\omega F_\ell q^{-\frac{\ell^T B^{-1}\ell}{4}}=\sum_{b\in \Z^\vert/B\Z^\vert} \sum_{\substack{\ell=2b+B(s-\varepsilon)\\\bmod 2B\Z^{V}}} \mathcal{C}_\ell^\omega F_{\ell}q^{-\frac{\ell^tB^{-1}\ell}{4}},
\end{equation}
where in the first equality we used Definition~\ref{def:Laplace-omega}. In the second equality, we split the sum over $\ell$ into a sum over $\Z^\vert/B\Z^\vert\cong H_1(M,\Z)$ and the sum over $\ell$ with fixed values modulo $2B\Z^\vert$. Proposition~\ref{prop:Laplace-omega-explicit} provides an explicit formula for the coefficients $\mathcal{C}_\ell^\omega$, which depend only on the value $\ell$ modulo $2B\Z^\vert$. Combined with the surgery formula~\eqref{Zhat-surgery-RHS} for $\hat{Z}$ in the case of a rational homology sphere it gives us the formula in the statement of the proposition.
\end{proof}
Note that $t$-regularization can be understood as the following $\C$-linear map:
\begin{Definition}
\label{def:t-reg}
\begin{align*}
 (\,\cdot\,)^t\colon\qquad\qquad\quad K[[x^\pm]] & \longrightarrow K[x^\pm][[t]],
 \\
 f(x)=\sum\limits_{\ell\in \Z^\vert} f_\ell x^\ell & \longmapsto f^t(x)=\sum\limits_{m\geq 0} \bigg(\sum\limits_{\ell\colon\|\ell\|=m}f_\ell x^\ell\bigg)t^m,
\end{align*}
where $\|\cdot\|$ is the $L^1$ norm.
\end{Definition}

The target space of this operation is an integral domain, and its ring of fractions is a subfield of $K(x)((t))$. In particular,
\begin{equation*}
 \frac{F^t_L\big(x^2,q\big)}{F^t_L\big(x^{2r},q^r\big)} \in \C\big(\big(q^{1/p}\big)\big)(x)((t)).
\end{equation*}
The Conjecture~\ref{conj:FADO2} then states that \smash{$\lim_{q\rightarrow {\rm e}^{{2\pi {\rm i}}/{r}}}$} takes it to an element in
\[
\C(x)(t)'\subset \C(x)(t)\subset \C(x)((t)),
\]
and \smash{$\lim_{t\rightarrow 1}$} then takes it further to
\begin{equation*}
 \xi^{-\frac{1}{2}\alpha^TB\alpha+\frac{(r-1)^2 \Tr B}{2}}\N_r(L_\alpha) \in \C(x)'\subset \C(x).
\end{equation*}

\begin{Theorem}
\label{thm:CGP-Zhat}
Let $M=S^3(L)$ and $W_L(x,q)$ be such that Conjecture~$\ref{conj:FADO2}$, part $(a)$, holds with the following additional assumptions:
\begin{enumerate}[label=$(\roman*)$]\itemsep=0pt
 \item $\Lo_\omega W_L\big(x^2,q\big)\prod_{I}\big(x_I-x_I^{-1}\big)$ is well-defined and \smash{$\lim_{t\rightarrow 1}\lim_{q\rightarrow {\rm e}^{\frac{2\pi {\rm i}}{r}}}=\lim_{q\rightarrow {\rm e}^{\frac{2\pi {\rm i}}{r}}}\lim_{t\rightarrow 1}$} when applied to the ratio
 \begin{equation*}
 \frac{\Lo_\omega W^t_L\big(x^2,q\big)\prod_{I}\big(x_I-x_I^{-1}\big)}{\big(\Lo_\omega|_{r=1} W^t_L\big(x^2,q\big)\prod_{I}\big(x_I-x_I^{-1}\big)\big)|_{q\rightarrow q^r}},
 \end{equation*}
 that is the maps in the diagram~\eqref{lim-transform-diagram} commute when restricted to this element in the top left corner.
 \item $\exists \alpha(q,t)\in \C\big(\big(q^{1/p}\big)\big)((t))$ such that $\lim_{q\rightarrow {\rm e}^{2\pi {\rm i}}}\alpha(q,t)W^t_L\big(x^2,q\big)\in \C[x^\pm]((t))$ exists and is non-zero.
\end{enumerate}
Then there exist $W_\mathfrak{s}(q)\in 2^{-c}q^{\Delta}\Z[[q]]$, $\mathfrak{s}\in \mathrm{Spin}^c(M)$ such that
\begin{equation*}
 \N_r(M,\omega)=
\lim_{q\rightarrow {\rm e}^{\frac{2\pi {\rm i}}{r}}}
\frac{\sum\limits_{b\in H_1(M;\Z)}C^r_{\omega,b}W_{\sigma(b,s)}(q)}{\sum\limits_{b\in H_1(M;\Z)}{\rm e}^{2\pi {\rm i} \omega(b)}W_{\sigma(b,s)}(q^r)},
\end{equation*}
where\vspace{1mm}
\begin{align*}
C^r_{\omega,b}:=
\frac{1}{|H_1(M;\Z)|}
\begin{cases}
 \sum\limits_{a,f\in H_1(M;\Z)}
 {\rm e}^{2\pi {\rm i}\left(-\frac{r-1}{4}\lk(a,a)
 +\lk(a,f-b)-\frac{1}{2}\omega(a)
 +\lk(f,f)-\frac{1}{4}\mu(M,s)+\frac{1}{2}\right)},
 \\[-3mm]\hspace{87mm}
 r=1 \bmod 4,
 \\
 |H_1(M;\Z)|^{1/2}\sum\limits_{a\in H_1(M;\Z)}
 {\rm e}^{-\frac{\pi {\rm i}r}{2} q_s(a)-2\pi {\rm i} \lk(a,b)-\pi {\rm i} \omega(a)},
\ \ \,
 r=2 \bmod 4,
 \\[4mm]
 \sum\limits_{a,f\in H_1(M;\Z)}
 {\rm e}^{2\pi {\rm i}\left(-\frac{r+1}{4}\lk(a,a)
 -\lk(a,f+b)-\frac{1}{2}\omega(a) -\lk(f,f)
+\frac{1}{4}\mu(M,s) +\frac{1}{2} \right)},
 \\[-3mm]\hspace{87mm}
 r=3 \bmod 4.
 \end{cases}
\end{align*}
\end{Theorem}

\begin{proof}
The Conjecture~\ref{conj:FADO2}, part $(a)$, states that
\begin{equation*}
 \xi^{-\frac{1}{2}\alpha^TB\alpha+\frac{(r-1)^2 \Tr B}{2}}\N_r(L_\alpha)=
 \lim_{t\rightarrow 1}
 \lim_{q\rightarrow {\rm e}^\frac{2\pi {\rm i}}{r}}\frac{W_L^t\big(x^2;q\big)}{W^t_L\big(x^{2r};q^r\big)}.
\end{equation*}
Multiplying both sides by $\prod_I\big(x_I-x_I^{-1}\big)/\big(x_I^r-x_I^{-r}\big)$ and applying $\Go_\omega$ we have
\begin{gather}
 \Go_\omega\bigg(\xi^{-\frac{1}{2}\alpha^TB\alpha+\frac{(r-1)^2 \Tr B}{2}}\N_r(L_\alpha)
 \prod_I\frac{\big(x_I-x_I^{-1}\big)}{\big(x_I^r-x_I^{-r}\big)}\bigg)\nonumber
 \\ \qquad
 {}= \lim_{t\rightarrow 1}\Go_\omega
 \lim_{q\rightarrow {\rm e}^\frac{2\pi {\rm i}}{r}}\frac{W^t_L\big(x^2,q\big)\prod_I\big(x_I-x_I^{-1}\big)}{W^t_L\big(x^{2r},q^r\big)\prod_I\big(x_I^r-x_I^{-r}\big)},
\label{thm-CGP-Zhat-intermediate-eq}
\end{gather}
where we could bring $\Go_\omega$ inside the limit since its definition involves taking a finite sum of evaluations of the rational functions in $x$ that appear in the coefficients of the series. By~Proposition~\ref{prop:Gauss-omega-CGP}, the left-hand side of the equation~\eqref{thm-CGP-Zhat-intermediate-eq} gives the left-hand side of the equation in the statement of the theorem, up to a simple factor.
In the right-hand side, inside the limit \smash{$\lim_{t\rightarrow 1}$} we have
\begin{align}
\Go_\omega &\lim_{q\rightarrow {\rm e}^\frac{2\pi {\rm i}}{r}}\frac{W_L^t\big(x^2,q\big)\prod_I\big(x_I-x_I^{-1}\big)}{W_L^t\big(x^{2r};q^r\big)\prod_I\big(x_I^r-x_I^{-r}\big)}
 = \Go_\omega \frac{\lim_{q\rightarrow {\rm e}^\frac{2\pi {\rm i}}{r}}W_L^t\big(x^2,q\big)\alpha(q^r,t)\prod_I\big(x_I-x_I^{-1}\big)}{\lim_{q\rightarrow {\rm e}^\frac{2\pi {\rm i}}{r}}W_L^t\big(x^{2r},q^r\big)\alpha(q^r,t)\prod_I\big(x_I^r-x_I^{-r}\big)}\nonumber
 \\
 &=
 \frac{\lim_{q\rightarrow {\rm e}^\frac{2\pi {\rm i}}{r} }\Lo_\omega W_L^t\big(x^2,q\big)\alpha(q^r,t)\prod_I\big(x_I-x_I^{-1}\big)}{\lim_{q\rightarrow {\rm e}^{2\pi {\rm i}}}W_L^t\big({\rm e}^{2\pi {\rm i}\mu},q\big)\alpha(q,t)\prod_I\big({\rm e}^{\pi {\rm i}\mu_I}-{\rm e}^{-\pi {\rm i}\mu_I}\big)},
 \label{thm-rhs-before-limit-exchange}
\end{align}
where in the first equality we used the the assumption (ii) of the theorem. In the second equality we used the results of the Propositions~\ref{prop:gauss-sum-factor-out} and~\ref{prop:Go-Lo-in-limit}.

Using Proposition~\ref{prop:Go-Lo-in-limit} for $r=1$, we have
\begin{gather*}
\xi^{\frac{1}{2}\mu^TB\mu}\lim_{q\rightarrow {\rm e}^{2\pi {\rm i}}}
W_L^t\big({\rm e}^{2\pi {\rm i}\mu},q\big)\alpha(q,t)\prod_I\big({\rm e}^{\pi {\rm i}\mu_I}-{\rm e}^{-\pi {\rm i}\mu_I}\big)
\\ \qquad
{}= \lim_{q\rightarrow {\rm e}^{2\pi {\rm i}}}\Go_\omega|_{r=1}W_L^t\big(x^2,q\big)\alpha(q,t)\prod_I\big(x_I-x_I^{-1}\big)
\\ \qquad
{}= \lim_{q\rightarrow {\rm e}^{2\pi {\rm i}}}\Lo_\omega|_{r=1}W_L^t\big(x^2,q\big)\alpha(q,t)\prod_I\big(x_I-x_I^{-1}\big).
\end{gather*}
It follows that the right-hand side of~\eqref{thm-CGP-Zhat-intermediate-eq} is equal to
\begin{gather}
\xi^{\frac{1}{2}\mu^TB\mu}
\lim_{t\rightarrow 1}\lim_{q\rightarrow {\rm e}^\frac{2\pi {\rm i}}{r}}
\frac{\Lo_\omega W^t_L\big(x^2,q\big)\prod_{I}\big(x_I-x_I^{-1}\big)}{\big(\Lo_\omega|_{r=1} W^t_L\big(x^2,q\big)\prod_{I}\big(x_I-x_I^{-1}\big)\big)\big|_{q\rightarrow q^r}}\nonumber
\\ \qquad
{}= \xi^{\frac{1}{2}\mu^TB\mu}
 \lim_{q\rightarrow {\rm e}^\frac{2\pi {\rm i}}{r}}
 \lim_{t\rightarrow 1}
 \frac{\Lo_\omega W^t_L\big(x^2,q\big)\prod_{I}\big(x_I-x_I^{-1}\big)}{\big(\Lo_\omega|_{r=1} W^t_L\big(x^2,q\big)\prod_{I}\big(x_I-x_I^{-1}\big)\big)\big|_{q\rightarrow q^r}}\nonumber
 \\ \qquad
{}= \xi^{\frac{1}{2}\mu^TB\mu}
 \lim_{q\rightarrow {\rm e}^\frac{2\pi {\rm i}}{r}}
 \frac{\Lo_\omega W_L\big(x^2,q\big)\prod_{I}\big(x_I-x_I^{-1}\big)}{\big(\Lo_\omega|_{r=1} W_L\big(x^2,q\big)\prod_{I}\big(x_I-x_I^{-1}\big)\big)\big|_{q\rightarrow q^r}},
 \label{thm-rhs-after-limit-exchange}
\end{gather}
where we used the assumption (i).

Let $W_\mathfrak{s}(q)$ be defined through $W_L(x,q)$ in the same way as $\hat{Z}_\mathfrak{s}(q)$ is defined through $F_L(x,q)$, i.e., surgery formula~\eqref{Zhat-surgery-RHS}. From Proposition~\ref{prop:Laplace-omega-Zhat} with $r=1$, we then have
\begin{gather}
 \Lo_\omega|_{r=1}W_L\big(x^2,q\big)\prod_I\big(x_I-x_I^{-1}\big)\nonumber
 \\ \qquad
 {}=\frac{(-1)^{b_+} {\rm e}^{\frac{\pi {\rm i}}{2}(3\sigma-\Tr B)}}{|H_1(M;\Z)|}
 \sum\limits_{a,b,f\in H_1(M;\Z)}
 {\rm e}^{2\pi {\rm i}\left(\lk(a,f-b)-\frac{1}{2}\omega(a) +\lk(f,f)
 -\frac{1}{4}\mu(M,s)+\frac{1}{2} \right)} W_{\sigma(b,s)}\nonumber
 \\ \qquad
{}= (-1)^{b_+} {\rm e}^{\frac{\pi {\rm i}}{2}(3\sigma-\Tr B)} {\rm e}^{\frac{\pi {\rm i}}{2} \mu^T B\mu} \sum\limits_{b\in H_1(M;\Z)} {\rm e}^{2\pi {\rm i}\left(\omega(b)
 -\frac{1}{4}\mu(M,s)+\lk(b,b)+\frac{1}{2}\right)} W_{\sigma(b,s)}\nonumber
 \\ \qquad
{}= (-1)^{b_+} {\rm e}^{\frac{\pi {\rm i}}{2}(3\sigma-\Tr B)} {\rm e}^{\frac{\pi {\rm i}}{2} \mu^T B\mu} \sum\limits_{b\in H_1(M;\Z)}
 {\rm e}^{2\pi {\rm i}\omega(b)} W_{\sigma(b,s)}\Big|_{q\rightarrow q {\rm e}^{-2\pi {\rm i}}},
 \label{thm-CGP-Zhat-r1-implication}
\end{gather}
where we have used the fact that, according to the formula~\eqref{Zhat-surgery-RHS}, the overall rational power shift of $q$-series $W_{\sigma(b,s)}$ modulo 1 is given by
\begin{equation*}
 \frac{3\sigma-\Tr B}{4}-\frac{(s-\varepsilon)^TB(s-\varepsilon)}{4}-b^TB^{-1}b=
 \frac{1}{2}+\frac{\mu(M,s)}{4}-\lk(b,b)\bmod 1.
\end{equation*}

Using~\eqref{thm-CGP-Zhat-r1-implication} in the right-hand side of~\eqref{thm-rhs-after-limit-exchange} and also using the Proposition~\ref{prop:Laplace-omega-Zhat} for general $r$, we conclude that the right-hand side of~\eqref{thm-CGP-Zhat-intermediate-eq} gives the the right-hand side of the equation in the statement of the theorem, up to a simple phase factor. Taking care of the phase factors on both sides of~\eqref{thm-CGP-Zhat-intermediate-eq} concludes the proof of the first part of the statement of the theorem.

If Conjecture~\ref{conj:FADO2}, part $(b)$, holds, we take $W_L(x,q)=F_L(x,q)$, and then by construction $W_\mathfrak{s}(q)=\hat{Z}_\mathfrak{s}(q)$.
\end{proof}

\begin{Theorem}
\label{thm:CGP-Zhat-torsion}
Let $M=S^3(L)$ and $W_L(x,q)$ be such that Conjectures~$\ref{conj:FADO2}$ and~$\ref{conj:FADO3}$, parts~$(a)$, hold with the following additional assumption:
\begin{enumerate}[label=$(\roman*)$]\itemsep=0pt
 \item $\Lo_\omega W^t_L\big(x^2,q\big)\prod_{I}\big(x_I-x_I^{-1}\big)$ is well defined and \smash{$\lim_{t\rightarrow 1}\lim_{q\rightarrow {\rm e}^{\frac{2\pi {\rm i}}{r}}}=\lim_{q\rightarrow {\rm e}^{\frac{2\pi {\rm i}}{r}}}\lim_{t\rightarrow 1}$} when applied to
 \begin{equation*}
 \Lo_\omega W^t_L\big(x^2,q\big)\prod_{I}\big(x_I-x_I^{-1}\big)
 \end{equation*}
 that is the maps in the diagram~\eqref{lim-transform-diagram} commute when restricted to this element in the top left corner.
\end{enumerate}
Then there exist $W_\mathfrak{s}(q)\in 2^{-c}q^{\Delta}\Z[[q]]$, $\mathfrak{s}\in \mathrm{Spin}^c(M)$ such that\begin{equation*}
\N_r(M,\omega)
=\mathcal{T}(M,[\omega])
\lim_{q\rightarrow {\rm e}^\frac{2\pi {\rm i} }{r}}\sum_{b\in H_1(M;\Z)}C^r_{\omega,b} W_{\sigma(b,s)},
\end{equation*}
where $C^r_{\omega,b}$ are the same as in Theorem~$\ref{thm:CGP-Zhat}$.
Moreover, if the parts $(b)$ of Conjectures~$\ref{conj:FADO2}$ and~$\ref{conj:FADO3}$ also hold and $F_L(x,q)$ satisfies the assumption $(i)$, then one can take $W_\mathfrak{s}(q)=\hat{Z}_\mathfrak{s}(q)$.
\end{Theorem}

\begin{proof}
One can follow the proof of the Theorem~\ref{thm:CGP-Zhat} since, from the Conjecture~\ref{conj:FADO3}, the assumption $(iv)$ of that theorem automatically holds with $\alpha(q,t)=1$. Moreover, the Conjecture~\ref{conj:FADO3} provides a relation between \smash[b]{$\lim_{t\rightarrow 1}\lim_{q\rightarrow 1} W^t_L\big(x^2,q\big)$} which appears in the denominator of~\eqref{thm-rhs-before-limit-exchange} and the Alexander--Conway function.
Using the surgery formula for the torsion (see Appendix~\ref{app:torsion}), we arrive at the statement of the theorem.
\end{proof}

\begin{Remark}
In the case of plumbing surgeries, the assumption (ii) holds if and only if the plumbing is \textit{weakly negative definite}, meaning the inverse of the linking matrix, $B^{-1}$, restricted on the vertices of degree $>2$ is negative definite (cf.~\cite{Gukov:2019mnk}).
\end{Remark}
\begin{Example}\label{ex:theoremworks}
Let $M$ be a rational homology sphere obtained by surgery over a plumbing link corresponding to a ``$Y$-shaped'' plumbing graph i.e. one formed by a single trivalent vertex corresponding to an unknot with strictly negative framing, three $1$-valent vertices and some $2$-valent vertices.
Then as shown in Appendix~\ref{app:commutlimits} the  hypothesis of Theorem~\ref{thm:CGP-Zhat-torsion} is satisfied.
So this provides an infinite family of examples in which Theorem~\ref{thm:CGP-Zhat-torsion} holds. This family of examples overlaps with the ones considered in~\cite{MR4400935,fuji2021witten,Gukov:2019mnk,Gukov:2016njj}.
\end{Example}

\subsection[Generalization to $b_1\geq 0$]{Generalization to $\boldsymbol{b_1\geq 0}$}

In this section, we briefly list modifications one needs to do in Section~\ref{sec:gauss-vs-laplace} in order to generalize the results to the case of general $b_1\geq 0$.
We will follow the conventions of Section~\ref{sec:Zhat-def-b1}. In~particular, we fix $U\in {\rm SL}(V,\Z)$ such that
\begin{equation*}
 UBU^{T}=
 \begin{pmatrix}
 B' & 0 \\
 0 & 0
 \end{pmatrix}\!,
\end{equation*}
where $\det B'\neq 0$ and use the following notations:
\begin{equation*}
\begin{pmatrix} \ell' \\ \ell'' \end{pmatrix} :=U\ell,\qquad
\begin{pmatrix} \mu' \\ \mu'' \end{pmatrix} :=\big(U^T\big)^{-1}\mu,\qquad
 \begin{pmatrix} s' \\ s'' \end{pmatrix} :=\big(U^T\big)^{-1}s,\qquad
\begin{pmatrix} \varepsilon' \\ \varepsilon'' \end{pmatrix} :=\big(U^T\big)^{-1}\varepsilon.
\end{equation*}

We redefine $p:=4|\det B'|$.
Definitions~\ref{def:t-limit-rational},~\ref{def:gauss-omega},~\ref{def:q-limit},
\ref{def:t-limit},~\ref{def:t-reg} and Propositions~\ref{prop:gauss-sum-factor-out},~\ref{prop:Gauss-omega-CGP} do not need to be modified. Definition~\ref{def:Laplace-omega} is generalized to

\begin{Definition}
Define the ``$\omega$-twisted Laplace transform'' on a subspace of $\C\big(\big(q^{1/p}\big)\big)[[x^\pm]]$ by the following formula:
\begin{align*}
 \Lo_\omega\colon\ \C\big(\big(q^{1/p}\big)\big)\big[\big[x^\pm\big]\big] & \longrightarrow \C\big(\big(q^{1/p}\big)\big),
 \\[2mm]
 \sum\limits_{\substack{\ell\in \Z^\vert \\ m\in \Z}} A_{\ell,m} q^{\frac{m}{p}} x^\ell &
 \longmapsto
 \sum\limits_{n\in \Z} q^{\frac{n}{p}} \sum\limits_{\substack{\ell\in \Z^\vert \\ m\in \Z \\ 4m-p(\ell')^T(B')^{-1}\ell'+2p\sum_I|\ell''_I||\varepsilon''_I|=4n}} A_{\ell,m}\cdot \mathcal{C}_\ell^\omega,
\end{align*}
where
\begin{align*}
 \mathcal{C}_\ell^\omega={}&\frac{{\rm e}^{\frac{\pi {\rm i}\sigma}{4}}(r/2)^{
 \frac{V-b_1}{2}} r^{b_1}}{|\det B'|^{1/2}}
 {\rm e}^{\pi {\rm i}(\ell'')^T(\epsilon''+\frac{\mu''}{r})}
 \delta_{\ell''=0\bmod r}
 \\
&\times \sum_{\tilde{a}\in \Z^{V-b_1} /2B'\Z^{V-b_1}}
 {\rm e}^{\pi {\rm i}\left(-\frac{r}{2} \tilde{a}^T (B')^{-1}\tilde{a}-\tilde{a}^T (B')^{-1}(\ell'+B'(\mu'+(r-1)\varepsilon'))
 \right)}.
\end{align*}
We say that $\Lo_\omega$ is well defined if the interior sum in the right-hand side of~\eqref{Laplace-omega-formula} has finite number of non-zero terms.
\label{def:Laplace-omega-b1}
\end{Definition}
With such modified definitions Proposition~\ref{prop:Go-Lo-in-limit}, then still holds by the similar argument. Propositions~\ref{prop:Laplace-omega-explicit} and~\ref{prop:Laplace-omega-Zhat} are modified respectively to the following two:

\begin{Proposition}
 When $r\neq 0\bmod 4$ and
 \begin{gather*}
 \ell' = 2b'+B'(s'-\varepsilon')\bmod 2B'\Z^{V-b_1},
 \\
 \ell'' = \mathrm{LCM}(r,2)m\in \Z^{b_1},
 \end{gather*}
 where $b'$ and $s$ represent elements of $\Tor H_1(M;\Z)$ and $\mathrm{Spin}(M)$, respectively, the coefficients~$\mathcal{C}^\omega_\ell$ in the Definition~$\ref{def:Laplace-omega-b1}$ admit the following expression:
 \begin{align*}
 \mathcal{C}_\ell^\omega={}&\frac{{\rm e}^{\frac{\pi {\rm i}\sigma}{4}}r^\frac{V+b_1}{2}}{|\Tor H_1(M;\Z)|} {\rm e}^{\frac{2\pi {\rm i}\omega''(m)}{\mathrm{GCD}(r,2)}}
 \\
&\times \begin{cases}
 {\rm e}^{\frac{\pi {\rm i}}{4}(5\sigma-2\Tr B)}
 \sum\limits_{a,f\in \Tor H_1(M;\Z)}
 {\rm e}^{2\pi {\rm i}\left( -\frac{r-1}{4}\lk(a,a)
 +\lk(a,f-b)-\frac{1}{2}\omega(a) +\lk(f,f)
 -\frac{1}{4}\mu(M,s)+\frac{1}{2} \right)},
 \\[-3mm]\hspace{109mm}
 r=1 \bmod 4,
 \\
 2^{\frac{V-b_1}{2}}
 |\Tor H_1(M;\Z)|^{1/2}\sum\limits_{a\in \Tor H_1(M;\Z)}
 {\rm e}^{-\frac{\pi {\rm i}r}{2} q_s(a)-2\pi {\rm i} \lk(a,b)-\pi {\rm i} \omega(a)},\ \ \
 r=2 \bmod 4,
 \\[3mm]
 {\rm e}^{\frac{\pi {\rm i}}{4}(-5\sigma+2\Tr B)}
 \sum\limits_{a,f\in \Tor H_1(M;\Z)}
 {\rm e}^{2\pi {\rm i}\left( -\frac{r+1}{4}\lk(a,a)
 -\lk(a,f+b)-\frac{1}{2}\omega(a) -\lk(f,f)
 +\frac{1}{4}\mu(M,s)+\frac{1}{2} \right)},
 \\[-3mm]\hspace{109mm}
 r=3 \bmod 4.
 \end{cases}\!\!\!\!
 \end{align*}
 \label{prop:Laplace-omega-explicit-b1}
\end{Proposition}

\begin{proof}
It is contained in Section~\ref{sec:plumbed}.
\end{proof}

\begin{Proposition}\label{prop:Laplace-omega-Zhat-b1}
\begin{align*}
\Lo_\omega&\bigg(F_L\big(x^2,q\big)\cdot \prod_{I\in \vert}\big(x_I-x^{-1}_I\big)\bigg)=
\frac{{\rm e}^{\frac{\pi {\rm i}\sigma}{4}}r^{\frac{V+b_1}{2}}}{|\Tor H_1(M;\Z)|}
\sum_{\substack{b\in\Tor H_1(M;\Z)\\ m\in \Z^{b_1}}}{\rm e}^{\frac{2\pi {\rm i}\omega''(m)}{\mathrm{GCD}(r,2)}}
\\
&\times \begin{cases}
 {\rm e}^{\frac{\pi {\rm i}}{4}(5\sigma-2\Tr B)}
 \sum\limits_{a,f\in \Tor H_1(M;\Z)}
 {\rm e}^{2\pi {\rm i}\left( -\frac{r-1}{4}\lk(a,a)
 +\lk(a,f-b)-\frac{1}{2}\omega(a) +\lk(f,f)
 -\frac{1}{4}\mu(M,s)+\frac{1}{2} \right)},
\\[-3mm]\hspace{110mm}
 r=1 \bmod 4,
 \\
 2^{\frac{V-b_1}{2}}
 |\Tor H_1(M;\Z)|^{1/2}\sum\limits_{a\in \Tor H_1(M;\Z)}
 {\rm e}^{-\frac{\pi {\rm i}r}{2} q_s(a)-2\pi {\rm i} \lk(a,b)-\pi {\rm i} \omega(a)},\quad
 r=2 \bmod 4 ,
 \\[3mm]
{\rm e}^{\frac{\pi {\rm i}}{4}(-5\sigma+2\Tr B)}
 \sum\limits_{a,f\in \Tor H_1(M;\Z)}
 {\rm e}^{2\pi {\rm i}\left(-\frac{r+1}{4}\lk(a,a)
 -\lk(a,f+b)-\frac{1}{2}\omega(a) -\lk(f,f)
 +\frac{1}{4}\mu(M,s)+\frac{1}{2}\right)},
 \\[-3mm]\hspace{110mm}
 r=3 \bmod 4
 \end{cases}
 \\
&\times (-1)^{b_+}q^{\frac{\Tr B-3\sigma}{4}} \hat{Z}_{\sigma(b\oplus \frac{mr}{\mathrm{GCD}(r,2)},s)}
\end{align*}
when the left-hand side is well defined.
\end{Proposition}
\begin{proof}
 Follows from the definition of $\Lo_\omega$, Proposition~\ref{prop:Laplace-omega-explicit-b1} and the conditional definition of $\hat{Z}_{\sigma(b,s)}$ through $F_L$ (equation~\eqref{Zhat-b1-def}).
\end{proof}
\begin{Remark}
 In Proposition~\ref{prop:Laplace-omega-Zhat-b1} and the theorems below, $\hat{Z}_{\sigma(b'\oplus b'',s)}$ is understood as a~parti\-cular representative in \smash{$2^{-c} q^{\Delta}\Z[[q]]\subset\C\big(\big(q^{1/p}\big)\big)$}, rather then an element of the quotient over the subspace \smash{$\big(1-q^{\mathrm{LCM}(2,\mathrm{GCD}(b''))}\big)\C\big(\big(q^{1/p}\big)\big)$}. This representative is fixed by the choice of the surgery link $L$ in the definition of $\hat{Z}$ via $F_L$ by the formula~\eqref{Zhat-b1-def}.
\end{Remark}

Theorems~\ref{thm:CGP-Zhat} and~\ref{thm:CGP-Zhat-torsion} are then modified respectively to the following two, with proofs following similar arguments:

\begin{Theorem}
\label{thm:CGP-Zhat-b1}
Let $M=S^3(L)$ and $W_L(x,q)$ be such that Conjecture~$\ref{conj:FADO2}$, part $(a)$, holds with the following additional assumptions:
\begin{enumerate}[label=$(\roman*)$]\itemsep=0pt
 \item $\Lo_\omega W_L\big(x^2,q\big)\prod_{I}\big(x_I-x_I^{-1}\big)$ is well-defined and \smash{$\lim_{t\rightarrow 1}\lim_{q\rightarrow {\rm e}^{\frac{2\pi {\rm i}}{r}}}=\lim_{q\rightarrow {\rm e}^{\frac{2\pi {\rm i}}{r}}}\lim_{t\rightarrow 1}$} when applied to the ratio
 \begin{equation*}
 \frac{\Lo_\omega W^t_L\big(x^2,q\big)\prod_{I}\big(x_I-x_I^{-1}\big)}{\left(\Lo_\omega|_{r=1} W^t_L\big(x^2,q\big)\prod_{I}\big(x_I-x_I^{-1}\big)\right)|_{q\rightarrow q^r}},
 \end{equation*}
 that is the maps in the diagram~\eqref{lim-transform-diagram} commute when restricted to this element in the top left corner.
 \item $\exists \alpha(q,t)\in \C\big(\big(q^{1/p}\big)\big)((t))$ such that $\lim_{q\rightarrow {\rm e}^{2\pi {\rm i}}}\alpha(q,t)W^t_L\big(x^2,q\big)\in \C[x^\pm]((t))$ exists and is non-zero.
\end{enumerate}
Then there exist $W_\mathfrak{s}(q)\in 2^{-c}q^{\Delta}\Z[[q]]$, $\mathfrak{s}\in \mathrm{Spin}^c(M)$ such that
\begin{equation*}
 \N_r(M,\omega)=
\lim_{q\rightarrow {\rm e}^{\frac{2\pi {\rm i}}{r}}}
\frac{\sum\limits_{\substack{b\in \Tor H_1(M;\Z)\\ m\in \Z^{b_1} }}C^r_{\omega,b,m}W_{\sigma(b\oplus \frac{mr}{\mathrm{GCD}(r,2)},s)}}{\sum\limits_{b\in H_1(M;\Z)}{\rm e}^{2\pi {\rm i} \omega(b)}W_{\sigma(b,s)}(q^r)},
\end{equation*}
where
\begin{align*}
C^r_{\omega,b,m}:={}&
\frac{r^{b_1} {\rm e}^{\frac{2\pi {\rm i}\omega''(m)}{\mathrm{GCD}(r,2)}}}{|\Tor H_1(M;\Z)|}
 \\
&\times \begin{cases}
 \sum\limits_{a,f\in \Tor H_1(M;\Z)}
 {\rm e}^{2\pi {\rm i}\left( -\frac{r-1}{4}\lk(a,a)
 +\lk(a,f-b)-\frac{1}{2}\omega(a) +\lk(f,f)
 -\frac{1}{4}\mu(M,s)+\frac{1}{2} \right)},
 \\[-3mm]\hspace{100mm}
 r=1 \bmod 4,
 \\
 |\Tor H_1(M;\Z)|^{1/2}\sum\limits_{a\in \Tor H_1(M;\Z)}
 {\rm e}^{-\frac{\pi {\rm i}r}{2} q_s(a)-2\pi {\rm i} \lk(a,b)-\pi {\rm i} \omega(a)},\quad
 r=2 \bmod 4,
 \\[2mm]
 \sum\limits_{a,f\in \Tor H_1(M;\Z)}
 {\rm e}^{2\pi {\rm i}\left(-\frac{r+1}{4}\lk(a,a)
 -\lk(a,f+b)-\frac{1}{2}\omega(a) -\lk(f,f)
 +\frac{1}{4}\mu(M,s)+\frac{1}{2}\right)},
 \\[-3mm]\hspace{100mm}
 r=3 \bmod 4.
 \end{cases}\!\!\!\!
\end{align*}
Moreover, if the parts $(b)$ of Conjectures~$\ref{conj:FADO2}$ and~$\ref{conj:FADO3}$ also hold and $F_L(x,q)$ satisfies the assumption~$(i)$, then one can take $W_\mathfrak{s}(q)=\hat{Z}_\mathfrak{s}(q)$.
\end{Theorem}

\begin{Theorem}
\label{thm:CGP-Zhat-torsion-b1}
Let $M=S^3(L)$ and $W_L(x,q)$ be such that Conjectures~$\ref{conj:FADO2}$ and~$\ref{conj:FADO3}$, parts~$(a)$, hold with the following additional assumption:
\begin{enumerate}[label=$(\roman*)$]\itemsep=0pt
 \item $\Lo_\omega W^t_L\big(x^2,q\big)\prod_{I}\big(x_I-x_I^{-1}\big)$ is well defined and \smash{$\lim_{t\rightarrow 1}\lim_{q\rightarrow {\rm e}^{\frac{2\pi {\rm i}}{r}}}=\lim_{q\rightarrow {\rm e}^{\frac{2\pi {\rm i}}{r}}}\lim_{t\rightarrow 1}$} when applied to
 \begin{equation*}
 \Lo_\omega W^t_L\big(x^2,q\big)\prod_{I}\big(x_I-x_I^{-1}\big)
 \end{equation*}
 that is the maps in the diagram~\eqref{lim-transform-diagram} commute when restricted to this element in the top left corner.
\end{enumerate}
Then there exist $W_\mathfrak{s}(q)\in 2^{-c}q^{\Delta}\Z[[q]]$, $\mathfrak{s}\in \mathrm{Spin}^c(M)$ such that
\begin{equation*}
\N_r(M,\omega)
=\mathcal{T}(M,[\omega])
\lim_{q\rightarrow {\rm e}^\frac{2\pi {\rm i} }{r}}\sum\limits_{\substack{b\in H_1(M;\Z)\\ m\in \Z^{b_1} }}C^r_{\omega,b,m}W_{\sigma(b\oplus \frac{mr}{\mathrm{GCD}(r,2)},s)},
\end{equation*}
where $C^r_{\omega,b,m}$ are the same as in Theorem~$\ref{thm:CGP-Zhat-b1}$.
Moreover, if the parts $(b)$ of Conjectures~$\ref{conj:FADO2}$ and~$\ref{conj:FADO3}$ also hold and $F_L(x,q)$ satisfies the assumption $(i)$, then one can take $W_\mathfrak{s}(q)=\hat{Z}_\mathfrak{s}(q)$.
\end{Theorem}

\begin{Example}\label{ex:theoremworks-b1}
Let $M$ be obtained by surgery over a plumbing link corresponding to a ``$Y$-shaped'' plumbing graph, i.e., one formed by a single trivalent vertex corresponding to an unknot with strictly negative framing, three $1$-valent vertices and some $2$-valent vertices.
Then, as shown in Appendix~\ref{app:commutlimits}, under a certain open condition on the linking matrix, the first hypothesis of Theorem~\ref{thm:CGP-Zhat-torsion-b1} is satisfied. Hypothesis $(iii)$ is satisfied as shown in Section~\ref{sub:goodlinks} and $\hat{Z}$ is defined via $F_L$. So this provides an infinite family of examples in which Theorem~\ref{thm:CGP-Zhat-torsion-b1} holds.
Note, the family of 3-manifolds in this example overlaps with families treated in connection with WRT invariants in~\cite{MR4400935,fuji2021witten,Gukov:2019mnk,Gukov:2016njj}.

\end{Example}

\subsection{Generalization to the spin case}\label{sec:spincase}

In this section, we briefly mention the relation between $\hat{Z}$ and the spin version of the $\N_r$ invariant for $r=0\bmod 4$ which was defined in~\cite{blanchet2014non}. We will omit any details and intermediate calculations, as they are analogous to the non-spin case. The spin version of the invariant depends on the choice of $\C/2\Z$-spin structure, which can be understood as a homotopy class of lifts
\begin{equation}
 \begin{tikzcd}
 & B\Spin(3,\C/2\Z)\ar[d] \\
 M \ar[ur,dashed]\ar[r] &
 B\operatorname{SO}(3),
 \end{tikzcd}
 \label{spincmod2z-structure}
\end{equation}
where $\Spin(3,\C/2\Z)$ is an extension of $\operatorname{SO}(3)$ by $\C/2\Z$ equipped with discrete topology:
\begin{equation*}
 (\C/2\Z)_\text{discrete}\longrightarrow \Spin(3,\C/2\Z)\equiv \Spin(3)\times_{\Z_2}(\C/2\Z)_\text{discrete}
 \longrightarrow \operatorname{SO}(3),
\end{equation*}
the vertical map in~\eqref{spincmod2z-structure} corresponds to the projection, and the horizontal map is the classifying map of the bundle of orthonormal frames in the tangent bundle,
which are non-canonically parametrized by $H^1(M;\C/2\Z)$. To formulate the relation, it will be useful to consider the canonical map
\begin{equation*}
 \tilde{\sigma}\colon \ H^1(M;\C/2\Z)\times \Spin(M) \longrightarrow \Spin(M,\C/2\Z)
\end{equation*}
similar to the map $\sigma$. It is induced by the $\Z_2$ quotient map of the product $\Spin(3)\times(\C/2\Z)_\text{discrete}$, taking into account that $B(\C/2\Z)_\text{discrete}=K(\C/2\Z,1)$. Then, for general $b_1$ the relation for $r=0\bmod 4$ reads
\begin{gather*}
\begin{split}
 N^\Spin_r(M,\tilde{\sigma}(\omega,s))
= {}&\frac{(-1)^{\mu(M,s)} r^{b_1} \CT([\omega])}{\sqrt{|\Tor H_1(M;\Z)|}}
\\
 &\times\!\!\!\lim_{q\rightarrow {\rm e}^{\frac{2\pi {\rm i}}{r}}}
 \!\!\sum_{\substack{a,b\in \Tor H_1(M;\Z) \\ m\in \Z^{b_1}}}\!\!\!\!\!\!
 {\rm e}^{-\frac{\pi {\rm i} r}{2}\lk(a,a)-2\pi {\rm i}\lk(a,b)-\pi {\rm i}\omega(a)+\pi {\rm i}\omega''(m)} \hat{Z}_{\sigma(b\oplus \frac{mr}{2},s)}(M),
 \end{split}\!\!\!
\end{gather*}
where $\mu(M,s)$ is the Rokhlin invariant, for which we used the surgery formula~\eqref{Rokhlin-mod4}. Note that modulo two it is actually independent of spin structure $s$.
For rational homology spheres, the relation simplifies to
\begin{gather*}
 N^\Spin_r(M,\tilde{\sigma}(\omega,s))
 \\ \qquad
 {}= \frac{(-1)^{\mu(M,s)} \CT([\omega])}{\sqrt{|H_1(M;\Z)|}}
 \lim_{q\rightarrow {\rm e}^{\frac{2\pi {\rm i}}{r}}}
 \sum_{{a,b\in H_1(M;\Z)}}
 {\rm e}^{-\frac{\pi {\rm i} r}{2}\lk(a,a)-2\pi {\rm i}\lk(a,b)-\pi {\rm i}\omega(a)}
 \hat{Z}_{\sigma(b,s)}(M).
\end{gather*}

\section{\texorpdfstring{$\boldsymbol{\hat{Z}}$ and $\boldsymbol{\N_r}$}{CGP and GPPV invariants} as decorated TQFTs}\label{sec:TQFT}

Both $\hat{Z}$ and $\N_r$ are examples of topological invariants that, on the one hand, behave well under cutting and gluing (admit surgery formulae) and, on the other hand, depend on additional data (decoration).
In this section, which is of more speculative nature, we describe how this additional data should behave under cutting and gluing as well as its transformation under the operation of taking the limit $q \to {\rm e}^{2\pi {\rm i} / r}$ that conjecturally relates $\hat{Z}$ and $\N_r$ invariants.

\subsection{Hilbert space on a torus}

In table~\eqref{NZhatcomparison} we omitted the comparison of Hilbert spaces $\CH_{\text{BCGP}} \big(T^2\big)$ and $\CH_{\text{GPPV}} \big(T^2\big)$ in the two theories on a 2-torus.\footnote{BCGP stands for the TQFT built out of the CGP invariants by Blanchet, Costantino, Geer and Patureau~\cite{blanchet2016non}.} These spaces control surgery operations and deserve a separate section. Note, in a semisimple TQFT with a finite-dimensional Hilbert space, we have
\begin{equation*}
\CH \big(T^2\big) = K^0 \big( \text{MTC} \big(S^1\big) \big).
\end{equation*}
Neither of the two theories we are trying to compare fits into this standard paradigm of the Reshetikhin--Turaev construction, which nevertheless can be used as a good motivation.
In~particular, it is important that both BCGP and GPPV theories are ``decorated'' TQFTs, with additional structure $\omega$ or $\mathfrak{s} \in \operatorname{Spin}^c (M) / \Z_2$ originating from the equivariance under $\mathbb{T}_{\C} \subset G_{\C}$, the maximal torus of $G_{\C} = {\rm SL}(2,\C)$.

Taking into account these extra structures, the first order approximation to $\CH \big(T^2\big)$ in our theories is as follows. From cutting and gluing (surgery) rules, we infer that the space $\CH_{\text{GPPV}} \big(T^2\big)$ has basis $| n, x \rangle_a$ where $n \in \Z$, $x \in \C^*$ and $a$ is a relative spin$^c$ structure. It is often convenient to replace $x$ by a dual variable $m \in \Z$, so that the Weyl group of $G_{\C} = {\rm SL}(2,\C)$ maps $| n,m \rangle_a \mapsto |{-}n, -m \rangle_a$. Therefore, for each given spin$^c$ structure (= choice of background), we have
\begin{equation}
\CH_{\text{GPPV}}^{(a)} \big(T^2\big) \cong \C\bigg[\frac{\Lambda \times \Lambda^{\vee}}{W}\bigg],
\label{HGPPVnaive}
\end{equation}
where, as usual, by $\C[S]$ we denote the space of complex valued functions\footnote{In principle, one has to impose a certain asymptotic behavior condition on the functions on the lattice. This condition should be coherent with a certain continuity condition on the allowed functions on the dual space, related by Fourier transform. In this work we do not specify such conditions.} on set~$S$. Note that, compared to the Hilbert space $\CH_{\text{WRT}} \big(T^2\big) = \C[{\Lambda}/({W \times k \Lambda})]$, the space of states~\eqref{HGPPVnaive} has two copies of the lattice, i.e., corresponds to a {\it toroidal} algebra, and has no cut-off due to the level~\cite{Chun:2019mal,Gukov:2019mnk}. The latter property is, of course, anticipated {\it a priori} since the invariants $\hat Z_b (M;q)$ depend on $q$, generic with $|q|<1$. From the Kazhdan--Lusztig correspondence and the theory of $W$-algebras, it is also natural to attribute this ``doubling'' to a larger symmetry associated with the action of two (quantum) groups, such that~\eqref{HGPPVnaive} is basically a product of root lattices for $G$ and its Langlands dual ${}^L G$.

At a similar level of approximation, for a fixed choice of the ``decoration''/equivariant parameter $\alpha$, we have
\begin{equation*}
\CH_{\text{BCGP}}^{(\alpha)} \big(T^2\big) \cong \C[H_r],\qquad
H_r:=\{1-r,3-r,\ldots,r-1\},
\end{equation*}
which is a direct consequence of the corresponding statement for categories $\mathcal{C} = \oplus_{\bar \alpha \in \C / 2\Z} \mathcal{C}_{\bar \alpha}$~\cite{blanchet2016non}.
It is very well known that many affine algebras and VOAs at ``critical level'' have large center. The same is true about quantum groups at roots of unity. Very often, the center has a nice geometric meaning as the space of functions on some variety (moduli space) $\mathcal{M}$, and in the case at hand $\mathcal{M} = \mathbb{T}_{\C}$.
Below we refine these descriptions of $\CH_{\text{BCGP}} \big(T^2\big)$ and $\CH_{\text{GPPV}} \big(T^2\big)$ by looking at these spaces from various angles.

Before continuing with $\Sigma = T^2$, it is instructive to pause for a moment and consider simpler cases of $\Sigma = S^2$ or $S^2$ with two points removed. In a 3d TQFT associated to $\mathcal{C}$, objects in $\mathcal{C}$ correspond to line operators whereas their non-trivial extensions are represented by non-trivial junctions, i.e., local operators where two line operators meet, illustrated in Figure~\ref{fig:junction}. The space of such local operators is the space of states on $\Sigma = S^2 \setminus \{ p_1, p_2 \}$; in particular, it is non-trivial for BCGP theory indicating the non-semisimple nature of the TQFT. As a special case, the space of local operators at the junction of two trivial lines is simply $\CH \big(S^2\big)$, and for GPPV theory its structure is conveniently encoded in the unreduced\footnote{This is the version that includes the contribution of the Cartan component of the adjoint chiral, e.g., $\hat Z_0^{(\text{unred})} \big(S^3\big) = \frac{1}{(t^2 q^2;q)_{\infty}}$ for $G=\operatorname{SU}(2)$. The reduced version is obtained by removing this contribution, i.e., via multiplying by $(qt;q)_{\infty}$. Moreover, there is also a factor of $(1;q)_{\infty}$ in the numerator which requires regularization. Sometimes it is simply removed. And, sometimes only the zero mode is removed so that $(1;q)_{\infty}$ is replaced by $(q;q)_{\infty}$, as in~\eqref{S1S2refined}.} version of the $\hat Z$-invariant for~$S^1 \times S^2$:
\begin{equation}
\frac{1}{2} \frac{(x;q)_{\infty} (q;q)_{\infty} \big(x^{-1};q\big)_{\infty}}{(qtx;q)_{\infty} (qt;q)_{\infty} \big(qtx^{-1};q\big)_{\infty}},
\label{S1S2refined}
\end{equation}
where the refinement variable $t$ corresponds to homological-type grading in the context of categorification.
In general, in a Rozansky--Witten theory the space $\CH \big(S^2\big) = \C [X]$ encodes the geometry of the target space $X$, and the partition function on $S^1 \times S^2$ can be identified with the Hilbert series of $X$. In the case of $\hat Z$ theory, there is a manifest symmetry between the numerator and the denominator of the above expression. It is typical for a partition function of the cotangent bundles, $X = T^* M$, where the numerator and denominator correspond to the fiber and base, respectively. Indeed, ``half'' of the expression~\eqref{S1S2refined} is precisely the Hilbert series of the affine Grassmannian, whereas the complete expression~\eqref{S1S2refined} is the Hilbert series of a model for $T^* \mathrm{Gr}_G$ that was used in~\cite{Gukov:2020lqm}.

\begin{figure}[ht]
	\centering
	\includegraphics[width=0.60\textwidth]{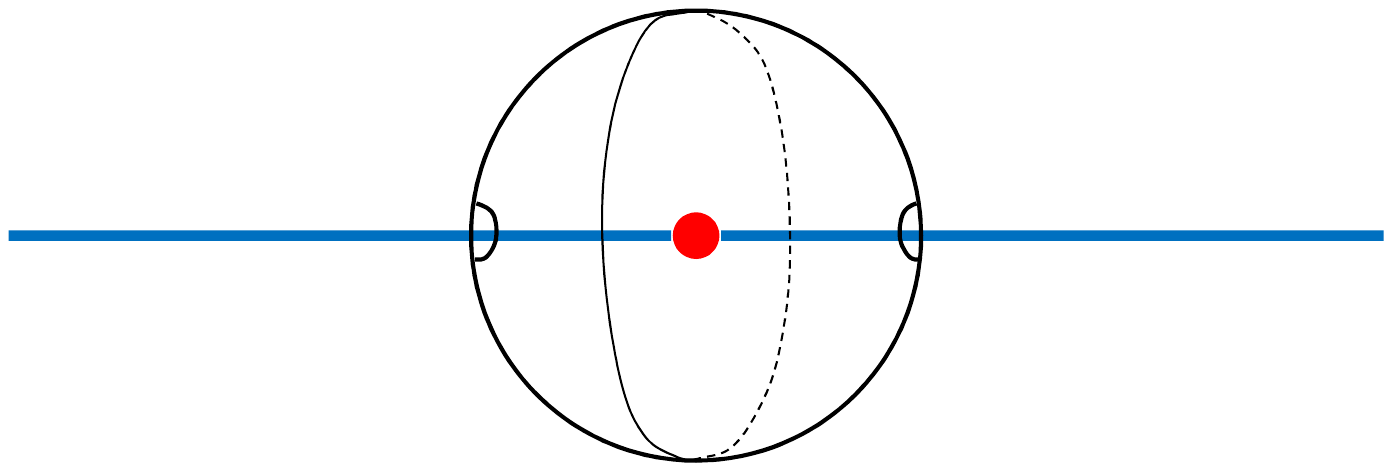}
	\caption{A sphere with two punctures surrounding a local operator at the junction of two line operators.}
	\label{fig:junction}
\end{figure}

Recall, that the affine Grassmannian for $G = \operatorname{SU}(2)$ has two connected components, with a~similar Morse cell complex. In each component, the transverse slices are labeled by a pair of non-negative integers, $m$ and $n$. In particular, the transverse slices with $n=0$ form a family of hyper-K\"ahler manifolds with quaternionic dimension $m^2$ that provide a finite-dimensional approximation to $\mathrm{Gr}_G$. Their Hilbert series is
\begin{equation*}
\frac{\big(q^{m+1};q\big)_{m}}{(qx;q)_{m} (q;q)_{m} \big(qx^{-1};q\big)_{m}}.
\end{equation*}
In the limit $m \to \infty$, we recover $\mathrm{Gr}_G$ itself (or, more precisely, one of its connected components) and the denominator of the formula~\eqref{S1S2refined} at $t=1$.

Now, returning to $\Sigma = T^2$, let $D^b \operatorname{Coh}(X)$ be the (bounded) derived category of coherent sheaves on $X$. Then, the proposal of~\cite{Gukov:2020lqm} that the GPPV TQFT can be equivalently formulated as the Rozansky--Witten theory with the target space $T^* \mathrm{Gr}_G$ implies, among other things,
\begin{equation}
\CH_{\text{GPPV}} \big(T^2\big) \cong K^0 \big( D^b_{\C^* \times \mathbb{T}} \operatorname{Coh}(T^* \mathrm{Gr}_G) \big),
\label{GPPVviaTGrG}
\end{equation}
where $\mathrm{Gr}_G$ is the affine Grassmannian of $G$ or, rather, the affine Grassmannian of $G_{\C}$ to be more precise. If the left-hand side, here is understood as the Hilbert~\eqref{HGPPVnaive} used in the cutting-and-gluing (surgery) formulae~\cite{Gukov:2019mnk}, then~\eqref{GPPVviaTGrG} is a priori not obvious at all. We shortly demonstrate that it is indeed the case, but also point out that producing similar arguments for other ingredients of GPPV theory is necessary to establish the equivalence of different definitions, such as the explicit integral formula for plumbed manifolds and the mathematically rigorous (but extremely hard to compute) definition based on the Rozansky--Witten theory with the affine Grassmannian as the target space.

Returning to the proof of~\eqref{GPPVviaTGrG}, first one needs to clarify what exactly is meant by $T^* \mathrm{Gr}_G$ which, for general $G$, is a singular infinite-dimensional space. Luckily, the ``space of triples''~\cite{Braverman:2016wma,Braverman:2016pwk,Nakajima:2017bdt} provides the right candidate for the total space of this cotangent bundle. Let $\text{St}_G$ be the affine Grassmannian analogue of the Steinberg variety. Then~\cite{MR2135527}
\begin{equation*}
K^{G(O)} (\text{St}_G) \cong \C \big[\mathbb{T}_{\C} \times {}^L\mathbb{T}_{\C}\big]^W
\end{equation*}
and
\begin{equation*}
\operatorname{Spec} K \big(D^b \operatorname{Coh}^{G(O)}_{\mathrm{St}_G} (T^* \mathrm{Gr}_G)\big) \cong
\operatorname{Spec} K^{G(O)} (\mathrm{Gr}_G) \cong \frac{\mathbb{T}_{\C} \times {}^L\mathbb{T}_{\C}}{W},
\end{equation*}
where $O = \C [[t]]$, $F = \C ((t))$, and $\mathrm{Gr}_G = G(F)/G(O)$. Enhancing the equivariant $G(O)$-action to $G(O) \times \operatorname{U}(1)$ action, where $\operatorname{U}(1)$ acts by loop rotation, corresponds to a non-commutative deformation (quantization) of this space. Up to a two-fold cover, this is precisely the quantization of the space $\mathcal{M}_{\text{flat}} \big(G_{\C},T^2\big) \cong \frac{\mathbb{T}_{\C} \times \mathbb{T}_{\C}}{W}$, thus demonstrating that the left-hand side of~\eqref{GPPVviaTGrG} given by~\eqref{HGPPVnaive} indeed is equal to the right-hand side of~\eqref{GPPVviaTGrG} computed geometrically.

\subsubsection{Quantization}

The Hilbert space of Chern--Simons theory and its close cousins on a surface $\Sigma$ can be obtained by quantizing a suitable ``phase space'' $M$, that in many interesting examples can be realized as a submanifold in the moduli space $\mathcal{M}_H (G,\Sigma) \cong \mathcal{M}_{\text{flat}} (G_{\C},\Sigma)$. For example, a ``real slice'' $M = \mathcal{M}_{\text{flat}} (G_{\mathbb{R}},\Sigma)$ that corresponds to a real form $G_{\mathbb{R}}$ of $G_{\C}$ gives the Hilbert space of ``$G_{\mathbb{R}}$ Chern--Simons theory'' (closely related to the Teichm\"uller theory~\cite{MR3204520}) whereas $M = \mathcal{M}_H (G,\Sigma)$ gives the Hilbert space of $\hat Z$ theory (which provides a non-perturbative definition to what one might call a ``$G_{\C}$ Chern--Simons theory''). In all of these cases, we can represent $\CH (\Sigma)$ as a~Hom-space ($=$ space of open strings) in the category of branes on $\mathcal{M}_H (G,\Sigma)$~\cite{MR2672467}:
\begin{equation}
\CH (\Sigma) = \operatorname{Hom} (\CB' , \CB_{cc}),
\label{HomBB}
\end{equation}
where $\CB_{cc}$ is a rank-1 brane that carries a line bundle of curvature $c_1 (\mathcal{L}) = \omega$ and $\CB'$ is supported on $M$. Note, at the current stage of development, the definition of the Fukaya category does not include coisotropic objects. However, as pointed out in~\cite{MR2672467}, the space~\eqref{HomBB} is best computed (and also defined!) in the derived category of coherent sheaves on $\mathcal{M}_H (G,\Sigma)$. Indeed, both $\CB'$ and $\CB_{cc}$ are objects in $D^b \operatorname{Coh}(\mathcal{M}_H (G,\Sigma))$, which is mathematically well defined and is much easier to work with than the Fukaya category. Then, in the case of compact $G$, the space~\eqref{HomBB} recovers\footnote{See~\cite{MR2672467} for details.} the standard result~\cite{MR1065677,Witten:1988hf} of the geometric quantization~\cite{MR0395610,MR0294568,MR1461545}, and in the complex case of $G_{\C}$ recovers the Hilbert space of the GPPV TQFT~\eqref{HGPPVnaive}. For compact $G$, there is a large body of work, that includes~\cite{MR3370620,MR1669720,MR1048605}, showing that the action of the mapping class group is equivalent to the one in the Reshetikhin--Turaev construction~\cite{reshetikhin1991invariants}, whereas in the complex case of $G_{\C}$ the mapping class group action has been studied in detail only in genus~1~\cite{Chun:2019mal,Gukov:2019mnk}.

In the case of BCGP invariants, we have~\cite[Remark~5.10]{blanchet2016non}
\begin{equation*}
\dim \CH_{\text{BCGP}} (\Sigma) =
\begin{cases}
r' & \text{if}\quad g=1, \\
r^{3g-3} & \text{if}\quad g>1~\text{and}~r~\text{is odd}, \\
\frac{r^{3g-3}}{2^{g-1}} & \text{if}\quad g>1~\text{and}~r~\text{is even}.
\end{cases}
\end{equation*}
Recall that $r'$ is $r$ (resp. $\frac{r}{2}$) when $r$ is odd (resp. even).
The choice of a background (``decoration'') that appears in $\N_r$ and $\hat Z$ invariants is of the same type as in $\mathbb{T}_{\C}$-crossed modules, thus allowing to describe both TQFTs also in the language of $\mathbb{T}_{\C}$-crossed modules.

\subsubsection{A prototypical example}
\label{sec:U1CS}

Consider bosonic Chern--Simons theory with gauge group $\operatorname{U}(1)$ and level $r\in \Z_+$. It has a~topological $\C/\Z$ (0-form) global symmetry. On the level of the path integral, on a closed 3-manifold~$M$, it can be coupled to a flat\footnote{Restriction of connections being flat can be interpreted as considering discrete topology on $\C/\Z$ symmetry group.} background $\C/\Z$ connection
\begin{equation*}
 \omega\in H^1(M;\C/\Z)\cong \Hom(H_1(M;\Z),\C/\Z) \cong
 \Hom\big(H^2(M;\Z),\C/\Z\big)
\end{equation*}
as follows:
\begin{equation}
 Z_{\operatorname{U}(1)_r}(M,\omega)=
 \int DA \exp\bigg\{ \frac{{\rm i}r}{2\pi}\int_{M} AdA +2\pi {\rm i}\omega(c_1) \bigg\},
 \label{Z-U1-CS-coupled}
\end{equation}
where $c_1\in H^2(M;\Z)$ is the first Chern class of the $\operatorname{U}(1)$ gauge connection (locally represented by 1-form $A$). On the level of charge/charged operators this 0-form global symmetry can be understood locally as follows (i.e., as in the general setting of~\cite{Gaiotto:2014kfa}). \textit{Charged} operators are 0-dimensional. An operator with charge $m\in \Z$ can be understood as the puncture in the spacetime with magnetic flux $m=(2\pi)^{-1}\int_{S^2} F$ over a small 2-sphere $S^2$ surrounding the puncture (i.e., monopole). A \textit{charge} operator $O_g(\Sigma)$ corresponding to a group element $g\in \C/\Z$ and supported on a 2-dimensional surface $\Sigma$ is simply
\begin{equation*}
 O_g(\Sigma)=\exp\bigg({\rm i}{g}\int_{\Sigma}F\bigg).
\end{equation*}
Globally, turning on a non-trivial $\omega$ in~\eqref{Z-U1-CS-coupled} can be realized by the insertion of the charge operator supported on a 2-chain representing the Poincar\'e dual of $\omega$, with charges on simplices given by the corresponding coefficients in $\C/\Z$. The $\operatorname{U}(1)_r$ Chern--Simons TQFT, ``enriched'' by this $\C/\Z$ global symmetry, then can be described in terms of a $G$-crossed MTC $\CC$, for $G=\C/\Z$, using the general formalism of~\cite{Barkeshli:2014cna} (cf.\ also~\cite{Benini:2018reh}) as follows. In the decomposition
\begin{equation*}
 \CC=\bigoplus_{g\in \C/\Z} \CC_g,
\end{equation*}
the component $\CC_g$ has simple objects that correspond to the line operators
\begin{equation*}
 W_e(\gamma)=\exp\bigg({\rm i}e \int_{\gamma}A\bigg)
\end{equation*}
with a complex charge $e$ with the fixed value $e\bmod 1=g\in \C/\Z$. For any fixed $g$ there are exactly $r$ distinct simple objects, as $W_r(\gamma)$ is known to be a trivial operator on the quantum level. That is, an equivalence class of the operator is determined by the value $e\bmod r\in \C/r\Z$. The fusion of the line operators is obviously consistent with the grading on the category.

Note that for $g\neq 0$, such line operators are not the ordinary line operators. They are not well defined by themselves, but are only allowed to live on a boundary of a surface $\Sigma$ (i.e., locally $\gamma=\partial\Sigma$) where a charge operator $O_g(\Sigma)$ is supported.

\subsection{Decorated TQFTs, gradings, and Fourier transform}\label{sec:TQFToperations}

Here we consider some basic operations on decorated TQFTs. We propose that BCGP and GPPV TQFTs are related by the combination of such operations (with slight modification related to the simplifications we will impose below). Namely, for $r=2\bmod 4$, BCGP TQFT can be obtained by applying ``Fourier transform" followed by ``$r$-wrapping" to GPPV TQFT. For $r=\pm 1\bmod 4$, the full transform is more involved, in particular on the level of the partition functions there is an extra summation over $\Tor H_1(M,\Z)$ in Conjectures~\ref{conj:cgpz} and~\ref{conj:cgpz2}. Note that in this relation, at all intermediate steps $r$ and $q$ should be considered as independent parameters. Only at the final step, one has to take the radial limit $q\rightarrow {\rm e}^{\frac{2\pi {\rm i}}{r}}$ (which, may lead to some divergences in certain cases). For the purpose of a more transparent exposition, instead of spin$^c$-TQFTs we consider $H^2(\cdot;\Z)$-decorated spin-TQFTs, and instead of $H^1(\cdot;\C/2\Z)$-decorated TQFTs we consider $H^1(\cdot;\R/\Z)$-decorated TQFTs. In the rest of the section, we study general TQFTs decorated by the structures as above and operations between them. We hope this discussion can be useful for other applications of decorated TQFTs, beyond the scope of the present paper.

\subsubsection{Decorated TQFTs and grading of Hilbert spaces}

``Decorated'' TQFTs (i.e., TQFTs defined on bordisms with additional structure) in general have induced grading on the vector spaces $V(\Sigma)$ associated to codimension-1 closed manifolds $\Sigma$ (i.e., the objects of the bordism category). The general rule is that the choice of the decoration on $\Sigma\times S^1$ decomposes into a choice of decoration on $\Sigma$ and a choice of the parameter \textit{dual} to the grading on $V(\Sigma)$. In particular, if the decoration on $\Sigma\times S^1$ is just a pullback of the decoration on $\Sigma$ (with respect to the projection $\Sigma\times S^1\rightarrow \Sigma$), then $Z\big(\Sigma\times S^1\big)=\dim V(\Sigma)$, where the right-hand side is the total dimension (over all gradings).

\subsubsection*{3d $\boldsymbol{H^1(\cdot;\R/\Z)}$-decorated TQFTs}
These are functors on the version of the category of $H^1(\cdot;\R/\Z)$-decorated cobordisms considered in~\cite[Sections~3.1 and~3.2]{blanchet2016non} with some additional data forgotten and the group $\C/2\Z$ replaced by~$\R/\Z$. The objects are closed surfaces $\Sigma$ with a distinguished point on each connected component and decorated by a class in $H^1(\Sigma,*;\R/\Z)$, where $*$ denotes the set of distinguished points. The morphisms are 3-dimensional cobordisms decorated by $H^1(M,*;\R/\Z)$ where $*$ are the set of the distinguished points on the boundary component.

The choice of the decoration on $\Sigma\times S^1$ is an element of
\begin{align*}
 H^1\big(\Sigma\times S^1;\R/\Z\big)
 &\cong \Hom (H_1(\Sigma;\Z),\R/\Z) \oplus \Hom(H_0(\Sigma;\Z),\R/\Z)
 \\
 &\cong H^1(\Sigma;\R/\Z)\oplus \Hom (H_0(\Sigma;\Z),\R/\Z).
\end{align*}
So $V(\Sigma)$ is naturally graded by $H_0(\Sigma)\cong H^2(\Sigma)$. The graded dimensions are given by the following relation:{\samepage
\begin{equation}
 \sum_{n\in H_0(\Sigma;\Z)}
 \dim V_{n}(\Sigma,\omega){\rm e}^{2\pi {\rm i}\alpha(n)}
 =Z\big(\Sigma\times S^1,\omega\oplus \alpha\big),
 \label{dec-dim1}
\end{equation}
where $\omega\in H^1(\Sigma)$ and $\alpha\in \Hom (H_0(\Sigma;\Z),\R/\Z)$.}

Physically, $\R/\Z$ is a 0-form symmetry. There are 0-dimensional charged operators with charges in $\Z$ and 2-dimensional charge operators labelled by $\R/\Z$, as in the example in Section~\ref{sec:U1CS}. Turning on $\alpha \in \Hom (H_0(\Sigma;\Z),\R/\Z)\cong H_2(\Sigma;\R/\Z)$ above corresponds to insertion of a charge operator along the spatial slice $\Sigma$.

\subsubsection*{3d $\boldsymbol{H^2(\cdot;\Z)}$-decorated TQFTs}

The choice of the decoration on $\Sigma\times S^1$ is an element of
\begin{equation*}
 H^2\big(\Sigma\times S^1;\Z\big)
 \cong H^2(\Sigma;\Z) \oplus H^1(\Sigma;\Z) \cong H^2(\Sigma;\Z) \oplus H_1(\Sigma;\Z).
\end{equation*}
So $V(\Sigma)$ is naturally graded by $\Hom(H_1(\Sigma;\Z),\R/\Z)\cong H^1(\Sigma;\R/\Z)$. The graded dimensions are given by the following relation:
\begin{equation}
 \sum_{\omega \in H^1(\Sigma;\R/\Z)}
 \dim V_{\omega}(\Sigma,n){\rm e}^{2\pi {\rm i}\omega(\gamma) }
 =Z\big(\Sigma\times S^1,n \oplus \gamma\big), \label{dec-dim2}
\end{equation}
where $n\in H^2(\Sigma;\Z)$ and $\gamma\in H_1(\Sigma;\Z)$. Since $H^1(\Sigma;\R/\Z)$ is in general not discrete, one has to specify what is meant by $\sum_{\omega}$. Consider the case of connected $\Sigma$ (the generalization to the disconnected case is straightforward). Let $g$ be the genus of $\Sigma$. Then $H^1(\Sigma;\R/\Z)\cong (\R/\Z)^{2g}=\big(S^1\big)^{2g}$, but non-canonically. There is a unique homogeneous form $\mu(\omega) \in \Omega^{2g}\big(H^1(\Sigma;\R/\Z)\big)$ normalized such that $\int \mu(\omega) =1$. Then
\begin{equation*}
 \sum_{\omega \in H^1(\Sigma;\R/\Z)}\ldots:=\int_{H^1(\Sigma;\R/\Z)} \mu(\omega)\ldots
\end{equation*}
It also satisfies
\begin{equation*}
 \int_{H^1(\Sigma;\R/\Z)} \mu(\omega) {\rm e}^{2\pi {\rm i}\omega(\gamma) }=\delta_{\gamma}:= \begin{cases}
 1, & \gamma=0, \\
 0, & \gamma \neq 0,
 \end{cases}
\end{equation*}
where $\gamma\in H_1(\Sigma;\Z)$.

Physically, $\Z$ is a 1-form symmetry. There are 1-dimensional charged operators with charges in $\R/\Z$ and 1-dimensional charge operators labelled by $\Z$. Turning on $\gamma \in H_1(\Sigma;\Z)$ above corresponds to insertion of a charge operator supported on a 1-dimensional curve in the spatial slice $\Sigma$.

\begin{table}[ht]
\centering
\caption{A quick guide to the symmetries.}
\begin{tabular}{c|c|c}
			\hline
			 &~$\phantom{\int^{\int^\int}}$~BCGP~$\phantom{\int^{\int^\int}}$~ &~$\phantom{\int^{\int^\int}}$~GPPV~$\phantom{\int^{\int^\int}}$~ \tabularnewline
				\hline
			Symmetry & $\operatorname{U}(1)$ & $\Z$ \tabularnewline
			$G$ & 0-form & 1-form \tabularnewline
			\hline
			Charged & $\Z$ & $\operatorname{U}(1)$ \tabularnewline
			objects & 0-dimensional & 1-dimensional \tabularnewline
			\hline
			$V (\Sigma)$ & \multirow{2}{*}{$H^1 (\Sigma,\operatorname{U}(1))$} & \multirow{2}{*}{$H^2 (\Sigma;\Z)$} \tabularnewline
			decorated by & & \tabularnewline
			\hline
			$V (\Sigma)$ & \multirow{2}{*}{$H^2 (\Sigma;\Z)$} & \multirow{2}{*}{$H^1 (\Sigma,\operatorname{U}(1))$} \tabularnewline
			graded by & & \tabularnewline		
			\hline			
		\end{tabular}
	\end{table}

\subsubsection{Fourier transform of TQFTs}

In this section, we suppose we are given two TQFTs, called $Z$ and $Z'$ defined on the categories of cobordisms decorated respectively by $H^2(\cdot;\Z)$ and by $H^1(\cdot;\R/\Z)$ classes, and we suppose that their partition functions on a closed 3-manifold $Y$ are related by a Fourier transform:
\begin{equation}
 Z'(Y,\omega) =\sum_{b\in H^2(Y;\Z)} {\rm e}^{2\pi {\rm i} \int_{Y}\omega \cup b} Z(Y,b),
 \label{fourier1}
\end{equation}
where $\omega\in H^1(Y;\C/\Z)$.
Such Fourier transform in particular appears as one of the two steps in the linear transform relating CGP and GPPV invariants in Conjecture~\ref{CGPZhatviac2}, with the second step being the ``$r$-wrapping'' that will be considered later.

We then try to extend this Fourier transform to the full category of cobordisms.
We first observe that, assuming all the integrals converge, a natural candidate for the inverse transform~is
\begin{equation}
 Z(Y,b) =\frac{1}{|\operatorname{Tor} H_1(Y;\Z)|}
 \int_{H^1(Y;\R/\Z)} \mu(\omega) {\rm e}^{2\pi {\rm i} \int_{Y}\omega \cup b} Z'(Y,\omega),
 \label{fourier2}
\end{equation}
where $\mu(\omega)$ is the homogeneous top degree form on $H^1(Y;\R/\Z)\cong \operatorname{Tor} H_1(Y;\Z) \times (\R/\Z)^{b_1}$ uniquely fixed by the condition that $\int_{\mathcal{M}} \mu(\omega)=1$
for each connected component $\mathcal{M}$ inside $H^1(Y;\R/\Z)$. The relation between the values of TQFT on a closed 2-manifold $\Sigma$ is given by the swap of grading with decoration
\begin{equation}
 V'_{n}(\Sigma,\omega) = V_{\omega}(\Sigma,n),
 \label{fourier-dim}
\end{equation}
where $n\in H^2(\Sigma;\Z)\cong H_0(\Sigma;\Z)$, $\omega\in H^1(\Sigma;\R/\Z)$ and the vector spaces are assumed to be finite dimensional. It is easy to see that~\eqref{fourier-dim} is consistent with~\eqref{dec-dim1},~\eqref{dec-dim2} combined with~\eqref{fourier1},~\eqref{fourier2}.

In order to extend the relation between the TQFTs to cobordisms, let us define the category of $H^2$-decorated cobordisms as follows:
\begin{itemize}\itemsep=0pt
\item The objects are pairs $(\Sigma,b)$ with $b\in H^2(\Sigma,\Sigma\setminus \{*\};\Z)$, where $\{*\}$ is the datum of one base point per connected component of $\Sigma$.
\item A morphism from $(\Sigma_-,b_-)$ to $(\Sigma_+,b_+)$ is a pair $(M,b)$ with $b\in H^2(M,\partial M\setminus \{*\};\Z)$, where $\partial M=\Sigma_+\sqcup \overline{\Sigma_-}$, $\{*\}=\{*_-\}\sqcup \{*_+\}$ is the set of one basepoint per component of $\Sigma_-\sqcup \Sigma_+$ and $b_\pm$ is the restriction of $b$ to $\Sigma_{\pm}$.
\end{itemize}
The composition of $M_1\colon(\Sigma_-,b_-)\to (\Sigma_0,b_0)$ and $M_2\colon(\Sigma_0,b_0)\to (\Sigma_+,b_+)$ is obtained by gluing~$M_1$ and $M_2$ along $\Sigma_0$ and defining $b$ on $M$ as $r(\phi^{-1}(b_-+b_+))$, where $\{*\}=\{*_+,*_-,*_0\}$,
\begin{equation*}
\phi\colon \ H^2(M,\partial M\cup \Sigma_0\setminus \{*\})\to H^2(M_-,\partial M_-\setminus \{*_-,*_0\})\oplus H^2(M_+,\partial M_+\setminus \{*_+,*_0\})
\end{equation*}
is coming from the Mayer--Vietoris sequence and
\begin{equation*}
r\colon \ H^2(M,\partial M\cup \Sigma_0\setminus \{*_+,*_0,*_-\})\to H^2(M,\partial M\setminus (\{*_+,*_-\}))
\end{equation*}
is the restriction map induced by the long exact sequence of the triple
\begin{equation*}
 (M,(\partial M \cup \Sigma_0) \setminus \{*_0,*_+,*_-\},\partial M \setminus \{*_+,*_-\}).
\end{equation*}

We shall now show that $V(\Sigma,b)$ is endowed by a $H^1(\Sigma;\R/\Z)$-grading as follows.
First of all remark that
\begin{equation*}
H^2(\Sigma\times [0,1],(\Sigma\setminus \{*\})\times \{0,1\};\Z)\cong H^1(\Sigma,\{*\};\Z)\oplus H^2(\Sigma,\Sigma\setminus \{*\};\Z),
\end{equation*}
where the injection $\delta\colon H^1(\Sigma;\Z)\cong H^1(\Sigma,\{*\};\Z)\to H^2(\Sigma\times [0,1],(\Sigma\setminus \{*\})\times \{0,1\};\Z)$ is induced by the exact sequence of the triple $(\Sigma\times [0,1],(\Sigma\setminus \{*\})\times \{0,1\},(\Sigma\setminus \{*\})\times \{0\}))$; and the generators $H^2(\Sigma,\Sigma\setminus \{*\};\Z)=\Z^{\pi_0(\Sigma)}$ are the Poincaré duals of the arcs $\{p\}\times [0,1]$ for $p\in \{*\}$.

Given $\omega\in H^1(\Sigma;\R/\Z)$, the $\omega$-homogeneous subspace of $Z(\Sigma, b)$ is defined as follows:
\begin{align*}
 V(\Sigma,b)_{\omega}:&=\big\{x\in V(\Sigma,b)\mid \forall c\in H^1(\Sigma;\Z)\
 Z(\Sigma\times [0,1],b+\delta(c))(x)
 =\exp(-2\pi {\rm i} \omega(c))x\big\}
\end{align*}
(remark that the restriction of $b+\delta(c)$ to $\Sigma\times \{0\}$ is $b$).

Now observe that there is a well defined map
\begin{equation*}
\int_M\colon \ H^1(M,\{*\};\R/\Z)\otimes H^2(M,\partial M\setminus \{*\};\Z)\to \R/\Z
\end{equation*}
defined equivalently as $\int_M \omega\otimes b:=\langle\omega\cup b,[M]\rangle=\omega(PD(b))$, where $[M]\in H^3(M,\partial M;\Z)$ is the fundamental class and $PD$ is Poincar\'e duality.

Considering all the morphisms $M\colon (\Sigma_-,b_-)\to (\Sigma_+,b_+)$ for which the underlying manifold is~$M$ we define their Fourier transform for any $\omega \in H^1(M,\{*\};\R/\Z)$ as
\begin{equation*}
Z'(M,\omega)_{b_-}:=\sum_{\substack{b\in H^2(M,\partial M\setminus \{*\};\Z)/\delta(H^1(\Sigma_-;\Z))\\ b|_{\Sigma_-}=b_-}} \exp\bigg(2\pi {\rm i} \int_{M} \omega\cup b\bigg) Z_{\omega}(M,b),
\end{equation*}
where
\begin{enumerate}[label=$(\roman*)$]\itemsep=0pt
\item $Z_{\omega}(M,b)\colon V_{\omega_-}(\Sigma_-,b_-)\to V_{\omega_+}(\Sigma_+,b_+)$
is the restriction of $Z(M,b)$ to the degree $\omega_{\pm}=\omega|_{\Sigma_\pm}$ vector subspaces and we use the identification $V_{\omega_\pm}(\Sigma_\pm, b_\pm)=V'_{b_\pm}(\Sigma_\pm, \omega_\pm)$ to interpret it as a map
$V'_{b_-}(\Sigma_-,\omega_-)\to V'_{b_+}(\Sigma_+,\omega_+)$;
\item $\delta\colon H^1(\Sigma_-;\Z)\to H^2(M,\partial M\setminus \{*\};\Z)$
is induced as above by the exact sequence of the triple $(M,\partial M\setminus \{*\}, \Sigma_-\setminus \{*_-\})$;
\item the sum is over all the representatives of classes $b$ restricting to $b_-$ on $\Sigma_-$ and the choice of a representative is irrelevant because we have:
\begin{align*}
&\exp\bigg(2\pi {\rm i} \int_{M} \omega\cup (b+\delta(c))\bigg) Z_{\omega}(M,b+\delta(c))
\\ &\qquad
=\exp\bigg(2\pi {\rm i} \int_{M} \omega\cup b+\omega\cup \delta(c)\bigg) Z_{\omega}(M,b)\circ Z_{\omega}(\Sigma_-\times [0,1],b_-+\delta(c))
\\ &\qquad
=\exp\bigg(2\pi {\rm i} \int_{M} \omega\cup b\bigg)\exp(2\pi {\rm i} \omega(c)) Z_{\omega}(M,b)\exp(-2\pi {\rm i} \omega(c))
\\ &\qquad
=\exp\bigg(2\pi {\rm i} \int_{M} \omega\cup b\bigg) Z_{\omega}(M,b),
\end{align*}
where in the second equality we used the definition of the grading on $V(\Sigma,b)$.
\end{enumerate}
Of course, in the above formula we assume that the sum is convergent.
So given a cobordism $M\colon\Sigma_-\to \Sigma_+$ and $\omega\in H^1(M,\{*\};\R/\Z)$ we will from now on say that $Z'(M,\omega)$ exists if it exists for all $b_-\in H^2(\Sigma_-,\Sigma_-\setminus \{*_-\};\Z)$.

\begin{Proposition}
Suppose that $M_-\colon\Sigma_-\to \Sigma_0$ and $M_+\colon \Sigma_0\to \Sigma_+$ are two cobordisms and let $M=M_+\circ M_-$. Let $\omega_\pm\in H^1(M_\pm,\{*_0,*_\pm\}; \R/\Z)$
and $\omega\in H^1(M,\{*_+,*_-\}; \R/\Z)$ be defined as $\operatorname{res}(\omega')$, where $\omega'\in H^1(M,\{*_-,*_+,*_0\};\R/\Z)$ restricts to both $\omega_\pm$, and
\[
\operatorname{res}\colon H^1(M,\{*_-,*_+,*_0\};\R/\Z)\to H^1(M,\{*_-,*_+\};\R/\Z)
\]
is the restriction map.
If $Z'(M_\pm,\omega_\pm)$ exist, then also $Z'(M,\omega)$ exists and it holds
\begin{equation*}
Z'(M,\omega)=Z'(M_+,\omega_+)\circ Z'(M_-,\omega_-).
\end{equation*}
\end{Proposition}

\begin{proof}
Let $b\in H^2(M,\Sigma_-\setminus \{*_-\}\cup \Sigma_+\setminus \{*_+\};\Z)$.
By the exact sequence of the triple
\begin{equation*}
(M,\Sigma_-\sqcup \Sigma_0\sqcup \Sigma_+\setminus \{*_-,*_0,*_+\},\Sigma_-\sqcup \Sigma_+\setminus \{*_-,*_+\}),
\end{equation*}
we have a surjective map
\begin{equation*}
\pi\colon \ H^2(M, \Sigma_-\sqcup \Sigma_0\sqcup \Sigma_+\setminus \{*_-,*_0,*_+\};\Z)\to H^2(M,\Sigma_-\sqcup \Sigma_+\setminus \{*_-,*_+\};\Z)
\end{equation*}
so that $\pi^{-1}(b)$ is well defined up to elements of the form $\delta(c)$ for some $c\in H^1(\Sigma_0\setminus \{*_0\});\R/\Z)$ where $\delta$ is induced by the same exact sequence.
So if $b'\in \pi^{-1}(b)$ then there are well-defined restrictions ${\rm res}_\pm(b')\in H^2(M_\pm,\Sigma_\pm\setminus \{*_\pm\}\sqcup \Sigma_0\setminus \{*_0\};\Z)$. Furthermore, one can check that the restrictions to $H^2(\Sigma_0,\Sigma_0\setminus \{*_0\};\Z)$ of ${\rm res}_+(b')$ and of ${\rm res}_-(b')$ coincide and depend only on $b$ (not on the choice of $b'$). Let us then denote this common restriction $b_0\in H^2(\Sigma_0,\Sigma_0\setminus \{*_0\};\Z)$.
We have then
\begin{align*}
Z'(M,\omega)_{b_-}&=\sum_{\substack{b\in \frac{H^2(M,\Sigma_-\setminus \{*_-\}\cup \Sigma_+\setminus \{*_+\};\Z)}{\delta(H^1(\Sigma_-\setminus \{*_-\};\Z))} \\ {\rm res}_-(b)=b_-}}
\exp\bigg(2{\rm i}\pi \int_M\omega\cup b\bigg) Z(M,b)
\\
&=\sum_{\substack{b'\in \frac{H^2(M,\Sigma_-\setminus \{*_-\}\cup \Sigma_0\setminus \{*_0\}\cup \Sigma_+\setminus \{*_+\};\Z)}{\delta(H^1(\Sigma_0\setminus \{*_0\}\sqcup \Sigma_-\setminus \{*_-\};\Z))} \\ {\rm res}_-(\pi(b'))=b_-}}
\exp\bigg(2{\rm i}\pi \int_M\omega\cup \pi(b')\bigg) Z(M,\pi(b'))
\\
&=\sum
\exp\bigg(2{\rm i}\pi \int_{M_-}\omega_-\cup b'_-+\int_{M_+}\omega_+\cup b'_+\bigg) Z(M,b'_+)Z(M,b'_-),
\end{align*}
where the last sum ranges over
\begin{equation*}
b'_-\in H^2(M_-,\Sigma_-\setminus \{*_-\}\cup \Sigma_0\setminus \{*_0\};\Z)/\delta\big(H^1(\Sigma_-\setminus \{*_-\};\Z)\big)
\end{equation*}
such that ${\rm res}_-(\pi(b'_-))=b_-$ and over
\begin{equation*}
b'_+\in H^2(M_+,\Sigma_+\setminus \{*_+\}\cup \Sigma_0\setminus \{*_0\};\Z)/\delta\big(H^1(\Sigma_0\setminus \{*_0\};\Z)\big)
\end{equation*}
such that ${\rm res}_0(\pi(b'_+))={\rm res}_0(\pi(b'_-))$ (and where we let $b'_\pm$ be the restriction of $b'$ to $(M_\pm,\Sigma_\pm\sqcup \Sigma_0\setminus \{*_\pm,*_0\})$.
The last equality uses the Mayer--Vietoris sequence for the pairs $(M,\Sigma_{+,-,0}\setminus \{*_{+,-,0}\})$ and $(M,\Sigma_{\pm,0}\setminus \{*_{\pm,0}\})$ and the fact that $Z$ is functorial.
\end{proof}

\subsubsection[``$r$-wrapping'' of $H^1(\cdot;\R/\Z)$]{``$\boldsymbol r$-wrapping'' of $\boldsymbol {H^1(\cdot;\R/\Z)}$}

The other operation which we will need to upgrade the relation~\eqref{CGP-Zhat-2mod4-mod} to TQFTs is the operation that takes an $H^1(\cdot;\R/\Z)$-TQFT $Z$ and produces another $H^1(\cdot;\R/\Z)$-TQFT $Z'$. For closed connected manifolds, we have (assume $r=2\bmod 4$):
\begin{align}
 Z'(M,\omega)={}&
 \frac{ r^{-(b_0+b_1)/2}}{\sqrt{| \Tor H_1(M;\Z)|}} \nonumber
 \\
 &\times\int_{H^1(M;\R/\Z)}\mu(\alpha)
 {\rm e}^{-\frac{\pi {\rm i}r}{2} q_s(\alpha')-2\pi {\rm i} \lk(\alpha',\omega')}
 r^{b_1}\delta(r\alpha''-\omega'')
 {Z}(M,\alpha),
 \label{r-wrap}
\end{align}
where we choose an explicit splitting
\begin{equation}
\omega =\omega'\oplus \omega'', \qquad
\alpha =\alpha'\oplus \alpha''\in H^1(M;\R/\Z)\cong \Tor H_1(M;\R/\Z)\oplus (\R/\Z)^{b_1}
\label{r-wrapping-splitting}
\end{equation}
as in Section~\ref{sec:b1-positive}.

Applying the formula~\eqref{r-wrap} to $M=\Sigma\times S^1$, where $\Sigma$ is a closed connected oriented genus $g$ surface, we get the following relation between the dimensions of the corresponding graded vector spaces:
\begin{equation*}
 \dim V'_n(\Sigma,\omega)=r^{-g}\sum_{m\in (\Z/r\Z)^{2g}}
 \dim V_{rn}\bigg(\Sigma,\frac{\omega+m}{r}\bigg),
\end{equation*}
where $n\in\Z$, $\omega\in H^1(\Sigma;\R/\Z)\cong (\R/\Z)^{2g}$.

This formula suggests the following generalisation to TQFTs. Assume that after a choice of a symplectic basis in $H_1(\Sigma;\Z)$ (that splits the generators into $g$ A-cycles and $g$ B-cycles) and the corresponding splitting $\omega = \omega_A\oplus \omega_B$ ($\omega_{A,B}\in (\R/\Z)^{g}$), one can identify $V_n(\Sigma,\omega_A\oplus \omega_B)$ for a fixed $\omega_A$ and all possible $\omega_B$. That is one can explicitly drop the dependence on $\omega_B$: $V_n(\Sigma,\omega_A):=V_n(\Sigma,\omega_A\oplus \omega_B)$. Then the same is true for $V'_n(\Sigma,\omega_A\oplus \omega_B)$, and we have
\begin{equation*}
 V'_n(\Sigma,\omega_A)=\bigoplus_{s\in (\Z/r\Z)^g} V_{rn}\bigg(\Sigma,\frac{\omega_A+s}{r}\bigg),
\end{equation*}
where $(\Z/r\Z)^g\subset (\R/\Z)^{g}$ is identified with the subgroup of holonomies of order $r$ along the A-cycles.

This is consistent with the conjectural relation between the CGP and $\hat{Z}$ invariants of knots (which are valued in the corresponding vector spaces above for $\Sigma =T^2$).

The extension of the relation~\eqref{r-wrap} to general bordisms turns out to be subtle and technically complicated. One of the reasons is that one requires to choose a splitting of $H^1(M;\R/\Z)$ in~\eqref{r-wrapping-splitting}, which is not a natural structure on manifolds and there are various ways one can extend it to the manifolds with boundary. We will not address this issue in this work.

\section{Relation to WRT invariants}
\label{sec:WRT}

In the previous sections, we focused on the relations between $\hat{Z}$ and $\N_r$. But there are also relations between and $\hat{Z}$ and $\WRT$ invariants and between $\N_r$ and $\WRT$ invariants (possibly in their cohomology-refined versions). In this section, we compare this triangle of relations and show that they are compatible with each other.

\subsection[Compatibility of the relations between ${\rm N}^0_r$, ${\rm WRT}_r$, and $\hat{Z}$]{Compatibility of the relations between $\boldsymbol{{\rm N}^0_r}$, $\boldsymbol{{\rm WRT}_r}$, and $\boldsymbol{\hat{Z}}$}

Let $M$ be a rational homology sphere. Consider
\begin{equation*}
 Z_a^\so(M):=\sum_{\mathfrak{s}\in \mathrm{Spin}^c(M)} S_{a,\mathfrak{s}}^{\so}\hat{Z}_{\mathfrak{s}}(M),\qquad a\in H_1(M;\Z),
\end{equation*}
where
\begin{align*}
S^{\so}_{a,\sigma(b,s)} := \frac{1}{{|H_1(M;\Z)|}}
\begin{cases}
\sum_{f} {\rm e}^{2\pi {\rm i}\left( -\frac{r-1}{4}\lk(a,a)
 +\lk(a,f-b) +\lk(f,f) -\frac{1}{4}\mu(M,s)\right)} & \text{if}~r=1~\text{mod}~4,
 \\[2mm]
 \sqrt{|H_1(M;\Z)|}
 {\rm e}^{-\frac{\pi {\rm i}r}{2} q_s(a)-2\pi {\rm i} \lk(a,b)} & \text{if}~r=2~\text{mod}~4,
 \\[2mm]
\sum_{f} {\rm e}^{2\pi {\rm i}\left(-\frac{r+1}{4}\lk(a,a)
 -\lk(a,f+b) -\lk(f,f) -\frac{1}{4}\mu(M,s)\right)} & \text{if}~ r=3~\text{mod}~4.
\end{cases}
\end{align*}
We have introduced $\operatorname{SO}(3)$ subscript in order to explicitly distinguish $Z_a^{{\rm SO}(3)}(M)$ from $Z_a(M)$ that appeared in~\cite{Gukov:2016njj,Gukov:2017kmk,Gukov:2016gkn}. The latter naturally appear in the decomposition of the standard WRT invariant (corresponding to $\operatorname{SU}(2)$ Chern--Simons gauge theory), while the former, as we will see below, appear naturally in the decomposition of its mod-2-cohomology refined version, which can be understood as $\operatorname{SO}(3)$ version of Chern--Simons gauge theory. In physical terms, the refinement parameter $\omega\in H^1(M,\Z/2\Z)$ appears in the action term $\pi\int_{M}\omega\cup w_2 $ where $w_2\in H^2(M,\Z/2\Z)$ is the second Stiefel--Whitney class of the $\operatorname{SO}(3)$ gauge bundle.

Conjecture~\ref{conj:cgpz} can be written as follows (for any $r\neq 0\bmod 4$):
\begin{equation}
 \N_r(M,\omega)=(-1)^r \CT(M,[\omega])\sum_{a\in H_1(M;\Z)}
 {\rm e}^{-\pi {\rm i}\omega(a)} Z^\so_a(M)\Big|_{q={\rm e}^{\frac{2\pi {\rm i}}{r}}}
 \label{CGP-Za}
\end{equation}
for $\omega\in H^1(M;\C/2\Z) \setminus H^1(M;\Z/2\Z)$.
Similarly, for the $H^1(M;\Z/2\Z)$-refined WRT invariant of~\cite{costantino2015relations} it is conjectured (see Appendix~\ref{app:wrt-refined} for details):
\begin{equation}
 \WRT_r(M,\omega)=\frac{1}{{\rm i}\sqrt{8r}}\sum_{a\in H_1(M;\Z)}
 {\rm e}^{-\pi {\rm i}\omega(a)} Z^\so_a(M)\Big|_{q={\rm e}^{\frac{2\pi {\rm i}}{r}}}
 \label{refined-wrt-Za}
\end{equation}
for $\omega\in H^1(M;\Z/2\Z)$.

Given $\omega\in H^1(M;\C/2\Z)$, as shown in~\cite{costantino2014quantum}, a normalized invariant $N_r^0(M,\omega)$ can be defined~by
\begin{equation}
 \N_r^0(M,\omega):=\frac{\N_r(M\# M',\omega\oplus \omega')}{\N_r(M',\omega')}
 \label{N0-conn-def}
\end{equation}
for any $M'$ and $\omega'\in H^1(M';\C/2\Z)$ such that both denominator and numerator in the right-hand side are well defined and moreover the denominator is non-zero. For $\omega \notin H^1(M;\Z/2\Z)$, we have
\begin{equation*}
 \N_r^0(M,\omega)=0.
\end{equation*}
And for $\omega \in H^1(M;\Z/2\Z)$, the conjecture of~\cite{costantino2015relations} in this paper's normalisation (theorem for knot surgeries) is that\footnote{For $b_1>0$, it is assumed that $|H_1(M;\Z)|=0$.}
\begin{equation}
 \N_r^0(M,\omega)=D\cdot |H_1(M;\Z)| \WRT_r(M,\omega).
 \label{N0-refined-wrt}
\end{equation}
The statement of the above conjecture follows from \cite[Theorem~1.4]{de2020nonsemisimple} and \cite[Theorem~1]{chen2009relation} at least for $r$ odd and for $\omega=0$.

Let us check that~\eqref{N0-refined-wrt} is consistent with our conjectures~\eqref{CGP-Za} and~\eqref{refined-wrt-Za}. Plugging~\eqref{CGP-Za} into the definition~\eqref{N0-conn-def} we have
\begin{gather*}
 \N_r^0(M,\omega)=\frac{\CT(M\# M',[\omega]\oplus [\omega'])}{\CT(M',[\omega'])}
 \cdot \frac{ \sum_{\substack{a\in H_1(M;\Z)\\ a'\in H_1(M';\Z)} }
 {\rm e}^{-\pi {\rm i}\omega(a)-\pi {\rm i} \omega'(a')} Z^\so_{a\oplus a'}(M\# M')}
 {\sum_{ a'\in H_1(M';\Z)} {\rm e}^{-\pi {\rm i} \omega'(a')} Z^\so_{a'}(M')}
 \Bigg|_{q={\rm e}^{\frac{2\pi {\rm i}}{r}}}.
\end{gather*}

Assume that $b_1(M')>0$ and $\omega'$ is a generic element of $H^1(M;\C/2\Z)$. Then (see, e.g.,~\cite{turaev2002torsions})
\begin{equation*}
 \frac{\CT(M\# M',[\omega]\oplus [\omega'])}{\CT(M',[\omega'])}= |H_1(M;\Z)|.
\end{equation*}
Taking a limit (possible since $b_1(M')>0$) where $\omega'$ tends to an element of $H_1(M;\Z/2\Z)$ (e.g., zero) and using~\eqref{refined-wrt-Za}, we then have
\begin{equation*}
 \N_r^0(M,\omega)=|H_1(M;\Z)|
 \frac{\WRT_r(M\# M',\omega \oplus \omega')}{\WRT_r(M',\omega')}.
\end{equation*}
Using $\WRT_r(M\# M',\omega \oplus \omega')=D\cdot \WRT_r(M,\omega)\WRT_r(M',\omega')$,
we then indeed arrive at~\eqref{N0-refined-wrt}.

\subsection{0-surgeries on twist knots}
\label{sec:twistknots}

0-surgeries on knots, $M = S^3_0 (K)$, all have $H_1 (M) = \Z$. In such cases, the relation to the standard (not refined) WRT invariants is expected to be especially simple~\cite{Chun:2019mal}:
\begin{equation*}
\mathrm{WRT}_r (S^3_0 (K)) = - \frac{1}{2 D^2} \big[ \hat Z_0^{(+)} + \hat Z_0^{(-)} \big] \big|_{q \to {\rm e}^{2\pi {\rm i} /r}}.
\end{equation*}
This relation was explicitly verified in~\cite{Chun:2019mal} for many twist knots and torus knots, and is expected to hold more generally. Here,
\begin{equation*}
\hat Z_0^{(+)} \big(S^3_0 (K)\big) = \mathrm{Res}_{x=0} \frac{x^{1/2} - x^{-1/2}}{x} F_K (x,q) = \frac{1}{2} f_1^K (q)
\end{equation*}
is $\hat Z$-invariant in the trivial spin$^c$ structure. For twist knots, we also have
\begin{equation*}
\hat Z_0^{(-)} \big(S^3_0 (K)\big) = \mathrm{Res}_{x=x_0} \frac{x^{1/2} - x^{-1/2}}{x} F_K (x,q),
\end{equation*}
where $x_0$ is the root of the Alexander polynomial with $\operatorname{Re} (x_0) > 0$. (Recall, that for a twist knot the Alexander polynomial has degree $2$, so that the corresponding residues differ by a sign.) Therefore, we can write
\begin{equation*}
\mathrm{WRT}_r \big(S^3_0 (K)\big) = - \frac{1}{2 D^2} \lim_{q \to {\rm e}^{2\pi {\rm i} /r}} \oint_C \frac{{\rm d}x}{2\pi {\rm i} x} \frac{x^{1/2} - x^{-1/2}}{x} F_K (x,q)
\end{equation*}
where the contour $C$ goes around $x=0$ and $x_0$. It has been checked in~\cite{Chun:2019mal} that this procedure indeed recovers the correct WRT invariants of $M = S^3_0 (K)$ for many twists knots $K$.

This way of recovering $\mathrm{WRT}_r \big(S^3_0 (K)\big)$ is {\it a priori} different from the strategy used earlier in this paper, where the CGP and WRT invariants of $S^3_0 (K)$ are obtained from $F_K (x,q)$ by first specializing $F_K (x,q)$ to a root of unity $q = {\rm e}^{2 \pi {\rm i} /r}$ and then summing over colors / decorations $x$ as in a typical surgery formula. Roughly speaking, this approach~-- based on the relation~\eqref{FKADO} to $\operatorname{ADO}_r$ invariants~-- exchanges the order of operations, so that the limit $q \to {\rm e}^{2 \pi {\rm i} /r}$ comes first. It is instructive to verify that these two methods are compatible and, therefore, form a~consistent network of proposed relations.
Namely, repeating the arguments of Appendix~\ref{app:wrt-refined} for the $0$-surgeries on knots, we expect the following relation:
\begin{equation}
\mathrm{WRT}_r \big(S^3_0 (K)\big) = - \frac{1}{2 D^2} \lim_{q \to {\rm e}^{2\pi {\rm i} /r}} \sum_{n \in \Z} \hat Z_{nk},\label{WRTZn}
\end{equation}
where, as in earlier sections, by $\hat{Z}_n$ we denote the invariant $\hat{Z}_\mathfrak{s}(S^3(K))$ associated to the spin$^c$ structure encoded by the integer $2n$ on the knot (see Section~\ref{sec:combinatorial}). Alternatively, according to~\eqref{FKADO},
\begin{equation*}
\mathrm{WRT}_r \big(S^3_0 (K)\big) = - \frac{1}{4 D^2} \sum_{n=0}^{r-1}
\big( \xi^{n} - \xi^{-n} \big)^2
\operatorname{ADO}_r \big(\xi^{2n-2};K\big),
\end{equation*}
where we used $\Delta_K (1) = 1$. Comparing the right-hand sides of the above two formulae we can eliminate $\mathrm{WRT}_r \big(S^3_0 (K)\big)$ from these relations. Similar considerations apply to the invariants $\N_{r}\big(S^3_0(K),\omega\big)$; the only difference is that there is still $x$-dependence in all of the expressions, and, as explained in Section~\ref{sec:0surgknot}, we obtain
\begin{equation}
\lim_{q \to {\rm e}^{2\pi {\rm i} /r}} \sum_{n \in \Z} \hat Z_{nr} x^{nr} = \sum_{n=0}^{r-1}
\frac{\operatorname{ADO}_r \big(x \xi^{2n-2};K\big)}{2 \Delta_K (x^{r})} \cdot \big( \xi^{n} x^{1/2} - \xi^{-n} x^{-1/2} \big)^2.
\label{ratfncnxr}
\end{equation}
Since $\operatorname{ADO}_r (x;K)$ is a polynomial in $x$, the right-hand side is expected to be a rational function in $x$. Let us illustrate how this works for the figure-8 knot. In the conventions~\eqref{FKgeneral}, for the figure-8 knot we have $f_1 = 1$, $f_3 = 2$, $f_5 = q^{-1} + 3 + q$, and so on. From~\eqref{ZnFm}, we get $\hat Z_0 = -2$, $\hat Z_1 = -1$, $\hat Z_2 = -q^{-1} - 1 - q$, etc. For example, for $r=2$ and any knot $K$, with our conventions we have $\operatorname{ADO}_2 (x;K) = \Delta_K (x)$ and the right-hand side of the above relation becomes
\begin{equation}
\frac{\big(x - 2 + x^{-1} \big) \cdot \Delta_K (- x) + \big({-} x - 2 - x^{-1} \big) \cdot \Delta_K (x)}{2 \Delta_K \big(x^{2}\big)}.
\label{AlexAlex}
\end{equation}
For example, for the figure-8 knot, $\Delta_{{\bf 4_1}} (x) = -x^{-1} + 3 - x$, and we get
\begin{equation*}
- \frac{x^2-4+x^{-2}}{x^2-3+x^{-2}} =
- 1 + x^2 + 3 x^4 + 8 x^6 + 21 x^8 + 55 x^{10}
+ 144 x^{12} + 377 x^{14} + 987 x^{16}
+ \cdots.
\end{equation*}
The coefficients of this expansion perfectly match $\hat Z_n \big|_{q \to -1}$ for even values of $n$.\footnote{Turning this around, we can say that $\big(x^{1/2} - x^{-1/2}\big) F_K (x,q)$ restricted to even powers of $x$ is a $q$-deformation of the rational function~\eqref{AlexAlex} determined by the Alexander polynomial.} Note, that~\eqref{ratfncnxr} can be viewed as a close cousin of the Conjecture~\ref{conj:park}, obtained from it by multiplying with $x^{\frac{1}{2}} - x^{- \frac{1}{2}}$, replacing $x$ by $\xi^{2j} x$, and then summing over $j = 0, \ldots, r-1$. This again verifies the consistency of various proposed relations.

In order to perform a similar computation for other values of $r$, it may be convenient, building on~\cite{Ekholm:2020lqy,Kucharski:2020rsp}, to express $\hat Z_n (q)$ in the quiver form,
\begin{equation*}
\hat Z_n (q) =
2 \tilde f_{n-1} (q) - \tilde f_{n-2} (q) - \tilde f_{n} (q),\qquad n > 1,
\end{equation*}
where, for the figure-8 knot,
\begin{equation*}
\tilde f_n (q) = \sum_{n_1+ \cdots + n_6 = n}
\big({-}q^{1/2}\big)^{n_4 + n_5 + n_6} q^{(n_2 + n_5)(n_6 - n_1) + \frac{1}{2}(n_4^2 + n_5^2 + n_6^2)}
\prod_{i=1}^6 \frac{1}{(q;q)_{n_i}}.
\end{equation*}
Curiously, much like $f_m (q)$, the coefficients $\tilde f_{n} (q)$ are all Laurent polynomials in $q$ with integer coefficients.

Now, once we managed to write the right-hand side of~\eqref{WRTZn} with the regularization parameter~$x$ as a rational function, it is straightforward to set $x=1$ and obtain the WRT invariant. For example, for any knot $K$ we have
\begin{equation*}
\mathrm{WRT}_2 \big(S^3_0 (K)\big) = 1.
\end{equation*}
This is indeed what we expect from the surgery formula for
\begin{equation*}
 \mathrm{WRT}_r \big(S^3_0 (K)\big) = D^{-2} \sum_{n=1}^{r-1} (-1)^{n+1}[n] J_{n-1} (q)\Big|_{q={\rm e}^\frac{2\pi {\rm i}}{r}}.
\end{equation*}

\subsection[From CGP to WRT via $\hat{Z}$ for 0-surgeries on knots]{From CGP to WRT via $\boldsymbol{\hat{Z}}$ for 0-surgeries on knots}

Consider the general surgery formula~\eqref{Zhat-b1-def} for $\hat{Z}$ in the case of 0-surgery on a knot $K$. We have $V=b_1=1$, $\varepsilon''=1$, $\sigma=0$, $B=0$ and
\begin{equation}
 \hat{Z}_{\sigma(b'',s)}=q^{b''} F_{2b''},
 \label{Z-hat-zero-surgery-from-F}
\end{equation}
where $b''\in \Z\cong H_1\big(S^3_0(K);\Z\big)$ runs over integers and $F_{2b''}$ are the coefficients of the following formal power series:
\begin{equation*}
 F(x,q) := F_K\big(x^2,q\big)\big(x-x^{-1}\big)=\sum_{\ell}F_\ell x^\ell.
\end{equation*}
For knots only even powers actually appear in the series $F(x,q)$ (not to be confused with $F_K(x,q)$).

\subsubsection[Odd level $r$]{Odd level $\boldsymbol r$}

As before, denote $\mu:=\omega(\mathfrak{m}) \in \C/2\Z$, where $\mathfrak{m}$ is the class of the meridian of the know in $H_1\big(S^3_0(K);\Z\big)$. The formula~\eqref{CGP-Zhat-1mod4-mod} in the case of a 0-surgery on $K$ then reads
\begin{equation*}
 \N_{r}\big(S^3_0(K),\omega\big)=
 r \frac{\Delta\big({\rm e}^{2\pi {\rm i} \mu}\big)}{\big({\rm e}^{\pi {\rm i}\mu}-{\rm e}^{-\pi {\rm i} \mu}\big)^2} \sum_{m\in \Z} {\rm e}^{2\pi {\rm i}\mu m}
 \hat{Z}_{\sigma(rm,s)}\Big|_{q\rightarrow {\rm e}^{\frac{2\pi {\rm i}}{r}}}.
\end{equation*}
Plugging in~\eqref{Z-hat-zero-surgery-from-F} and using $\Delta(1)=1$, we then have
\begin{equation}
 \lim_{\mu\rightarrow 0}[r\mu]^2 \N_{r}\big(S^3_0(K),\omega\big)=
 \lim_{\mu\rightarrow 1}[r\mu]^2 \N_{r}\big(S^3_0(K),\omega\big)=
 -\frac{D^2}{2}\sum_{m\in\Z}
F_{2mr}\Big|_{q\rightarrow {\rm e}^{\frac{2\pi {\rm i}}{r}}}.
 \label{CGP-zero-limit-odd}
\end{equation}
On the other hand, for WRT invariant we have
\begin{align}
 \mathrm{WRT}_r\big(S^3_0(K)\big)&=D^{-2}\sum_{n=1}^{r-1}(-1)^{n+1}J_{n-1}(q)[n]=-\frac{1}{2r}\sum_{n\in \Z_r}F\big(q^{n/2},q\big)|_{q\rightarrow {\rm e}^\frac{2\pi {\rm i}}{r}}\nonumber
 \\
 &=-\frac{1}{2r}\sum_{\substack{ n\in \Z_r \\ b''\in \Z}}{\rm e}^\frac{2\pi {\rm i}nb''}{r} F_{2b''}\Big|_{q\rightarrow {\rm e}^\frac{2\pi {\rm i}}{r}}
 =-\frac{1}{2}\sum_{m\in\Z}F_{2mr}\Big|_{q\rightarrow {\rm e}^\frac{2\pi {\rm i}}{r}}.
 \label{WRT-zero-odd}
\end{align}
Combining~\eqref{CGP-zero-limit-odd} and~\eqref{WRT-zero-odd} we get
\begin{equation*}
 \lim_{\mu\rightarrow 0}[r\mu]^2 \N_{r}\big(S^3_0(K),\omega\big)=
 \lim_{\mu\rightarrow 1}[r\mu]^2 \N_{r}\big(S^3_0(K),\omega\big)=
 D^2\operatorname{WRT}_r\big(S^3_0(K)\big),
\end{equation*}
which is in agreement with Theorem~\ref{thm-CGP-knot-surgery-limit}.

\subsubsection[Even level $r$]{Even level $\boldsymbol r$}

 The formula~\eqref{CGP-Zhat-2mod4-mod} in the case of a 0-surgery on $K$ then reads
\begin{equation*}
 \N_{r}(M,\omega)=
 r \frac{\Delta_K\big({\rm e}^{2\pi {\rm i}\mu}\big)}{\big({\rm e}^{\pi {\rm i}\mu}-{\rm e}^{-\pi {\rm i} \mu}\big)^2}
 \sum_{ m }
 {\rm e}^{\pi {\rm i}\mu m} \hat{Z}_{\sigma(\frac{rm}{2},s)}\Big|_{q\rightarrow {\rm e}^{\frac{2\pi {\rm i}}{r}}}.
\end{equation*}
Plugging in~\eqref{Z-hat-zero-surgery-from-F} we then have
\begin{equation}
 \lim_{\mu\rightarrow 0}[r\mu]^2 \N_{r}\big(S^3_0(K),\omega\big)=
 -\frac{D^2}{2}\sum_{m\in\Z}
(-1)^m F_{mr}\Big|_{q\rightarrow {\rm e}^{\frac{2\pi {\rm i}}{r}}},
 \label{CGP-zero-limit-even}
\end{equation}
and
\begin{equation}
 \lim_{\mu\rightarrow 1}[r\mu]^2 \N_{r}\big(S^3_0(K),\omega\big)=
 -\frac{D^2}{2}\sum_{m\in\Z}
 F_{mr}\Big|_{q\rightarrow {\rm e}^{\frac{2\pi {\rm i}}{r}}}.
 \label{CGP-zero-limit-one-even}
\end{equation}
On the other hand, for $r=2\bmod 4$ one can consider WRT invariant refined by an element $\gamma\in H^1\big(S^3_0(K);\Z_2\big)\cong \Z_2$~\cite{kirbymelvin}. Let $c=\gamma(\mathfrak{m})\in \Z_2$. We have
\begin{align}
 &\mathrm{WRT}_r(S^3_0(K),\gamma)\nonumber
 \\
 &\qquad{}=D^{-2}\sum_{\substack{n=0\\n=c+1\bmod 2}}^{r-1}(-1)^{n+1}J_{n-1}(q)[n]
=D^{-2}\sum_{n\in 2\Z_{r/2}+c+1}F\big(q^{n/2}\big)\Big|_{q\rightarrow {\rm e}^\frac{2\pi {\rm i}}{r}}\nonumber
 \\
 &\qquad{}=-\frac{1}{2r}\sum_{\substack{ a\in \Z_{r/2} \\ b''\in \Z}}{\rm e}^\frac{2\pi {\rm i}(2a+c+1)b''}{r} F_{2b''}\Big|_{q\rightarrow {\rm e}^\frac{2\pi {\rm i}}{r}}
 =-\frac{1}{4}\sum_{m\in\Z}(-1)^{m(c+1)} F_{mr}\Big|_{q\rightarrow {\rm e}^\frac{2\pi {\rm i}}{r}}.
 \label{WRT-zero-even}
\end{align}
Combining~\eqref{CGP-zero-limit-even}, \eqref{CGP-zero-limit-one-even} and~\eqref{WRT-zero-even}, we get
\begin{equation*}
 \lim_{\mu\rightarrow c}[r\mu]^2 \N_{r}\big(S^3_0(K),\omega\big)=
 2D^2\operatorname{WRT}_r\big(S^3_0(K),\gamma\big),\qquad c=0,1,
\end{equation*}
which is again consistent with Theorem~\ref{thm-CGP-knot-surgery-limit}.

\appendix

\section[Spin and spin$^{c}$ sign-refined torsion]{Spin and spin$^{\boldsymbol{c}}$ sign-refined torsion}
\label{app:torsion}

The Reidemeister torsion has an intrinsic sign ambiguity. As was shown by Turaev, it is possible to fix it by choosing an Euler structure~\cite{turaev1990euler}. In the case of 3-manifolds, such choice is equivalent to a choice of a spin$^c$ structure~\cite{turaev1997torsion}. Consider a 3-manifold $M=S^3(\mathcal{L})$ obtained by a surgery on a framed link $\mathcal{L}\in S^3$. As before, let $\vert$ be the set of components of the link and $B_{IJ}$, $I,J\in \vert$ its linking matrix. Denote by $\sigma_K(M)\in \mathrm{Spin}^c$ a spin$^c$ structure that corresponds to a characteristic vector $K\in \Z^\vert/2B\Z^\vert$, $K_I=B_{II}\bmod 2$, as in~\eqref{plumbed-spinc}. Let $a\in H^1(M;\C/\Z)$ and $\alpha_I:=a(\mathfrak{m}_I)\in \C/\Z$. They satisfy the condition $\sum_{J}B_{IJ}\alpha_J=0\bmod 1$. As before, denote $\varepsilon=(1,1,\ldots,1)\in \Z^\vert$. The Turaev's sign refined torsion of $M=S^3(\mathcal{L})$ is then given by the following formula~\cite{turaev2002torsions} (see also~\cite{blanchet2016non}):
\begin{equation*}
 \CT(M,a,\sigma_K)= (-1)^{b_+}
 \prod_{I}\frac{1}{{\rm e}^{\pi {\rm i} \alpha_I}-{\rm e}^{-\pi {\rm i} \alpha_I}}
 \nabla_{\mathcal{L}} \big(\big\{{\rm e}^{\pi {\rm i} \alpha_I}\big\}_I\big)
 {\rm e}^{\pi {\rm i}(-\alpha^TK+\varepsilon^T B\alpha )},
\end{equation*}
where $\nabla_{\mathcal{L}}$ is the Alexander--Conway function of the link $\mathcal{L}$.

The above torsion relates to the invariant $\N_2(M,\omega)$, where $\omega\in H^1(M;\C/2\Z)\setminus H^1(M;\Z/2\Z)$.
Indeed, in~\cite{costantino2014quantum} (Theorem 6.23, taking into account the different normalisation used in the present paper) the following was proved:
\begin{equation*}
\CT\bigg(M,\frac{\omega}{2},\sigma_K\bigg)=(-2)^{1+b_1(M)}\bigg(\frac{\rm i}{4}\bigg)^{b_1(M)}\frac{{\rm i} \N_2(M,\omega)}{2} {\rm i}^{-\frac{\mu^TB\mu}{2}-K^T\mu},
\end{equation*}
where we used that $\frac{\omega}{2}\in H^1(M;\C/\Z)$ is the well defined cohomology class the value of which on the meridian $\mathfrak{m}_I$ is $\frac{\omega(\mathfrak{m}_I)}{2}=\frac{\mu_I}{2}$ and we encoded the spin$^c$ structure $\sigma_K$ via $K$ as above.

In particular, using the canonical map $i\colon\Spin(M)\to \mathrm{Spin}^c(M)$ we can define an invariant which depends only on a spin structure $s$ as
\begin{equation*}
\CT_{s}\bigg(M,\frac{\omega}{2}\bigg):=\CT\bigg(M,\frac{\omega}{2},i(s)\bigg).
\end{equation*}
More explicitly, if we now let $\tilde{s}\in \Z^{\vert}$ be such that $\tilde{s}= s \bmod 2$ and $K=\sum_{I,J\in \vert}B_{IJ}\tilde{s}_J$ then it holds:
\begin{equation*}
\CT_{s}\bigg(M,\frac{\omega}{2}\bigg):=(-2)^{1+b_1}\bigg(\frac{\rm i}{4}\bigg)^{b_1(M)}\frac{ {\rm i} \N_2(M,\omega)}{2} {\rm i}^{-\frac{\mu^TB\mu}{2}-\tilde{s}B\mu}.
\end{equation*}

It is possible to define the version of the torsion depending only on spin structure (cf.~\cite{Mikhaylov:2015nsa}). Using the map~\eqref{spin-spinc-map}, consider
\begin{equation*}
 \CT_{s}(M,a):= \CT(M,a,\sigma(b,s))
 {\rm e}^{2\pi {\rm i} a(b)}.
\end{equation*}
From the surgery formula, it is easy to see that the right-hand side depends only on spin structure $s\in \Spin(M)$, but not $b\in H_1(M;\Z)$. Namely, using the correspondence~\eqref{plumbed-spin}, we have $K=2b+Bs$ for $\sigma_K=\sigma(b,s)$, where $b\in \Z^\vert/B\Z^\vert$, $s\in \Z^\vert/2\Z^\vert$, $\sum_JB_{IJ}s_J=B_{II}\bmod 2$. Therefore, for $M=S^3(\mathcal{L})$ we have
\begin{equation*}
 \CT_s(M,a)= (-1)^{b_+}
 \prod_{I}\frac{1}{{\rm e}^{\pi {\rm i} \alpha_I}-{\rm e}^{-\pi {\rm i} \alpha_I}}
 \nabla_{\mathcal{L}}
 \big(\big\{{\rm e}^{\pi {\rm i} \alpha_I}\big\}_I\big)
 {\rm e}^{\pi {\rm i}(\varepsilon-s)^TB\alpha }.
\end{equation*}
In particular, for a plumbed $M$ we have
\begin{equation}
 \CT_s(M,a)=
 (-1)^{b_+}
 \prod_{I}\big({\rm e}^{\pi {\rm i} \alpha_I}-{\rm e}^{-\pi {\rm i} \alpha_I}\big)^{\deg(I)-2}
 {\rm e}^{\pi {\rm i}(\varepsilon-s)^TB\alpha }.
 \label{torsion-plumbed-spin}
\end{equation}
Note that if $a=\omega \bmod H^1(M;\Z/2\Z)$, where $\omega\in H^1(M;\C/2\Z)$ the dependence on the spin structure disappears, as $B\alpha=B\omega \in 2\Z^\vert$.

\section[$H^1(M;\Z/2\Z)$-refined WRT invariant and $\hat{Z}$]{$\boldsymbol{H^1(M;\Z/2\Z)}$-refined WRT invariant and $\boldsymbol{\hat{Z}}$}
\label{app:wrt-refined}

In~\cite{costantino2015relations}, building on~\cite{blanchet1992invariants,kirbymelvin} the authors define a refined Witten--Reshetikhin--Turaev invariant $\WRT_r(M,\omega)$ for $r\neq 0 \bmod 4$ that depends on a choice of $\omega\in H^1(M;\Z/2\Z)$, see Definition~\ref{def-wrt-refined}. In this appendix we relate this invariant to $\hat{Z}$.
Let $M=S^3(L)$ be the 3-manifold obtained by a~surgery on a~framed link~$L$. We will use the same conventions as before. Let $J_{n-\varepsilon}(L)\in \Z\big[q^{\pm 1/4}\big]$ be the Jones polynomial of $L$ colored by $\mathfrak{sl}_2$ representations of dimensions $n\in \{1,2,3,\ldots\}^\vert$.
Then
\begin{align*}
 &\WRT_r\big(S^3(\CL),\omega\big)
 \\
 &\qquad {}=
 \frac{D^{-b_0-b_1}}{(\Delta_+^{{\rm SO}(3)})^{b_+}(\Delta_-^{{\rm SO}(3)})^{b_-}}
\sum_{\substack{n\in \{1,2,\ldots,r-1\}^\vert\\n=\mu+\varepsilon \bmod 2}} J_{n-1}[\CL]\prod_{I\in \vert} (-1)^{n_I+1}\frac{q^{n_I/2}-q^{-n_I/2}}{q^{1/2}-q^{-1/2}}
 \bigg|_{q={\rm e}^{\frac{2\pi {\rm i}}{r}}},
\end{align*}
where as before $\mu_I:=\omega_I(\mathfrak{m}_I)$. For $r=2\bmod 4$, this invariant is a slight modification of the invariant of~\cite{kirby1990p} (see also~\cite{blanchet1992invariants}). For odd $r$, the invariant satisfies
\begin{equation*}
 \WRT_r(M,0)=\WRT_r^{{\rm SO}(3)}(M),
\end{equation*}
where $\WRT_r^{{\rm SO}(3)}(M)$ is the $\operatorname{SO}(3)$ version of the WRT invariant introduced in~\cite{kirby1990p}.

To conjecture a relationship between $\WRT_r(M,\omega)$ and $\hat{Z}_\mathfrak{s}(M)$ consider again the case of plumbing surgery. In this case,
\begin{align*}
 J[\CL]_{n-\varepsilon}={}&\frac{q^{\frac{\sum_{I}B_{II}(n_I^2-1)}{4}}}{q^{1/2}-q^{-1/2}}
 \prod_{I \in \vert}(-1)^{(n_I+1)(B_{II}+1)}\big({q^{n_I/2}-q^{-n_I/2}}\big)^{1-\text{deg}(I)}
 \\
&\times\prod_{(I,J) \in \text{Edges}}\big(q^{n_{I}n_{J}/2}-q^{-n_{I}n_{J}/2}\big).
\end{align*}

We proceed similarly to Section~\ref{sec:plumbed}
\begin{gather*}
\WRT_r(M,\omega) =\tilde{\mathcal{A}}\cdot \tilde{\mathcal{C}},
\\
 \tilde{\mathcal{A}}=\frac{(-1)^{b_-} r^{-V/2} \xi^{\frac{3\sigma-\Tr B}{2}}}{{\rm i}\sqrt{8r}}\cdot
 \begin{cases}
 {\rm e}^{\frac{\pi {\rm i}\sigma}{2}} &\text{if}\quad r=1\bmod 4, \\
 2^{-V/2} {\rm e}^{\frac{3\pi {\rm i}\sigma}{4}} &\text{if}\quad r=2\bmod 4, \\
 {\rm e}^{\pi {\rm i}\sigma} &\text{if}\quad r=3\bmod 4,
 \end{cases}
\\
 \tilde{\mathcal{C}}=\sum_\ell \tilde{\mathcal{C}}_\ell F_\ell,
\qquad
\tilde{\mathcal{C}}_\ell=\sum_{\substack{\tilde{n}\in (\Z/2r\Z)^\vert\\ \tilde{n}=\mu+(r-1)\varepsilon \bmod 2}}
 \xi^{\ell^T\tilde{n}+\frac{1}{2} \tilde{n}^TB\tilde{n}}.
\end{gather*}
We have used the following fact:
\begin{equation*}
(-1)^{\sum_{I}(n_I+1)B_{II}}\big|_{n=\mu+\varepsilon\bmod 2}=
(-1)^{\sum_{I}\mu_IB_{II}}=(-1)^{\mu^TB\mu}=1
\end{equation*}
since $B\mu=0\bmod 2$,
where $F_\ell$ are the same as in Section~\ref{sec:plumbed}.
Applying Gauss reciprocity we have
\begin{gather*}
 \tilde{\mathcal{C}}_\ell=
 \xi^{-\frac{\ell^TB^{-1}\ell}{2}}
 \frac{{\rm e}^{\frac{\pi {\rm i}\sigma}{4}}(r/2)^{V/2}}{|\det B|^{1/2}}
 \underbrace{
 \sum_{\tilde{a}\in \Z^\vert /2B\Z^{\vert}}
 {\rm e}^{-\frac{\pi {\rm i} r}{2} \tilde{a}^T B^{-1}\tilde{a}-\pi {\rm i} \tilde{a}^T B^{-1}(\ell+B(\mu+\varepsilon))}
 }_{=:\tilde{\mathcal{C}}_\ell'},
\\
 \tilde{\mathcal{C}}_\ell'=
 \sum_{{a}\in \Z^\vert /B\Z^{\vert}}\sum_{A\in \Z^\vert/2\Z^\vert}\!\!\!
 {\rm e}^{-\frac{\pi {\rm i} r}{2} {a}^T B^{-1}{a}-\pi {\rm i}r A^Ta-\frac{\pi {\rm i}r}{2} A^TBA
 -2\pi {\rm i} a^TB^{-1}b-\pi {\rm i} a^T(s+\mu)
 -\pi {\rm i} A^TBs}.
 \label{refined-wrt-C-prime}
\end{gather*}

\subsection[Level $r=2\bmod 4$]{Level $\boldsymbol{r=2\bmod 4}$}

The sum in Appendix~\ref{refined-wrt-C-prime} simplifies to
\begin{equation}
 \tilde{\mathcal{C}}_\ell'=
 2^V\sum_{{a}\in \Z^\vert /B\Z^{\vert}}
 \exp\left\{-\frac{\pi {\rm i} r}{2} {a}^T B^{-1}{a}
 -2\pi {\rm i} a^TB^{-1}b-\pi {\rm i} a^T(s+\mu)\right\}.
\end{equation}
Combining everything together, we then have
\begin{align*}
 \WRT_r(M,\omega)={}&
 \frac{(-1)^{b_+}}{{\rm i}\sqrt{8r}|\det B|^{1/2}} \xi^{\frac{3\sigma-\Tr B}{2}}
 \\
 &\times \sum_{\ell \in \Z^\vert }\sum_{{a}\in \Z^\vert /B\Z^{\vert}}
 F_\ell \xi^{-\frac{\ell^TB^{-1}\ell}{2}}
 {\rm e}^{-\frac{\pi {\rm i} r}{2} {a}^T B^{-1}{a}
 -2\pi {\rm i} a^TB^{-1}b-\pi {\rm i} a^T(s+\mu)}.
\end{align*}
We can then conjecture the following general relation for a rational homology $M$ and $r=2\allowbreak \bmod 4$:
\begin{equation*}
 \WRT_r(M,\omega)=
 \frac{1}{{\rm i}\sqrt{8r|H_1(M;\Z)|}}
 \sum_{a,b\in H_1(M;\Z)}
 {\rm e}^{-\frac{\pi {\rm i}r}{2} q_s(a)-2\pi {\rm i} \lk(a,b)-\pi {\rm i} \omega(a)}
 \hat{Z}_{\sigma(b,s)}\Big|_{q\rightarrow {\rm e}^{\frac{2\pi {\rm i}}{r}}}.
\end{equation*}

\subsection[Level $r=1\bmod 4$]{Level $\boldsymbol{r=1\bmod 4}$}

Applying a version of the Gauss reciprocity formula to the sum over $A$ in Appendix~\ref{refined-wrt-C-prime}, we can rewrite it as follows:
\begin{align*}
 \tilde{\mathcal{C}}_\ell'={}
 \frac{{\rm e}^{-\frac{\pi {\rm i}\sigma}{4}}2^{V/2}}{|\det B|^{1/2}}
 \sum_{a,f\in \Z^\vert/B\Z^\vert}
 \exp&\biggl\{-\frac{\pi {\rm i} (r-1)}{2} a^TB^{-1}a
 -2\pi {\rm i} a^TB^{-1}b-\pi {\rm i} a^T\mu
\\
 &\hphantom{\biggl\{}+2\pi {\rm i} f^TB^{-1}f
 +2\pi {\rm i} f^TB^{-1}a
 +\frac{\pi {\rm i}}{2} s^TBs
 \biggr\}.
\end{align*}
Combining everything together, we have
\begin{align*}
 \WRT_r(M,\omega)={}&
 \frac{(-1)^{b_+}}{{\rm i}\sqrt{8r}|\det B|}
 {\rm e}^{\frac{\pi {\rm i}}{2}(s^TBs-\sigma )}
 \xi^{\frac{3\sigma-\Tr B}{2}}
 \sum_{\ell \in \Z^\vert }\sum_{a,f\in \Z^\vert /B\Z^{\vert}}
 F_\ell \xi^{-\frac{\ell^TB^{-1}\ell}{2}}
 \\
 &\times{\rm e}^{-\frac{\pi {\rm i} (r-1)}{2} a^TB^{-1}a
 -2\pi {\rm i} a^TB^{-1}b-\pi {\rm i} a^T\mu
 +2\pi {\rm i} f^TB^{-1}f
 +2\pi {\rm i} f^TB^{-1}a }.
\end{align*}
 Taking into account~\eqref{Rokhlin-mod4}, we can then conjecture the following general relation for a rational homology $M$ and $r=1\bmod 4$:
\begin{align*}
 &\WRT_r(M,\omega)
 \\
 &\qquad=
 \frac{{\rm e}^{-\frac{\pi {\rm i}}{2} \mu(M,s)}}{{\rm i}\sqrt{8r}|H_1(M;\Z)|}
 \sum_{a,b,f\in H_1(M;\Z)}
 {\rm e}^{2\pi {\rm i}\left( -\frac{r-1}{4}\lk(a,a)
 +\lk(a,f-b)-\frac{1}{2}\omega(a) +\lk(f,f)\right)}
 \hat{Z}_{\sigma(b,s)}\Big|_{q\rightarrow {\rm e}^{\frac{2\pi {\rm i}}{r}}}.
 \label{refined-wrt-Zhat-1mod4}
\end{align*}

\subsection[Level $r=3\bmod 4$]{Level $\boldsymbol{r=3\bmod 4}$}

This case is analogous to the case $r=1\bmod 4$ considered above. Applying a version of the Gauss reciprocity formula to the sum over $A$ in Appendix~\ref{refined-wrt-C-prime}, we can rewrite it as follows:
\begin{align*}
 \tilde{\mathcal{C}}_\ell'=
 \frac{{\rm e}^{\frac{\pi {\rm i}\sigma}{4}}2^{V/2}}{|\det B|^{1/2}}
 \sum_{a,f\in \Z^\vert/B\Z^\vert}
 \exp&\biggl\{-\frac{\pi {\rm i} (r+1)}{2} a^TB^{-1}a
 -2\pi {\rm i} a^TB^{-1}b-\pi {\rm i} a^T\mu
 \\
 &\hphantom{\biggl\{}-2\pi {\rm i} f^TB^{-1}f -2\pi {\rm i} f^TB^{-1}a
 -\frac{\pi {\rm i}}{2} s^TBs \biggr\}.
\end{align*}
Combining everything together we then have
\begin{align*}
 \WRT_r(M,\omega)={}&
 \frac{(-1)^{b_+}}{{\rm i}\sqrt{8r}|\det B|}
 {\rm e}^{\frac{\pi {\rm i}}{2}\left(\sigma-s^TBs \right)}
 \xi^{\frac{3\sigma-\Tr B}{2}}
 \sum_{\ell \in \Z^\vert }\sum_{a,f\in \Z^\vert /B\Z^{\vert}}
 F_\ell \xi^{-\frac{\ell^TB^{-1}\ell}{2}}
 \\
&\times {\rm e}^{-\frac{\pi {\rm i} (r-1)}{2} a^TB^{-1}a
 -2\pi {\rm i} a^TB^{-1}b-\pi {\rm i} a^T\mu
 +2\pi {\rm i} f^TB^{-1}f
 +2\pi {\rm i} f^TB^{-1}a }.
\end{align*}
We can then conjecture the following general relation for a rational homology $M$ and $r=3$ $\bmod~4$:
\begin{align*}
& \WRT_r(M,\omega)
\\
&\qquad{}=
 \frac{{\rm e}^{\frac{\pi {\rm i}}{2}\mu(M,s)}}{{\rm i}\sqrt{8r}|H_1(M;\Z)|}
 \sum_{a,b,f\in H_1(M;\Z)}
 {\rm e}^{2\pi {\rm i}\left( -\frac{r+1}{4}\lk(a,a)
 -\lk(a,f+b)-\frac{1}{2}\omega(a) -\lk(f,f) \right)}
 \hat{Z}_{\sigma(b,s)}\Big|_{q\rightarrow {\rm e}^{\frac{2\pi {\rm i}}{r}}}.
\end{align*}

\section{Graded bases of the CGP TQFT}

In~\cite{blanchet2016non}, a TQFT $\VV$ was built by applying the universal construction to the invariants $\N_r$.
It~turns out that the functor one gets is a symmetric monoidal one from a suitable category of decorated cobordisms into that of graded vector spaces, endowed with the symmetry which is the flip if $r$ is odd and is the supersymmetric exchange if $r$ is even.

The cobordisms considered for this construction are $3$-manifolds $M$ with boundary endowed with cohomology classes $\omega\in H^1(M,\{*\};\C/2\Z)$ where $\{*\}$ is the choice of a base point per each connected component of the boundary (besides other standard decorations). Furthermore, by definition
\begin{equation*}
\VV_k(\Sigma,\omega)= V\big((\Sigma,\omega)\sqcup \big(S^2,V_0^{\otimes 2}\otimes \sigma^{\otimes k}\big)\big),
\end{equation*}
where $\sigma$ is the $1$-dimensional module over $U^H_q\mathfrak{sl}_2$ with weight $2r'$ and $V$ is the vector space associated to a surface by applying directly the universal construction. Finally, one lets $\VV(\Sigma,\omega)=\bigoplus_k \VV_k(\Sigma,\omega)$ which is $\Z$-graded, but it turns out that it is always finite dimensional.
As proved in~\cite{blanchet2016non} Theorem 5.9, letting
\begin{equation*}
\dim_t(\VV(\Sigma))=\begin{cases}
\sum t^k \dim( \VV_k(\Sigma)) & \text{if\quad$r$ is odd,}\\
\sum (-t)^k \dim( \VV_k(\Sigma)) & \text{if\quad$r$ is even},
\end{cases}
\end{equation*}
then if $\omega\in H^1(\Sigma;\C/2\Z)\setminus H^1(\Sigma;\Z/2\Z)$ then it holds
\begin{equation*}
\dim_{t^{-2r'}}(\VV(\Sigma_g))= \frac{(r')^g}{r}\sum_{k\in H_r} \bigg(\frac{t^r-t^{-r}}{t\xi^k-t^{-1}\xi^{-k}}\bigg)^{2g-2}.
\end{equation*}
In particular, the ungraded dimension for $g\geq 2$ is $r^{3g-3}$ if $r$ is odd and $\frac{r^{3g-3}}{2^{g-1}}$ if $r$ is even.

\begin{Remark}
The value of the Verlinde formula coincides with the value of the invariant $Z_r$ which is NOT equal to $\N_r$:
\[
Z_r=\bigg(\frac{(-1)^{r-1}}{\sqrt{r'}}\bigg)^{b_0}\bigg(\frac{(-1)^{r-1}\sqrt{r'}}{r}\bigg)^{b_1}\N_r.
\]
\end{Remark}

\begin{Example}
If $g=3$ the graded dimensions of $\dim_{t^{-2r'}}\VV(\Sigma)$ are
\begin{gather*}
r=2 \colon \ \dim_{t^{-2r'}}\VV(\Sigma) =t^4-4t^2+6-4t^{-2}+t^{-4},
\\
r=3\colon \ \dim_{t^{-2r'}}\VV(\Sigma) =108t^6+ 513+ 108t^{-6},
\\
r=4\colon \ \dim_{t^{-2r'}}\VV(\Sigma) =-8t^{12}+ 80 t^8-248t^4+352-248t^{-4}+80t^{8}-8t^{-12}.
\end{gather*}
\end{Example}

\begin{Example}
The above Verlinde formula applies only when $(\Sigma,\omega)$ is admissible.
For instance, when $\Sigma=S^2$ this is not the case and indeed the invariant is not a Laurent polynomial:
\begin{equation*}
\N_r\big(S^2\times S^1,\beta\big)=\sum_{k\in H_r} \frac{\big(q^{\beta+k}-q^{-\beta-k}\big)^2}{\big(q^{r\beta}-q^{-r\beta}\big)^2}= \frac{2r}{\big(q^{r\beta}-q^{-r\beta}\big)^2},
\end{equation*}
where $\beta\in H^1\big(S^2\times S^1;\C/2\Z\big)$ is a cohomology class the value of which on $\{pt\}\times S^1$ is $\beta\in \C/2\Z$.
\end{Example}

\begin{Example}
A special case is when $r=2$.
In this case, if $(\Sigma, \omega)$ is admissible one gets
\begin{equation*}
\dim_{t^{-2}}\VV(\Sigma)=\big(t-t^{-1}\big)^{2g-2}{\rm i}^{2-2g}.
\end{equation*}
Then applying \cite[Proposition~6.22]{blanchet2016non}, one recovers the Conway polynomial of the link in Figure~\ref{fig:verlinde}:
\begin{equation*}
\nabla_L=\big(q^{\alpha_0}-q^{-\alpha_0}\big)^{2g-1}\prod_{i=1}^{2g}\big(q^{\alpha_i}-q^{-\alpha_i}\big),
\end{equation*}
where $\alpha_0$ is the color of the main strand and $\alpha_i$ are the colors of the remaining $2g$ ones.
\begin{figure}
 \centering
 \includegraphics[width=7cm]{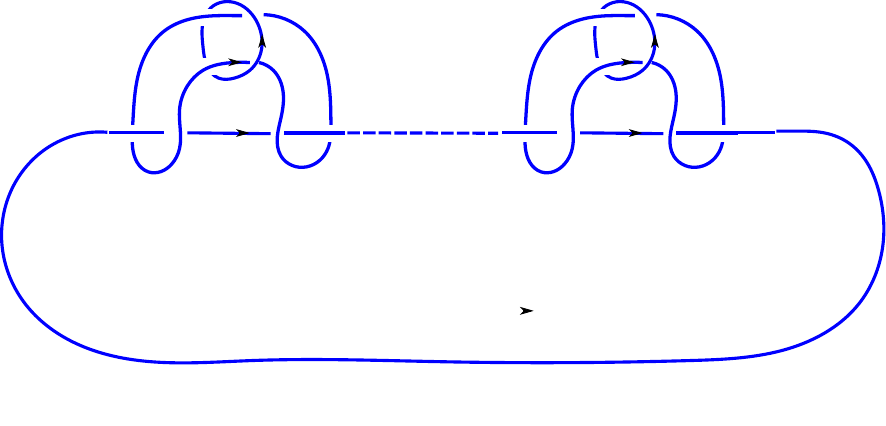}
 \caption{The link surgering on which provides $\Sigma_g\times S^1$.}
 \label{fig:verlinde}
\end{figure}
\end{Example}
The definition of the $\Z$-grading of $\VV(\Sigma,\omega)$ can be given in a more intrinsic way as follows.
Let~$Y$ be a three-manifold obtained by taking the complement of an open ball in a three-dimensional handlebody, so that $\partial Y=\Sigma\sqcup S^2$.
As detailed in~\cite{blanchet2016non}, $\VV(\Sigma, \omega)$ is $\Z$-graded with the grading being induced by the action of $H^0(\Sigma)$ as follows: if $\phi\in H^0(\Sigma;\C/2\Z)$ we can map it to a cohomology class in $H^0(\{*\};\C/2\Z)$ (where $\{*\}$ is set formed by base point on $\Sigma$ and one on $S^2$) by extending it to $0$ on $S^2$; then let $\delta(\phi)\in H^1(Y,\{*\};\C/2\Z)$ be the cohomology class induced by the exact sequence of the pair $(Y,\{*\})$. Observe that its restriction to $(\partial Y\setminus \{*\})$ is the zero cohomology class so that if $W\in H^1(Y,\{*\};\C/2\Z)$ is a class the restriction of which to $H^1(\Sigma,\{*\};\C/2\Z)$ is $\omega$ then also $W+\delta(\phi)$ is. We say that a vector $[Y,W]\in \VV(\Sigma,\omega)$ is of degree~$ k$ if for each $\phi\in H^0(\Sigma;\C/2\Z)$ we have $[Y,\omega+\delta\phi]=[Y,\omega]q^{2r'k\phi}$. It turns out that only some integer values of $k$ are possible.

For generic $\omega$, a basis of $\VV_0(\Sigma_g,\omega)$ is obtained as follows. Let $\Gamma$ be an oriented trivalent graph the thickening of which is a handlebody $H_g$ of genus $g$ the boundary of which is identified with~$\Sigma$. An edge $e$ of $\Gamma$ is colored by $\overline{\alpha}(e):=\omega(m_e)\in \C/2\Z$ where $m_e$ is the oriented meridian of the edge. Then consider a lift $\alpha\colon \operatorname{Edges}(\Gamma)\to \C$ of $\overline{\alpha}$ such that $(\partial \alpha)(v)\in H_r$ for every vertex $v\in \Gamma$. Furthermore, restrict to those $\alpha$ such that the real part of $\alpha(e)$ is between $[0,2r[$ for a~fixed arbitrary edge $e$.
Such a set is a basis of $\VV_0(\Sigma_g,\omega)$.

\section{Commutativity of limits}\label{app:commutlimits}

In this section, we show that the assumption (i) of the Theorem~\ref{thm:CGP-Zhat-torsion} is satisfied for a certain subclass of plumbing links. We will need the following proposition, which is slight generalization of a corollary to a proposition in~\cite{lawrence1999modular}.\footnote{The proposition in~\cite{lawrence1999modular} considers a slightly less general regularization.}

\begin{Proposition}
\label{prop:limits-lemma}
Let $C\colon\Z\rightarrow \C$ be a function with a period $M\in \Z_+$ and mean value 0. Then
\begin{enumerate}\itemsep=0pt
\item[$(i)$]$\displaystyle\lim_{\epsilon\rightarrow 0+}
 \sum_{n\geq 1}C(n){\rm e}^{-\epsilon(n+\gamma)}=
 -\sum_{n=1}^{M}\frac{n}{M} C(n)$,

\item[$(ii)$]$\displaystyle\lim_{\epsilon\rightarrow 0+}
 \sum_{n\geq 1}C(n){\rm e}^{-\epsilon(n^2+2\alpha n+\beta)}=
 -\sum_{n=1}^{M}\frac{n}{M} C(n)$
for any $\alpha$, $\beta$ and~$\gamma$.
\end{enumerate}
\end{Proposition}

\begin{proof}
Can be given for example using Euler--Maclaren asymptotic summation formula:
\begin{equation*}
	\sum_{n\geq 0}f(n) =\int_{0}^\infty f(x)\,{\rm d}x -\sum_{r\geq 1}\frac{B_r^-}{r!} f^{(r-1)}(0).
\end{equation*}
Only the integral part and the term $r=1$ in the sum will contribute to constant and possibly singular terms of the expansion in $\epsilon$. The singular terms and the constant terms depending on~$\alpha$,~$\beta$ and~$\gamma$ will cancel out due to the zero mean value condition.
\end{proof}

Consider a plumbing tree with vertex $I=0$ of valency 3, three vertices of valency one ($I=1,2,3$) and possibly other vertices of valency two. We then have
\begin{equation}
 F_L\big(x^2;q\big)\prod_I\big(x_I-x_I^{-1}\big)=\big(x_1-x_1^{-1}\big)\big(x_2-x_2^{-1}\big)\big(x_3-x_3^{-1}\big)
 \sum_{n\geq 1}\big(x_0^{2n-1}-x_0^{-2n+1}\big).
 \label{Y-shaped-F-expansion}
\end{equation}
In the case $\det B\neq 0$ (i.e., $b_1=0$), after $t$-regularization and the Laplace transform the sum over $n$ above will take the following form (up to a finite number of terms, which do not affect the issue of commutativity of the limits):
\begin{equation}
 \Lo_\omega F_L\big(x^2;q\big)\prod_I\big(x_I-x_I^{-1}\big)=
 \sum_{n\geq 1}\sum_{\alpha,\beta,\gamma} C_{\alpha,\beta,\gamma}(n)q^{\frac{-B^{-1}_{00}}{4}(n^2+2\alpha n+\beta)} t^{2n+\gamma},
 \label{Y-shaped-L-transform}
\end{equation}
where the sum over $\alpha$, $\beta$ and~$\gamma$ is over a finite set of rational numbers and $C_{\alpha,\beta,\gamma}(n)$ are periodic in~$n$. For the sum to give a well-defined element in $\C\big(\big(q^{1/p}\big)\big)$, we require $B_{00}^{-1}<0$. The zero mean value condition is satisfied because of the alternating signs in the expansion of~\eqref{Y-shaped-F-expansion}. We~then have
\begin{equation*}
 \lim_{q\rightarrow {\rm e}^{\frac{2\pi {\rm i}}{r}}}
 \lim_{t\rightarrow 1}\Lo_\omega F_L\big(x^2;q\big)\prod_I\big(x_I-x_I^{-1}\big)=
 \lim_{\epsilon \rightarrow 0+}
 \sum_{n\geq 1}\sum_{\alpha,\beta,\gamma} \tilde{C}_{\alpha,\beta,\gamma}(n){\rm e}^{-\epsilon(n^2+2\alpha n+\beta)}
\end{equation*}
and
\begin{equation*}
\lim_{t\rightarrow 1}
 \lim_{q\rightarrow {\rm e}^{\frac{2\pi {\rm i}}{r}}}
 \Lo_\omega F_L\big(x^2;q\big)\prod_I\big(x_I-x_I^{-1}\big)=
 \lim_{\epsilon \rightarrow 0+}
 \sum_{n\geq 1}\sum_{\alpha,\beta,\gamma} \tilde{C}_{\alpha,\beta,\gamma}(n){\rm e}^{-\epsilon(n+\gamma)},
\end{equation*}
where
\begin{equation*}
 \tilde{C}_{\alpha,\beta,\gamma}(n):=
 C_{\alpha,\beta,\gamma}(n){\rm e}^{\frac{-2\pi {\rm i}B^{-1}_{00}}{4r}(n^2+2\alpha n+\beta)}
\end{equation*}
are also periodic (generally with a larger period). Using the Proposition~\ref{prop:limits-lemma}, we can then check commutativity of the limits.

The analysis can be extended to the case $b_1>0$. This in particular covers the case of 0-surgeries on torus knots, which can be related to the plumbings of this type by Kirby moves. The main modification is that in~\eqref{Y-shaped-L-transform} one has to replace $B_{00}^{-1}$ with $\sum_{i,j=1}^{V-b_1}U_{i0}U_{j0}(B')^{-1}_{ij}$ (see Section~\ref{sec:Zhat-def-b1} for the notation), which is again required to be negative in order for $\CL_\omega$ operation to be well-defined.

\subsection*{Acknowledgements} We would like to thank Francesco Benini, Christian Copetti, Boris Feigin, Azat Gainutdinov, Hiraku Nakajima, Sunghyuk Park, Du Pei, and Nicolai Reshetikhin for helpful discussions and the anonymous Referees for the valuable suggestions on the improvement of the paper. We~also would like to thank the organizers of the 2019 conference ``New Developments in Quantum Topology'' at UC Berkeley, where the discussion on the relation between the two invariants was initiated. The work of S.G.~is supported by the U.S.~Department of Energy, Office of Science, Office of High Energy Physics, under Award no.~DE-SC0011632, and by the National Science Foundation under Grant no.~NSF DMS 1664227. The work of F.C.~was supported by the French Agence Nationale de la Recherche via the ANR Project QUANTACT and by the Labex CIMI ANR-11-LABX-0040.

\LastPageEnding

\end{document}